\documentclass[a4paper,reqno,11pt]{amsart}

\usepackage{amssymb,amsmath,amsfonts,amsthm,amscd,amstext,mathtools}
\usepackage[english]{babel}
\usepackage[latin9]{inputenc}
\usepackage{graphicx}
\usepackage[msc-links]{amsrefs}
\usepackage{perpage} 
\MakePerPage[1]{footnote}

\usepackage{url}
\usepackage{color}
\usepackage[a4paper,scale={0.72,0.74},marginratio={1:1},footskip=7mm,headsep=10mm]{geometry}
\usepackage{enumerate}

\numberwithin{equation}{section}

\newtheorem{theorem}{Theorem}[section]
\newtheorem{atheo}{Theorem}

\newtheorem{lemma}[theorem]{Lemma}
\newtheorem{proposition}[theorem]{Proposition}
\newtheorem{corollary}[theorem]{Corollary}

\newtheorem{definition}[theorem]{Definition}

\newtheorem{remark}[theorem]{Remark}

\DeclareMathOperator{\Hess}{Hess}

\long\def\xcom#1{}

\newcommand{\cA}{{\ensuremath{\mathcal A}} }
\newcommand{\cB}{{\ensuremath{\mathcal B}} }
\newcommand{\cC}{{\ensuremath{\mathcal C}} }
\newcommand{\cD}{{\ensuremath{\mathcal D}} }
\newcommand{\cE}{{\ensuremath{\mathcal E}} }

\newcommand{\cH}{{\ensuremath{\mathcal H}} }

\newcommand{\cK}{{\ensuremath{\mathcal K}} }
\newcommand{\cL}{{\ensuremath{\mathcal L}} }

\newcommand{\cO}{{\ensuremath{\mathcal O}} }

\newcommand{\cS}{{\ensuremath{\mathcal S}} }
\newcommand{\cT}{{\ensuremath{\mathcal T}} }
\newcommand{\cU}{{\ensuremath{\mathcal U}} }
\newcommand{\cV}{{\ensuremath{\mathcal V}} }
\newcommand{\cW}{{\ensuremath{\mathcal W}} }

\newcommand{\gep}{\varepsilon}       

\renewcommand{\tilde}{\widetilde}          
\DeclareMathSymbol{\leqslant}{\mathalpha}{AMSa}{"36} 
\DeclareMathSymbol{\geqslant}{\mathalpha}{AMSa}{"3E} 
\DeclareMathSymbol{\eset}{\mathalpha}{AMSb}{"3F}     

\newcommand{\R}{\mathbb{R}}

\newcommand{\Z}{\mathbb{Z}}
\newcommand{\N}{\mathbb{N}}

\def\bs{\boldsymbol}

\newcommand\bP{\ensuremath{\bs{\mathrm{P}}}}
\newcommand\bE{\ensuremath{\bs{\mathrm{E}}}}

\renewcommand{\epsilon}{\varepsilon}
\renewcommand{\theta}{\vartheta}
\renewcommand{\phi}{\varphi}
\DeclareMathOperator\argmax{arg\, max}

{\end{list}
}

\newenvironment{myitemize}{%
\begin{list}{$\bullet$}%
        {%
        \setlength{\itemsep}{0.4em}%
        \setlength{\topsep}{0.5em}%
        \setlength\leftmargin{2.45em}%
        \setlength\labelwidth{2.05em}%
        \setlength{\labelsep}{0.4em}%
        }%
        }%
{\end{list}}

\renewenvironment{itemize}{
\begin{myitemize}}%
{\end{myitemize}}

 \newcommand{\be}[1]{\begin{equation}\label{#1}}
 \newcommand{\ee}{\end{equation}}

 \newcommand{\bl}[1]{\begin{lemma}\label{#1}}
 \newcommand{\el}{\end{lemma}}

 \newcommand{\br}[1]{\begin{remark}\label{#1}}
 \newcommand{\er}{\end{remark}}

 \newcommand{\bt}[1]{\begin{theorem}\label{#1}}
 \newcommand{\et}{\end{theorem}}

 \newcommand{\bd}[1]{\begin{definition}\label{#1}}
 \newcommand{\ed}{\end{definition}}

 \newcommand{\bcl}[1]{\begin{claim}\label{#1}}
 \newcommand{\ecl}{\end{claim}}

 \newcommand{\bp}[1]{\begin{proposition}\label{#1}}
 \newcommand{\ep}{\end{proposition}}

 \newcommand{\bc}[1]{\begin{corollary}\label{#1}}
 \newcommand{\ec}{\end{corollary}}

 \newcommand{\bpr}{\begin{proof}}
 \newcommand{\epr}{\end{proof}}

 \newcommand{\bi}{\begin{itemize}}
 \newcommand{\ei}{\end{itemize}}

\newcommand{\Ll}{\mathfrak{L}}
\newcommand{\Llam}{\Ll_\Lambda}
\newcommand{\Llamn}{\Ll_{\Lambda_n}}
\newcommand{\ullamn}{\frac{1}{n}\Llamn}

\newcommand{\hbeta}{\mathfrak{h}_\beta}
\newcommand{\hnbeta}{\mathfrak{h}_{N,\beta}}

\def\etp#1{{\left ( {#1} \right )}}
\def\etc#1{{\left [ {#1} \right ]}}

\def\norme#1{{\left \Vert #1 \right \Vert}}

\newcommand\floor[1]{\lfloor#1\rfloor}
\author{Philippe Carmona}

\address{Laboratoire de Math\'ematiques Jean Leray UMR 6629,
Universit\'e de Nantes, 2 Rue de la Houssini\`ere,
BP 92208, F-44322 Nantes Cedex 03, France}

\email{philippe.carmona@univ-nantes.fr}

\author{Gia Bao Nguyen}

\address{Laboratoire de Math\'ematiques Jean Leray UMR 6629,
Universit\'e de Nantes, 2 Rue de la Houssini\`ere,
BP 92208, F-44322 Nantes Cedex 03, France}

\email{gia-bao.nguyen@univ-nantes.fr}

\author{Nicolas P\'etr\'elis}

\address{Laboratoire de Math\'ematiques Jean Leray UMR 6629,
Universit\'e de Nantes, 2 Rue de la Houssini\`ere,
BP 92208, F-44322 Nantes Cedex 03, France}

\email{nicolas.petrelis@univ-nantes.fr}



\thanks{{\it Acknowledgments.} We are grateful to Remco van der Hofstad for fruitful discussions.}

\title[IPDSAW : phase transition and collapsed phase geometry]
{Interacting partially directed self avoiding walk. From phase
  transition to the geometry of the collapsed phase.}

\date{\today}

\newcommand{\valabs}[1]{\left|#1\right|}
\BibSpec{article}{%
    +{}  {\PrintAuthors}                {author}
    +{,} { \textit}                     {title}
    +{.} { }                            {part}
    +{:} { \textit}                     {subtitle}
    +{,} { \PrintContributions}         {contribution}
    +{.} { \PrintPartials}              {partial}
    +{,} { }                            {journal}
    +{}  { \textbf}                     {volume}
    +{}  { \PrintDatePV}                {date}
    +{,} { \issuetext}                  {number}
    +{,} { \eprintpages}                {pages}
    +{,} { }                            {status}
    +{,} { \PrintDOI}                   {doi}
    +{,} { available at \eprint}        {eprint}
    +{}  { \parenthesize}               {language}
    +{}  { \PrintTranslation}           {translation}
    +{;} { \PrintReprint}               {reprint}
    +{.} { }                            {note}
    +{.} {}                             {transition}
    +{}  {\SentenceSpace \PrintReviews} {review}
}

\begin{document}

\begin{abstract}
In this paper, we investigate a model for a $1+1$ dimensional
self-interacting and partially directed self-avoiding walk, usually
referred to by the acronym IPDSAW. The interaction intensity and the free energy of the system are denoted by $\beta$ and $f$, respectively. The IPDSAW is known to undergo a collapse transition at $\beta_c$.   We provide the precise asymptotic of the free energy close to criticality, that is we show that $f(\beta_c-\gep)\sim \gamma \gep^{3/2}$ where $\gamma$ is computed explicitly and interpreted in terms of an associated continuous model.
We also establish some path properties of the 
random walk inside the collapsed phase $(\beta>\beta_c)$. We prove
that the geometric conformation adopted by the polymer is made of a
succession of long vertical stretches that attract each other to form
a unique macroscopic bead and we 
establish the
convergence of the region occupied by the path properly rescaled towards a deterministic Wulff shape.
\end{abstract}
\maketitle

\tableofcontents

\section{Introduction}

\subsection{Model and physical insight}\label{s11}

A solvent is said to be "poor" for a given homopolymer if the chemical affinity between the solvent and the monomers constituting the homopolymer is low.  When dipped in such a solvent, the homopolymer folds itself up to exclude the solvent and therefore adopts a collapsed conformation, that looks like a compact ball. If the quality of the solvent improves, the chemical affinity raises until it reaches a threshold above which the polymer extends itself in such a way that a positive fraction of its monomers are in contact with the solvent.

The interacting partially directed self-avoiding walk (IPDSAW) was introduced in \cite{ZL68} as a partially directed model of an homopolymer in a poor solvent. The spatial configurations of the polymer of length $L$ ($L$ monomers) are modeled by the trajectories of a \emph{self-avoiding} random walk on $\mathbb{Z}^2$ that only takes unitary steps \emph{upwards, downwards and to the right}. Thus, the set of allowed $L$-step paths is
\begin{align}\label{defWL}
\nonumber\mathcal{W}_L=\{w=(w_i)_{i=0}^L\in(\mathbb{N}_0\times\mathbb{Z})^{L+1}:\,&w_0=0,\, w_L-w_{L-1}=\rightarrow,\\
\nonumber&w_{i+1}-w_i\in\{\uparrow,\downarrow,\rightarrow\}\;\, \forall 0\leq i<L-1,\\
\nonumber&w_i\neq w_j\;\,\forall i<j\}.
\end{align}
Note that the choice of $w$ ending with an horizontal step is made for convenience only. We consider two different a priori laws on $\mathcal{W}_L$, uniform and non-uniform.\\
\\
(1) The uniform model: all $L$-step paths have the same probability, i.e.,
\begin{equation}\label{defPP}
\mathbf{P}_L(w)=\frac{1}{|\mathcal{W}_L|},\quad w\in\mathcal{W}_L.
\end{equation}

\noindent (2) The non-uniform model: the $L$-step paths have the following law
\begin{itemize}
  \item At the origin or after an horizontal step: the walker must step north, south or east with equal probability $1/3$.
  \item After a vertical step north (respectively south): the walker must step north (respectively south) or east with probability $1/2$.
\end{itemize}

Henceforth, we will focus on the uniform model since all our results can be adapted straightforwardly to the non-uniform model 
modulo a shift in the critical point $\beta_c$ and in the value of the
constant $a_\beta$ defined before the Shape Theorem.

\begin{figure}[ht]
	\includegraphics[width=.5\textwidth]{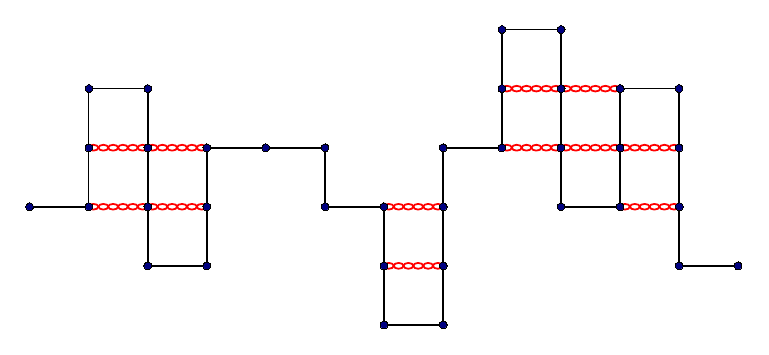}
	\caption{Example of a trajectory with 12 self-touchings in light grey.}
	\label{fig:self}
\end{figure}

The monomer-solvent interactions are not taken into account directly in the IPDSAW. We rather consider that, when dipped in a poor solvent, the monomers try to exclude the solvent and therefore attract one another. For this reason, any non-consecutive vertices of the walk though adjacent on the lattice are called \textit{self-touchings} (see Fig.~\ref{fig:self}) and the interactions between monomers are taken into account by assigning an energetic reward $\beta\geq0$ to the polymer for each self-touching (consequently, a lower chemical affinity corresponds to a larger $\beta$). Thus, we associate with every random walk trajectory $w=(w_i)_{i=0}^L\in\mathcal{W}_L$ the Hamiltonian 
\begin{equation}\label{eq:Hal}
H_{L}(w):=\sum_{\substack{i,j=0\\i<j-1}}^L\mathbf{1}_{\{\lVert w_i-w_j\rVert=1\}},
\end{equation}
which allows to define the law $P_{L,\beta}$ of the polymer in size $L$ as, 
\begin{equation}
P_{L,\beta}(w)= \frac{ e^{\beta H_{L,\beta}(w)}}{Z_{L,\beta}} 
\mathbf{P}_L(w),
\end{equation}
where $Z_{L,\beta}$ is the normalizing constant known as the partition function of the system. Henceforth,  we remove the term $1/|\cW_L|$ from the definition of  $\mathbf{P}_L$ (recall \eqref{defPP}) and from the computation of the partition function $Z_{L,\beta}$. Although $\mathbf{P}_L$ is not a probability law anymore, the latter simplification is harmless, because it does not change the polymer law $P_{L,\beta}$ and because it only induces a constant shift of the free energy $f(\beta)$ introduced in Section \ref{free} below.


\subsubsection{From random walk paths to vertical stretches}

It is easy to see that any path in $\cW_L$ can be decomposed into a collection of vertical stretches separated by one horizontal step. Thus, we set 
$\Omega_L:=\bigcup_{N=1}^L\mathcal{L}_{N,L}$, where $\mathcal{L}_{N,L}$ is the set of all possible configurations consisting of $N$ vertical stretches that have a total length $L$, that is
\begin{equation}\label{defLL}
\textstyle\mathcal{L}_{N,L}=\Bigl\{l\in\mathbb{Z}^N:\sum_{n=1}^N|l_n|+N=L\Bigr\}.
\end{equation}
We build the natural one to one correspondence between $\Omega_L$ and $\cW_L$ 
by associating with a given $l\in \Omega_L$ the path of $\cW_L$ that starts at $0$,
takes $|l_1|$ vertical steps north if $l_1>0$ and south if $l_1<0$, then takes one horizontal step, then takes $|l_2|$ vertical steps north if $l_2>0$ and south if $l_2<0$ then takes one horizontal step and so on...  (see Fig. \ref{fig:stretches}).
The Hamiltonian associated with a given path of $\cW_L$ can be rewritten in terms of its associated collection of vertical stretches $l\in \Omega_L$ as
\begin{equation}
\textstyle H_{L}(l_1,\ldots,l_N)=\sum_{n=1}^{N-1}(l_n\;\tilde{\wedge}\;l_{n+1})
\end{equation}
where
\begin{equation}
x\;\tilde{\wedge}\;y=\begin{dcases*}
	|x|\wedge|y| & if $xy<0$,\\
  0 & otherwise.
  \end{dcases*}
\end{equation}
Therefore, the partition function can be rewritten under the form

\begin{equation}\label{pff}
Z_{L,\beta}=\sum_{N=1}^{L}\sum_{l\in\mathcal{L}_{N,L}} \   e^{\beta\sum_{i=1}^{N-1}(l_i\;\tilde{\wedge}\;l_{i+1})}.
\end{equation}

\begin{figure}[ht]\center
	\includegraphics[width=.3\textwidth]{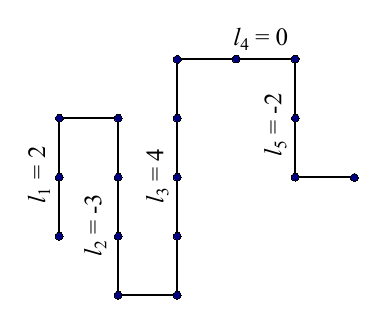}
	\caption{Example of a trajectory with $N=5$ vertical stretches
          and length $L=16$.}
	\label{fig:stretches}
\end{figure}

\subsection{Free energy and collapse transition}\label{free}
The sequence $\{\log Z_{L,\beta}\}_L$ is super-additive and the Hamiltonian in \eqref{eq:Hal} is obviously bounded from above by $\beta L$. As a consequence, we can define the free energy per step $f:(0,\infty)\to\mathbb{R}$ as 
\begin{equation}
f(\beta)=\lim_{L\to\infty}\frac{1}{L}\log Z_{L,\beta}=\sup_{L\in\mathbb{N}}\frac{1}{L}\log Z_{L,\beta}\leq\beta.
\end{equation}

The collapse transition corresponds to a loss of analyticity of $\beta \mapsto f(\beta)$ at some critical parameter $\beta_c\in (0,\infty)$ above which the density of self-touchings performed by the polymer equals $1$. In this collapsed phase, the expression of the free energy per step is rather simple, i.e., $\beta+\kappa$, where $\kappa$ is the entropic constant associated to those trajectories in $\cW_L$
whose self-touching density is equal to $1+o(1)$. To achieve such a saturation of its self-touching, the  polymer must choose its configuration among those satisfying two major 
geometric restrictions, i.e.,
\begin{itemize}\label{2cond}
\item the number of horizontal steps is $o(L)$
\item most pairs of consecutive vertical stretches are of opposite directions.
\end{itemize}
It turns out that an appropriate choice of a trajectory satisfying both restrictions above  is sufficient to exhibit the collapsed free energy.  To that aim, we pick $L\in \N\colon \sqrt{L}\in \N$ and consider the trajectory $l^{*}\in \cL_{\sqrt{L},L}$ defined as $l^{*}_i=(-1)^{i-1} (\sqrt{L}-1)$ for $i\in \{1,\dots,\sqrt{L}\}$. By computing the contribution of  $l^*$ to $Z_{L,\beta}$ 
one  immediately obtain that\footnote{In a previous paper~\cite{NGP13} the authors obtained
  the lower bound of $f(\beta)\geq \beta-\log(1+\sqrt{2})$. The
  difference comes from the omission of the normalizing factor $1/\valabs{\cW_L}$.}, 
for $\beta>0$,
\begin{equation}\label{le1*}
f(\beta)\geq \beta.
\end{equation}
At this stage, we can define the \textit{excess free energy} $\tilde{f}(\beta):=f(\beta)-\beta$, which is always non negative by \eqref{le1*}.  We define the critical parameter  
\begin{equation}
\beta_c:=\inf\{\beta\ge0:\tilde{f}(\beta)=0\},
\end{equation}
and the convexity of $\beta\mapsto \tilde{f}(\beta)$  allows us to partition $[0,\infty)$ into a collapsed phase denoted by $\mathcal{C}$ and an extended phase denoted by $\mathcal{E}$, i.e,
\begin{equation}\label{eq:colphase}
\mathcal{C}:=\{\beta:\tilde{f}(\beta)=0\}=\{\beta:\beta\geq\beta_c\}
\end{equation}
and
\begin{equation}\label{eq:extphase}
\mathcal{E}:=\{\beta:\tilde{f}(\beta)>0\}=\{\beta:\beta<\beta_c\}.
\end{equation}

\subsection{Main results}\label{maires}
The main results of this paper are Theorems A,B,C,D,E and F. Theorems A and B are dedicated to the investigation of the phase transition while the path properties of the polymer inside its collapsed phase are studied with Theorems C,D,E and F.

Before stating the Theorems we need to introduce $\mathbf{P}_{\beta}$ the law of an auxiliary symmetric random walk $V:=(V_n)_{n\in\mathbb{N}}$ with geometric increments, i.e.,
$V_0=0$, $V_n=\sum_{i=1}^n U_i$ for $n\in\mathbb{N}$ and $(U_i)_{i\in\mathbb{N}}$ is an i.i.d sequence under the law $\mathbf{P}_{\beta}$, with distribution 
\begin{equation}\label{lawP}
\mathbf{P}_{\beta}(U_1=k)=\tfrac{e^{-\frac{\beta}{2}|k|}}{c_{\beta}}\quad\forall k\in\mathbb{Z}\quad\text{with}\quad c_{\beta}:=\tfrac{1+e^{-\beta/2}}{1-e^{-\beta/2}}.
\end{equation}
Then,  for $\delta \geq 0$  we set

 \begin{equation}\label{eq:funch}
\hbeta(\delta):=\lim_{N\to\infty}\frac{1}{N}\log\mathbf{E}_\beta\bigl(e^{-\delta A_N(V)}\bigr),
\end{equation}
where $A_N(V):=\sum_{i=1}^N |V_i|$ gives the geometric area below the $V$ trajectory after $N$ steps. We will prove in Section \ref{sec:chbeta} below that the limit in \eqref{eq:funch} exists and that $\delta\mapsto \hbeta(\delta)$ is non-positive, non-increasing and continuous on $[0,\infty)$.
We finally define  $\Gamma(\beta)$ an energetic term of crucial importance as
\begin{equation}\label{sqq}	
	\Gamma(\beta)=\tfrac{c_\beta}{e^\beta},
\end{equation}
and we will see for instance in \eqref{eq:partfunc} below  that $\Gamma(\beta)$  penalizes the horizontal steps when it is smaller than $1$ and favors them when it is larger than $1$.

\subsubsection{\bf A sharper asymptotic of the free energy close to criticality}

With Theorem A,  we  give a new expression of the excess free energy. 
\begin{atheo}[Free energy equation]\label{Thm3}
The excess free energy $\tilde{f}(\beta)$ is the unique solution of the equation $\log(\Gamma(\beta))-\delta+\hbeta(\delta)=0$ if such a solution exists and $\tilde{f}(\beta)=0$ otherwise.
\end{atheo}

Note  that Theorem \ref{Thm3} and the obvious equality $\hbeta(0)=0$ are sufficient to check that the critical parameter $\beta_c$ is 
the unique solution of $\Gamma(\beta)=1$.  
One of the main interest of Theorem \ref{Thm3} is that it allows us to use the analytic properties of 
$\delta\mapsto \hbeta(\delta)$ at $0^+$ to investigate the regularity of $\beta\mapsto \tilde{f}(\beta)$ at $\beta_c$.

\begin{atheo}[Phase transition asymptotics]\label{Thm2}
The phase transition is second order with critical exponent $3/2$ and the first order assymptotic of the excess free energy at $(\beta_c)^-$ is given by
\begin{equation}
\lim_{\gep\to 0^+}\frac{\tilde{f}(\beta_c-\epsilon)}{\gep^{3/2}}=\Big(\frac{\varsigma_1}{\varsigma_2}\Big)^{3/2},
\end{equation}
where
\begin{equation}
\varsigma_1=1+\tfrac{e^{-\beta_c/2}}{1-e^{-\beta_c}},
\end{equation}
and where
\be{c1}
\varsigma_2=-\lim_{T\to \infty}\frac{1}{T}\log\mathbf{E}\bigr(e^{-\sigma_{\beta_c}\int_0^T|B(t)|dt}\bigr)=2^{-1/3}|a'_1|\sigma_{\beta_c}^{2/3},
\ee
with $\sigma_\beta^2=\mathbf{E}_\beta(U_1^2)$, with $a'_1$ the smallest zero (in absolute value) of  the first derivative of the Airy function and with $(B_s)_{s\in [0,\infty)}$ a standard Brownian motion.
\end{atheo}
\begin{remark}
The Laplace transform $\mathbf{E}(e^{-s\int_0^1 |B_s| ds})$ for $s>0$ was first computed analytically by Kac in \cite{K46} and studied by Takacs \cite{TA93} (see for instance the survey by Janson \cite{J07}).
\end{remark}

\begin{remark}
The critical exponent $3/2$ is given by the leading term of the Taylor
expansion of $\hbeta$ at $0^+$, i.e., $\hbeta(\gamma)\sim -c\,
\gamma^{2/3}$ (with $c>0$). The method of proof we used consists in
cutting the trajectories into blocs of size $\gamma^{-2/3}$. This very
method was used in \cite{HDHK03}, in dimension $d=1$, to prove that 
discrete Domb-Joyce type models converge towards  continuous Edwards
type models in the weak coupling limit.
\end{remark}
\begin{remark}The asymptotic $\hbeta(\gamma)\sim -c\,
\gamma^{2/3}$ is closely related to the
investigation of the so called pre-wetting phenomenon (see \cite{HV04},
where the scaling exponent is obtained from a renormalization procedure
similar to ours). The pre-wetting phenomenon  is observed when a thermodynamically stable 
gas is in contact with a substrate (hard-wall) that has a strong preference for the liquid 
phase. In such a situation, a thin layer of liquid may appear that separates the substrate from the gas. When the temperature $T$ gets closer to the liquid/gas boiling temperature $T_b$, the layer of liquid becomes thicker.
The liquid-gas interface can therefore be modeled by a random walk trajectory constrained to remain positive and whose area is penalized via a Gibbs factor $\delta A_N(V)$ where $\delta$ vanishes as $T\to T_b$. Close to criticality ($\delta=0$), the correlation length of the system varies as 
$\delta^{-2/3}$ which explains the $2/3$ exponent of $\hbeta$ at $0^+$. 

\end{remark}

The determination   of the precise  asymptotics of the free energy close
to $\beta_c$
brings the IPDSAW into a thin class of statistical mechanical  models for which 
the behavior of the free energy close to criticality is well understood. This is the case, for instance
for the pinning/wetting model (see \cite[Chapter 2]{GIA11}). Perturbing such models by adding a
weak random component to their interactions is physically relevant (see
\cite{DHV92})
and gives rise to complex mathematical issues (see \cite{AS06}).
For the model of a polymer pinned by a linear interface,
the issue of the disorder relevance on the phase transition   was 
controversial until it was settled recently (see  \cite{DGLT09} or
\cite[Chapters 4 and 5]{GIA11}, for a survey).
For the IPDSAW, a natural way of introducing the disorder would be to
assign an energetic price $\beta+s\xi_{i,j}$ to the self-touching between
monomers $i$ and $j$. The mechanism governing the phase transition being
quite different from its counterpart in the pinning model, the
investigation of the disorder effect is relevant both mathematically and
physically.

\subsubsection{\bf Path properties inside the collapsed phase} 
The main result of this paper is concerned with the path behavior of the polymer inside its collapsed phase $(\beta>\beta_c)$. We divide each trajectory into a succession of beads. Each bead is made of vertical stretches of strictly positive length and arranged in such a way that two consecutive stretches have opposite directions (north and south) and are separated by one horizontal step (see Fig. \ref{fig:beads}).  A bead ends when the polymer gives the same direction to two consecutive vertical stretches or when a zero length 
stretch appears, which corresponds to two consecutive horizontal
steps. We will prove that the polymer folds itself up into a {\it
  unique macroscopic bead} and we will identify its horizontal
extension and its asymptotic deterministic shape.  To quantify these results we need the following notations.

\begin{figure}[ht]\center
	\includegraphics[width=.55\textwidth]{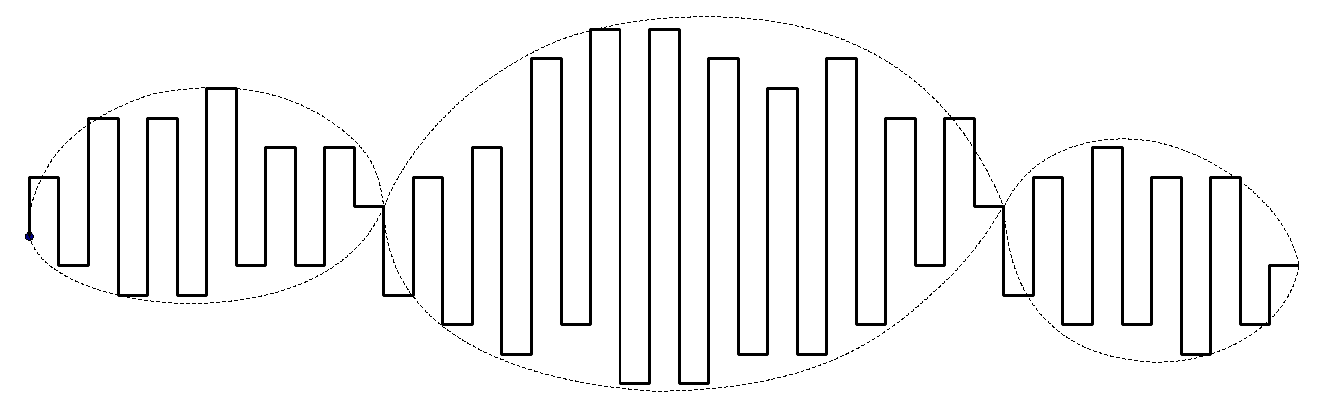}
	\caption{Example of a trajectory with $3$ beads.}
	\label{fig:beads}
\end{figure}

\subsubsection{Number of beads}
Let $l\in \Omega_L$  and denote by  $N_L(l)$ its horizontal extension, i.e., $N_L(l)$ is the integer $N$ such that $l\in \cL_{N,L}$.
Pick $l\in \cL_{N,L}$ and let $(u_j)_{j=1}^N$ be the sequence of
accumulated lengths of the polymer after each vertical stretch, adding
the lengths of the one step horizontal steps, that is   $u_j=|l_1|+\dots+ |l_j|+j$ for $j\in \{1,\dots,N\}$. For convenience only, set $l_{N+1}=0$. Set also $x_0=0$ and for $j\in \N$ such that $x_{j-1}<N$, set $x_j=\inf\{i\geq x_{j-1}+1\colon\, l_i\;\tilde{\wedge}\;l_{i+1}=0\}$ (see Fig. \ref{fig:transform}). Finally, let  $n_L(l)$ be the index of the last $x_j$ that is well defined, i.e., $x_{n_L(l)}=N$. Thus we can decompose any trajectory $l\in \Omega_L$ into a succession of $n_L(l)$ beads, each of them being associated with a subinterval of  $\{1,\dots,L\}$ written as 
\begin{equation}\label{ut}
I_j=\{u_{x_{j-1}}+1,\dots,u_{x_j}\}, \quad\text{for}\quad j\in \{1,\dots,n_L(l)\},
\end{equation}
and therefore, we can partition $\{1,\dots,L\}$ into $\cup_{j=1}^{n_L(l)} I_j$.
At this stage, we can define the largest bead of a trajectory $l\in \Omega_L$ as $I_{j_{\text{max}}}$ with
\begin{equation}\label{ut2}
j_{\text{max}}=\argmax\big\{|I_j|, j\in \{1,\dots,n_L(l)\}\big\}.
\end{equation}
With Theorem \ref{Thm4} below, we claim that, in the collapsed phase, there is only one macroscopic bead. 
\begin{atheo}[One bead Theorem] \label{Thm4}
For  $\beta>\beta_c$, there exists a $c>0$ such that 
\begin{equation}
\lim_{L\to \infty} P_{L,\beta}\big(|I_{j_{\text{max}}}|\geq L-c\, (\log L)^4\big)=1.
\end{equation}
\end{atheo}
\begin{remark}
Dividing trajectories into beads does not give rise to an underlying renewal process as for instance,
for the homogeneous pinning model when the trajectory is divided into excursions away from the origin (see for instance \cite[Chapter 2]{cf:Gia}).  The fact that, after a bead of length $1$ the first stretch of the following bead can be either positive or negative whereas its orientation is constrained when the former bead is strictly larger than $1$ creates a dependency between consecutive beads that prevents us from rewriting the partition function with the help of an associated renewal process. 
However, if we omit the dependency between consecutive beads then, thanks
to Proposition \ref{prop1}, the "bead process" 
$(u_{x_j})_{j=0}^{n_L(l)}$ under $P_{L,\beta}$ can be related to a sub-exponential defective renewal process 
$\tau=(\tau_i)_{i\geq 0}$ conditioned on  $L\in \tau$. This latter
process is characterized by an inter-arrival law $K:\overline{\N}\to [0,1]$ that
satisfies $K(\infty)>0$ and $K(n)= k(n) e^{-c\sqrt{n}}$ with $k:\N\to
\N$ a slowly varying function. Once conditioned by $\{L\in \tau\}$, it can be proven  (see \cite[Appendix A.5]{cf:Gia} for the heavy tailed case or more recently \cite{NT14} where the sub-exponential case is explicitly treated) that 
the number of renewals is $O(1)$  and that again there is only one
macroscopic renewal (see e.g. \cite{Asmussen03} for a general background
on renewal theory). 

\end{remark}

\subsubsection{Shape Theorem}
First, recall the one-to-one correspondence between $\Omega_L$ and $\cW_L$ described in Section \ref{s11} and denote by $w_l$  
the path in $\cW_L$ associated with a given family of vertical stretches $l\in   \Omega_L$.  Then,  identify each  $l\in \Omega_L$ with a connected compact subset of $\R^2$ denoted by $S_L(l)$ that extends the sites of $\Z^2$
occupied by $w_l$ to squares of length $1$, i.e.,  
\be{defS}
S_L(l)=\Big\{ \cup_{i=0}^{L} w_l(i)+[-\tfrac12,\tfrac12]^2\Big\}, \quad l\in \Omega_L.
\ee 
With Theorem \ref{Theo-shape} below we prove that, once rescaled horizontally and vertically by $\sqrt{L}$ the subset 
$S_L(l)$ converges in probability and for the Hausdorff distance towards $\cS_\beta$ a deterministic subset
of $\R^2$.  Before defining $\cS_\beta$ we need to settle some notation. 

First, we denote by $\Ll(h),h\in (-\tfrac{\beta}{2},\tfrac{\beta}{2})$ the logarithmic moment generating function of the random variable $U_1$, i.e, 
\begin{equation}\label{defL}
\mathfrak{L}(h):=\log\mathbf{E}_\beta[e^{h U_1}],
\end{equation}
and  we introduce   $\Llam$ 
\begin{equation}\label{LLam}
{\textstyle \Llam({\bf h}):=\int_0^1 \Ll(x  h_0+ h_1)dx,}
\end{equation}
which is defined on
\begin{equation}\label{eq:Llambda}
{\textstyle\mathcal{D}:=\Bigl\{{\bf h}=( h_0, h_1)\in\mathbb{R}^2\colon h_1\in\bigl(-\tfrac{\beta}{2},\tfrac{\beta}{2}\bigr),\  h_0+ h_1\in\bigl(-\tfrac{\beta}{2},\tfrac{\beta}{2}\bigr)\Bigr\}.}
\end{equation}
Then, we let
 $\tilde{{\bf h}}(q,0):=(\tilde{h}_0(q,0),\tilde{h}_1(q,0))$ be the unique solution of the equation
\begin{equation}\label{eq:deftil}
\nabla \Llam({\bf h})=(q,0)\,.
\end{equation}

Since for $\beta> \beta_c$ the function
\begin{equation}\label{defg}
\tilde{G}(a):=a\log\Gamma(\beta)-\tfrac{1}{a}\,\tilde h_0\bigl(\tfrac{1}{a^2},0\bigr)+a\Llam\bigl(\tilde {\bf h}\bigl(\tfrac{1}{a^2},0\bigr)\bigr),
\end{equation}
defined on $(0,\infty)$ is $C^\infty$, strictly concave and negative
(see Section \ref{pp5}),
we let $a_\beta >0$ be its unique maximizer.

We let  $\gamma^*_{\beta}$ be the Wulff shape minimizing the rate function of  Mogulskii large deviation principle  (see \cite[Theorem 5.1.2]{DZ} ) applied to the random walk of law $\bP_\beta$, on the set containing the cadlag functions $\gamma:[0,1]\to \R$ satisfying $\gamma(1)=0$ and 
$\int_0^1 \gamma(t) dt=1/a_\beta^2$ and endowed with the supremum norm 
$|| \cdot ||_\infty$.  The following explicit formula holds (see
Section \ref{pr:Thm7}):
\be{defgamma}
\gamma^*_{\beta}(s)=\int_0^s \Ll'\big[ (\tfrac12-x)\,  \tilde h_{0}\big(\tfrac{1}{a_\beta^2},0\big)\big] dx, \quad s\in [0,1]\,.
\ee
Eventually we define the limiting shape
\be{defSS}
\cS_\beta=\big\{(x,y)\in \R^2\colon\, x\in[0,a_\beta], y\in \big[-\tfrac{1}{2}\, a_\beta \, \gamma^*_\beta(x/a_\beta) ,\tfrac{1}{2}\, a_\beta\,  \gamma^*_\beta(x/a_\beta)\big]\big\}
\ee
and we denote by $d_H$ the Hausdorff distance between subsets of $\R^2$.

\begin{atheo}[Shape Theorem]\label{Theo-shape}
For $\beta>\beta_c$, we have convergence in $P_{L,\beta}$ probability for the
Hausdorff distance towards a deterministic shape 
\begin{equation}
\lim_{L\to \infty}
P_{L,\beta}\Big(d_H\Big(\tfrac{S_L(l)}{\sqrt{L}},\cS_\beta\Big)>\epsilon\Big)=0
\quad(\forall \epsilon >0).
\end{equation} 
\end{atheo}

This shape Theorem is equivalent to the combination of Theorem \ref{Prop5} and Theorem \ref{Convenv}
below. We display in Appendix \ref{equiv:thDEF} a proof of this equivalence.

\begin{atheo}[Horizontal extension]\label{Prop5}
For $\beta>\beta_c$, for all $\epsilon>0$  
\begin{equation}
\lim_{L\to \infty} P_{L,\beta}\biggl(\Bigl|\frac{N_L(l)}{\sqrt{L}}-a_\beta\Bigr|>\epsilon\biggr)=0.
\end{equation} 
\end{atheo}

\begin{remark}
Determining the horizontal extension is challenging in the extended regime $(\beta<\beta_c)$ 
and in the critical regime ($\beta=\beta_c)$ as well. In the extended regime, we can already derive from the variational formula of the free energy in \cite[Theorem 1.2]{NGP13} that there exists $c_2>c_1>0$ so that $\lim_{L\to \infty} P_{L,\beta}( N_L(l)/L\in [c_1,c_2])=1$. The extension is therefore of order $L$ and we expect that a law of large numbers also holds so that 
$N_L(l)/L$ converges in $P_{L,\beta}$ probability towards some constant $e_\beta\in (0,1)$. 
The critical regime is more delicate. In view of the random walk representation and since $\Gamma(\beta_c)=1$,  the law of 
$N_L(l)$  when $l$ is sampled from $P_{L,\beta}$ is exactly that of  
the stopping time $\tau_L:=\inf\{n\geq 1\colon n+A_n(V)\geq L\}$ of a random walk $V$ of law $\bP_\beta$ 
conditioned on $\{V_{\tau_L}=0, A_{\tau_L}=L-\tau_L\}$.  We expect that a Donsker type invariance principle will hold there so that 
typically $A_{\tau_L}\sim \tau_L^{3/2}$ and thus we expect  $N_L(l)/L^{2/3}$ to be  tight under $P_{L,\beta}$.
\end{remark}

The next Theorem gives the scaling limit of the upper and lower envelopes of the path in the collapsed phase.
 Pick $l \in \cL_{N,L}$ and let $\cE_l^+=(\cE^+_{l,i})_{i=0}^{N+1}$  be  the path that links the 
top  of each stretch consecutively (see Figure \ref{fig:envel}), while  $\cE_l^-=(\cE^-_{l,i})_{i=0}^{N+1}$ is the counterpart of 
$\cE_l^+$ that links the 
bottom   of each stretch consecutively. Thus, $\cE^+_{l,0}=\cE^-_{l,0}=0$,
\begin{align}\label{trek}
\cE^+_{l,i}&=\max\{l_1+\dots+l_{i-1}, l_1+\dots+l_{i}\},\quad i\in \{1,\dots,N\},\\
\cE^-_{l,i}&=\min\{l_1+\dots+l_{i-1}, l_1+\dots+l_{i}\},\quad i\in \{1,\dots,N\},
\end{align}
and $\cE^+_{l,N+1}=\cE^-_{l,N+1}=l_1+\dots+l_{N}$.
Then, let $\tilde\cE_l^+$ and $\tilde\cE_l^-$ be  the time-space rescaled  cadlag  processes associated with 
$\cE_l^+$ and $\cE_l^-$ and defined as
\begin{equation}\label{defti}
\tilde\cE_l^+(t)= \frac{1}{N+1}\,  \cE^+_{l,{\lfloor t\,(N+1)\rfloor}}\quad \text{and}\quad \tilde\cE_l^-(t)= \frac{1}{N+1}\,  \cE^-_{l,{\lfloor t\,(N+1)\rfloor}},\quad t\in [0,1].
\end{equation}

\ 

\begin{figure}[ht]
	\includegraphics[width=.5\textwidth]{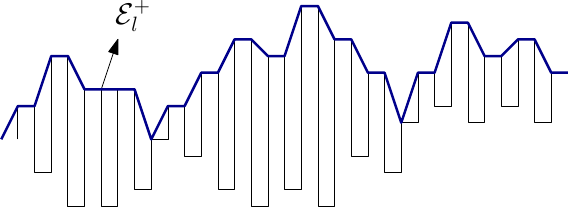}
	\caption{Example of the upper  envelope of a trajectory}
	\label{fig:envel}
\end{figure}

\begin{atheo}[Wulff shape]\label{Convenv} 
For  $\beta>\beta_c$ and  $\gep>0$,
\begin{align}\label{conven}
\nonumber \lim_{L\to \infty} P_{L,\beta}\Big(  \big\|\tilde\cE^+_{l}-\frac{\gamma^*_{\beta}}{2}\big\|_{\infty} >\gep \Big)&=0,\\
\lim_{L\to \infty} P_{L,\beta}\Big(  \big\|\tilde\cE^-_{l}+\frac{\gamma^*_{\beta}}{2}\big\|_{\infty} >\gep \Big)&=0.
\end{align}
\end{atheo}
\bigskip

Note that  $\tilde\cE^+_{l}-\tilde\cE^-_{l}$ (respectively, $(\tilde\cE^+_{l}+\tilde\cE^-_{l})/2$) is the rescaled version of the process that associates with each index $i\in \{1,\dots, N_L(l)\}$ the length $|l_i|$ of the $i$-th stretch (respectively, the height of the middle of the $i$-th stretch $l_1+\dots+l_{i-1}+\frac{l_i}{2}$).   In view of Theorem \ref{Convenv}, the Wulff shape  $\gamma^*_{\beta}$ happens to be the limit, as $L\to \infty$, of $\tilde\cE^+_{l}-\tilde\cE^-_{l}$.
Such Wulff shape  was identified, for instance in \cite{DH96}, as the limit of a random walk trajectory conditioned by fixing a large algebraic area between the path and the $x$-axis.  However, the latter convergence is not sufficient to prove \eqref{conven}. We must indeed show that 
$(\tilde\cE^+_{l}+\tilde\cE^-_{l})/2$ converges to $0$ in probability.
\begin{remark}
The Wulff shape construction, initially displayed in \cite{W1901} appears in many models of statistical mechanics to describe the limiting shape of properly rescaled interfaces separating pure phases. 
Their construction is achieved  by  minimizing  the integral of the surface tension 
along the continuous contours that satisfy some particular geometric constraint.
A famous example arises from 2D Ising model in the phase transition regime. When considering a large square box of size $N$ with $-$ boundary condition and $T<T_c$, and by conditioning the total magnetization to 
be shifted  from its mean ($-m^* N^2$) by a factor $a_N\gg N^{4/3}$, it was proven in \cite{DKS92} at low temperature and then in \cite{I94}, \cite{I95} and \cite{IS98} up to $T_c$  that  this magnetization shift is due to a unique $+$ island whose  boundary, once rescaled by $1/\sqrt{a_N}$, converges towards a Wulff shape.   
\end{remark}


\subsection{Relationship to earlier work}
The IPDSAW and its continuous versions have attracted a lot of attention from \emph{physicists} until very recently (see for instance \cite{BDL09} or \cite{HST13}). 
The main method that has been employed  to investigate the IPDSAW involves combinatorial techniques (see \cite{BGW92}, \cite{OPB93} or more recently  \cite{OP07}). 
To be more specific,  this method consists in providing an analytic expression of the generating function $G(z)=\sum_{L=1}^\infty Z_{L,\beta} z^L$ whose radius of convergence $R$ satisfies $f=-\log R$.  For a detailed version of the computations, we refer to \cite[p.~371--375]{CHP12}.

The computation of the generating function $G$ allows us to determine the exact value of $\beta_c$
and to predict the behavior of the free energy close to criticality. However,  the analytic expression of $G$ is very complicated and only gives an indirect access to the free energy. Furthermore, this combinatorial method does not allow to study an observable 
which does not grow like $L$, for instance, inside the collapsed phase, the horizontal extension is of order $\sqrt{L}$ and this can not be proven by such method.

\bigskip

\emph{A new approach} has been developed in \cite{NGP13} to work with the partition function directly. 
With the help of an algebraic manipulation of the Hamiltonian, that
will be described in Section \ref{sec:rep}, it is indeed possible to rewrite the partition function in \eqref{pff} under the form 
\begin{equation}\label{eq:partfunc}
Z_{L,\beta}=c_\beta\,e^{\beta L}\sum_{N=1}^{L}\left(\Gamma(\beta)\right)^N\mathbf{P}_{\beta}(\mathcal{V}_{N+1,L-N}),
\end{equation}
where we recall \eqref{lawP} and \eqref{sqq} and where $\mathcal{V}_{n,k}$ is the set of those $n$-step trajectories of the random walk $V$ whose geometric area $A_n=\textstyle\sum_{i=1}^n|V_i|$ equals $k$, i.e.,
\begin{equation}\label{defcV}
\mathcal{V}_{n,k}:=\big\{(V_i)_{i=0}^n:\, A_n=k,\,V_{n}=0\big\}.
\end{equation}
Thus, the excess free energy $\tilde{f}(\beta)$ is the exponential growth rate of the summation in  \eqref{eq:partfunc}.   
In this new expression of the partition function, the term indexed by $N\in \{1,\dots,L\}$ in the summation corresponds to the contribution to the partition function of those trajectories  $l\in \cL_{N,L}$ (making $N$ horizontal steps).   
%

This new approach was used in \cite[Theorem 1.2]{NGP13}, to derive a  variational expression of the excess free energy, which allowed us to prove that the collapsed transition is second order with critical exponent $3/2$.  
\begin{theorem}[\cite{NGP13}, Theorem 1.4]\label{Thm3pt}
The phase transition is of order $3/2$. That is, there exist two constants $c_1,c_2>0$ such that for $\epsilon$ small enough
\begin{equation}
c_1\,\epsilon^{3/2}\leq\tilde{f}(\beta_c-\epsilon)\leq c_2\,\epsilon^{3/2}.
\end{equation}
\end{theorem}

%

%

With the present paper, we take the analysis of the phase transition two steps further (see Theorem \ref{Thm2}). In the first step, we establish the precise asymptotic: $\tilde f(\beta_c-\gep)\sim \gamma \gep^{3/2}$ as $\gep \searrow 0$  with $\gamma$ an explicit constant. In the second step, we give an expression of $\gamma$ in terms of the free energy of an auxiliary continuous model, that is $\mathsf{F}_c=\lim_{T\to \infty}\frac{1}{T}\log\mathbf{E}[\exp(-\int_0^T |B(t)|dt)]$. Moreover, the Laplace transform of $\int_0^T |B(t)|dt$ was computed by Kac in \cite{K46} and this allows us to express $\mathsf{F}_c$ with the smallest zero (in modulus) of the derivative of the Airy function.

The question of the geometric conformation adopted by the polymer inside the collapsed phase 
has been raised and discussed by physicists in several papers, as  for instance \cite{POBG93}. It was believed that the monomers arrange themselves in a succession of long vertical stretches of opposite directions that constitute large beads. In this paper, we prove with Theorem \ref{Thm4}, that the polymer makes \emph{only one macroscopic bead} and that the number of monomers (located at the beginning and at the end of the polymer) which do not belong to this bead grows at most like $(\log L)^4$.  We also make rigorous the conjecture concerning the horizontal extension of the polymer, since we identify the limit in probability 
of $N_L/\sqrt{L}$, which turns out to be  the constant extracted from
an optimization procedure.  We also establish the convergence of
properly rescaled lower and upper envelopes to a deterministic Wulff
shape. In particular, the typical vertical displacement of the middle
point, the $L/2$-th monomer in a chain of length $L$, is of order $\sqrt{L}$.

There are numerical evidences that  the vertical displacement of the
endpoint grows as $L^{1/4}$ (see \cite{POBG93},  table II page
2394). This turns out to be a consequence of the typical behavior of the fluctuations of
the envelopes around the Wulff shape, and this is not the topic of the
present paper.

\bigskip
\emph{Finally}, let us stress the fact that the convergence, in the
collapsed phase, to a deterministic Wulff shape (see Theorem E) comes
from a fairly complex procedure that needs to establish three properties:
\begin{enumerate}[(i)]
\item The horizontal extension $N_L$ is of order $\sqrt{L}$.
\item There is only one macroscopic bead
\item When conditioned to be abnormally large, the geometric area of
  the associated $V$ random walk ($\sum_i \valabs{V_i}$) is
  close to the modulus of its algebraic counterpart ($\valabs{\sum V_i}$).
\end{enumerate}
There is no clear order in which to establish these properties and the
proofs are intricate. For example, we need weak versions of (i) and
(iii) to prove (ii) and then get a stronger version of (i).

\section{Preparation : the main tools.}
In this section, we introduce the three main tools that are used in this paper. In Section \ref{sec:rep} we show how the partition function can be rewritten in terms of the random walk $V$ of law $\mathbf{P}_{\beta}$ (recall \ref{lawP}) and how studying this random walk under an appropriate conditioning can be used to derive some path properties under the polymer measure.
In Section \ref{sec:chbeta}, we define the function $\delta \mapsto \hbeta(\delta)$ that appears in the expression of the excess free energy in Theorem \ref{Thm3} and we study its regularity. 
In Section \ref{sec:lcltp}, we consider the probability of some large deviations events under $\mathbf{P}_{\beta}$, 
and following Dobrushin and Hryniv~\cite{DH96}, we introduce an appropriate tilting under which the probability of such events 
decays only polynomially fast. 

\subsection{Probabilistic representation of the  partition function}\label{sec:rep}
In the first part of  this section we prove formula  \eqref{eq:partfunc} and we show how the polymer measure can be expressed as the image measure by an appropriate transformation of the geometric random walk $V$ introduced in \eqref{lawP}.  In the second part of the section, we focus on those trajectories that 
make only one bead and we show that, in terms of the auxiliary random walk $V$, these beads become excursions away from the origin. 

\begin{figure}[ht]\center
	\includegraphics[width=.65\textwidth]{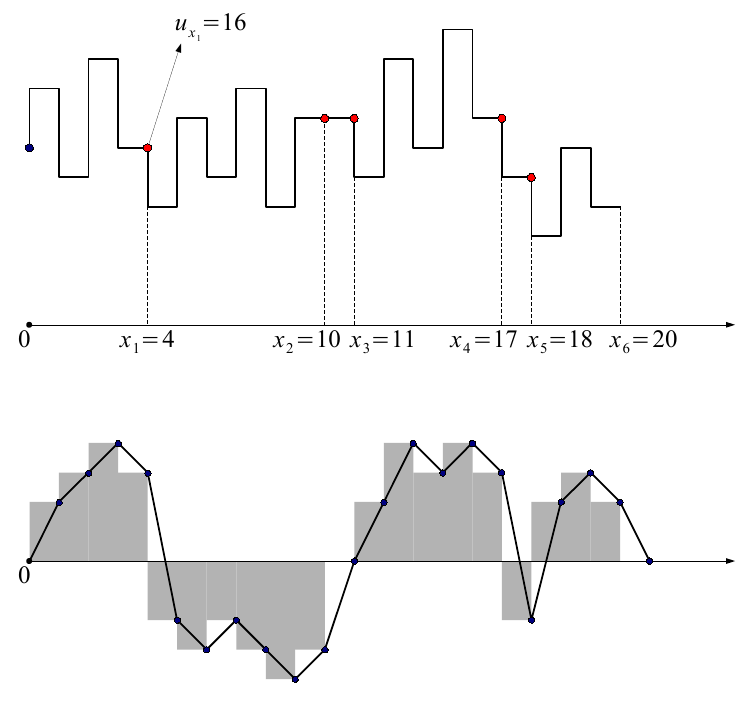}
	\caption{An example of a trajectory $l=(l_i)_{i=1}^{20}$ with 6 beads is drawn on the  upper picture.  The auxiliary random walk $V$ associated with $l$, i.e., $(V_i)_{i=0}^{21}=(T_{20})^{-1}(l)$  is drawn on the lower picture.}
	\label{fig:transform}
\end{figure}

\subsubsection{Auxiliary random walk}
We display here the details of the proof of formula  \eqref{eq:partfunc}.  Recall (\ref{defLL}--\ref{pff}) and note that the $\tilde{\wedge}$ operator can be written as
\begin{equation}
x\;\tilde{\wedge}\;y=\left(|x|+|y|-|x+y|\right)/2,\quad\forall x,y\in\mathbb{Z}.
\end{equation}
Hence, for $\beta>0$ and $L\in\mathbb{N}$, the partition function in \eqref{pff} becomes
\begin{align}\label{ls}
\nonumber Z_{L,\beta}
&=\sum_{N=1}^{L} \sum_{\substack{l\in\mathcal{L}_{N,L}\\l_0=l_{N+1}=0}}\exp{\Bigl(\beta\sum_{n=1}^N{|l_n|}-\tfrac{\beta}{2}\sum_{n=0}^N{|l_n+l_{n+1}|}\Bigr)}\\
&=c_\beta\, e^{\beta L} \sum_{N=1}^{L}\left(\tfrac{c_\beta}{e^\beta}\right)^N\sum_{\substack{l\in\mathcal{L}_{N,L}
\\l_0=l_{N+1}=0}}\prod_{n=0}^{N}\frac{\exp{\Bigl(-\tfrac{\beta}{2}|l_n+l_{n+1}|\Bigr)}}{c_\beta},
\end{align}
where $c_\beta$ was defined in \eqref{lawP}.
At this stage, we pick $N\in \{1,\dots,L\}$ and we introduce the  one-to-one correspondence
$T_N:\cV_{N+1,L-N}\mapsto \cL_{N,L}$ defined as $T_N(V)_i=(-1)^{i-1} V_i$ for all $i\in \{1,\dots N\}$. We pick $l\in \cL_{N,L}$, we consider  $V=(T_N)^{-1}(l)$ (see Fig. \ref{fig:transform}) and we note that the increments $(U_i)_{i=1}^{N+1}$ of $V$ necessarily satisfy $U_i:=(-1)^{i-1}(l_{i-1}+l_i)$. Thus, the partition function in \eqref{ls} becomes 
\begin{equation}\label{tgh}
Z_{L,\beta}=c_\beta  e^{\beta L} \sum_{N=1}^{L} \left(\tfrac{c_\beta}{e^\beta}\right)^N \sum_{V\in \cV_{N+1,L-N}} \mathbf{P}_{\beta}(V),
\end{equation}
which immediately implies \eqref{eq:partfunc}. 
A useful consequence of formula \eqref{tgh} is that, once conditioned on taking a given number of horizontal steps $N$, the polymer measure is exactly the image measure by the $T_N$-transformation of the geometric random walk $V$ conditioned to return to the origin after N+1 steps and 
to make a geometric area $L-N$, i.e.,  
\begin{equation}\label{transf2}
P_{L,\beta}\bigl(l\in\cdot \mid N_L(l)=N\bigr)=\mathbf{P}_\beta\bigl(T_N(V)\in\cdot \mid V_{N+1}=0, A_N=L-N\bigr).
\end{equation}

\subsubsection{From beads to excursions}
We  define $\Omega_L^\mathsf{o}$ as the subset of $\Omega_L$ containing those trajectories $l\in\Omega_L$ that have only one bead, i.e. $n_L(l)=1$. Thus, we can rewrite $\Omega^\mathsf{o}_L:=\bigcup_{N=1}^L\mathcal{L}^\mathsf{o}_{N,L}$, where $\mathcal{L}^\mathsf{o}_{N,L}$ is the subset of $\mathcal{L}_{N,L}$ defined as
\begin{equation}\label{deftu}
\textstyle\mathcal{L}^\mathsf{o}_{N,L}=\bigl\{l\in\mathcal{L}_{N,L}\colon\,l_i\;\tilde{\wedge}\;l_{i+1}\neq 0\ \forall i\in\{1,\dots,N-1\}\bigr\},
\end{equation}
and we denote by $Z^\mathsf{o}_{L,\beta}$ the contribution to the partition function of those trajectories in $\Omega^\mathsf{o}_L$, i.e.,
\begin{equation}\label{eq:dist1}
Z^\mathsf{o}_{L,\beta}=\sum_{l\in\Omega_L^\mathsf{o}}e^{\beta H_{L}(l)}.
\end{equation}
We let also $\mathcal{V}_{n,k}^{+}$ be the subset containing those trajectories that return to the origin after $n$ steps, satisfy $A_n=k$ and are strictly positive on $\{1,\dots,n\}$, i.e.,
\begin{equation}\label{defVV+}
\mathcal{V}_{n,k}^+:=\{V\colon\,V_n=0,\,A_n=k,\,V_i>0\ \forall i\in\{1,\dots,n-1\}\}.
\end{equation}
By mimicking \eqref{ls} and by noticing that by the $T_N$-transformation, the subset  $\mathcal{L}^\mathsf{o}_{N,L}$ becomes $\mathcal{V}^+_{N+1,L-N}$ we obtain  
\begin{equation}\label{eq:dist11}
Z^\mathsf{o}_{L,\beta}=2\, c_\beta\, e^{\beta L} \sum_{N=1}^{L}\left(\Gamma(\beta)\right)^N\mathbf{P}_{\beta}(\mathcal{V}^+_{N+1,L-N}).
\end{equation}

\subsection{Construction and regularity of \texorpdfstring{$\hbeta$}{h}}\label{sec:chbeta}
We define the function $\hbeta$ in a slightly different way from \eqref{eq:funch}, but we will see at the end of this section that the two definitions are equivalent. For $N\in\mathbb{N},\delta\geq0$, define
\begin{equation}\label{eq:funch2}
\hnbeta(\delta):=\frac{1}{N}\log\mathbf{E}_\beta\bigl(e^{-\delta A_N}\mathbf{1}_{\{V_{N}=0\}}\bigr)\quad\text{and let}\quad \hbeta(\delta)=\lim_{N\to\infty}\hnbeta(\delta).
\end{equation} 

\begin{lemma}\label{funch}
(i) $\hbeta(\delta)$ exists and is finite, non-positive for all $\beta>0,\delta\geq0$.\\
(ii) $\delta\mapsto \hbeta(\delta)$ is continuous, convex and non-increasing on $[0,\infty)$.
\end{lemma}

\begin{proof}
(i) For $N,M\in\mathbb{N}$, we restrict the partition of size $N+M$ to those trajectories that return to the origin at time $N$ and use the Markov property to obtain
\begin{equation}
\mathbf{E}_\beta\bigl(e^{-\delta A_{N+M}}\mathbf{1}_{\{V_{N+M}=0\}}\bigr)\geq\mathbf{E}_\beta\bigl(e^{-\delta A_N}\mathbf{1}_{\{V_{N}=0\}}\bigr)\mathbf{E}_\beta\bigl(e^{-\delta A_M}\mathbf{1}_{\{V_{M}=0\}}\bigr).
\end{equation}
Thus, $\{\log\mathbf{E}_\beta\bigl(e^{-\delta A_N}\mathbf{1}_{\{V_{N}=0\}}\bigr)\}_{N\in\mathbb{N}}$ is a super-additive sequence that is bounded above by $0$ and therefore 
the limit in \eqref{eq:funch2} exists, is finite and satisfies
\begin{equation}\label{eq:hbetasup}
\hbeta(\delta)=\sup_{N\in\mathbb{N}}\frac{1}{N}\log\mathbf{E}_\beta\bigl(e^{-\delta A_N}\mathbf{1}_{\{V_{N}=0\}}\bigr)\leq 0.
\end{equation}
(ii) The fact that $A_N\geq 0$ for all $N\in\mathbb{N}$ immediately entails that $\delta\mapsto  \hbeta(\delta)$ is non-increasing on $[0,\infty)$. By H\"{o}lder's inequality, the function $\delta\mapsto \hnbeta(\delta)$ is convex for all $N\in\mathbb{N}$ and hence so is $\delta\mapsto \hbeta(\delta)$. Convexity and finiteness imply continuity on $(0,\infty)$. In order to prove the continuity at $0$, we first note that $\lim_{\delta\to0}\hbeta(\delta)=\sup_{\delta\geq0}\hbeta(\delta)$. Then, with the help of formula \ref{eq:hbetasup} and via an exchange of suprema we obtain
\begin{align}\label{suprema} 
\lim_{\delta\to0}\hbeta(\delta)&=\sup_{\delta\geq0}\sup_{N\in\mathbb{N}}\frac{1}{N}\log\mathbf{E}_\beta\bigl(e^{-\delta A_N}\mathbf{1}_{\{V_{N}=0\}}\bigr)\\
\nonumber &=\sup_{N\in\mathbb{N}}\frac{1}{N}\log\mathbf{P}_\beta(V_{N}=0)=0.
\end{align}

\end{proof}
It remains to show that the two definitions of $\hbeta$ in \eqref{eq:funch} and \eqref{eq:funch2} coincide. To that aim it suffices to show that
\begin{equation}\label{eq:ppr}
\limsup_{N\to\infty}\frac{1}{N}\log\mathbf{E}_\beta\bigl(e^{-\delta A_N}\bigr)\leq\lim_{N\to\infty}\frac{1}{N}\log\mathbf{E}_\beta\bigl(e^{-\delta A_N}\mathbf{1}_{\{V_N=0\}}\bigr).
\end{equation}
We set $\mathcal{I}_{N^2}:=[-N^2,N^2]\cap\mathbb{Z}$\, and we decompose $\mathbf{E}_\beta\bigl(e^{-\delta A_N}\bigr)$ into the two partition functions $C_{N,\beta}$ and $B_{N,\beta}$ defined as
\begin{equation}\label{eq:defAB}
C_{N,\beta}=\mathbf{E}_\beta\bigl(e^{-\delta A_N}\mathbf{1}_{\{V_N\in\mathcal{I}_{N^2}\}}\bigr)\quad\text{and}\quad B_{N,\beta}=\mathbf{E}_\beta\bigl(e^{-\delta A_N}\mathbf{1}_{\{V_N\notin\mathcal{I}_{N^2}\}}\bigr).
\end{equation}
Since $A_N\geq 0$ and  since $\mathbf{E}_\beta\etc{e^{\frac{\beta |U_1|}{4}}}<\infty$,  Markov's inequality gives
\begin{equation}
B_{N,\beta}\leq\mathbf{E}_\beta\etc{\mathbf{1}_{\{V_N\notin\mathcal{I}_{N^2}\}}}\leq \mathbf{P}_\beta\Bigl(\sum_{i=1}^N|U_i|\geq N^2\Bigr)\leq\frac{\mathbf{E}_\beta\etc{e^{\frac{\beta |U_1|}{4}}}^N}{e^{(\beta/4)N^2}},
\end{equation}
which immediately implies that $\limsup_{N\to\infty}\frac{1}{N}\log B_{N,\beta}=-\infty$. Consequently
\begin{equation}\label{eq:btps}
\limsup_{N\to\infty}\frac{1}{N}\log\mathbf{E}_\beta\bigl(e^{-\delta A_N}\bigr)=\limsup_{N\to\infty}\frac{1}{N}\log C_{N,\beta},
\end{equation}
and since the cardinality of $\mathcal{I}_{N^2}$ grows polynomially, the proof of \eqref{eq:ppr} will be complete once we show that 
\begin{equation}\label{eq:hbetauppb2}
\limsup_{N\to\infty}\frac{1}{N}\log \sup_{x\in\mathcal{I}_{N^2}}\mathbf{E}_\beta\bigl(e^{-\delta A_N}\mathbf{1}_{\{V_N=x\}}\bigr)\leq\lim_{N\to\infty}\frac{1}{N}\log\mathbf{E}_\beta\bigl(e^{-\delta A_N}\mathbf{1}_{\{V_N=0\}}\bigr).
\end{equation}
For $x\in \Z$, we denote by $\mathbf{P}_{\beta,x}$ the law of  $x+V$ where $V$ is the random walk of law $\mathbf{P}_\beta$.  We consider the partition function of size $2N$ and use Markov property at time $N$ to obtain
\begin{equation}\label{eq:partE2N}
\mathbf{E}_\beta\bigl(e^{-\delta A_{2N}}\mathbf{1}_{\{V_{2N}=0\}}\bigr)\geq \mathbf{E}_\beta\bigl(e^{-\delta A_{N}}\mathbf{1}_{\{V_{N}=x\}}\bigr)\mathbf{E}_{\beta,x}\bigl(e^{-\delta A_{N}}\mathbf{1}_{\{V_{N}=0\}}\bigr),\quad x\in\mathbb{Z}.
\end{equation}
By using the time reversal property of the random walk $V$, we can assert that $(V_N-V_{N-n},\,0\leq n\leq N)\overset{d}{=}(V_n-V_0,\,0\leq n\leq N)$ and consequently, for all $x\in \Z$, it comes that
\begin{align}\label{inegi}
\nonumber\mathbf{E}_{\beta,x}\bigl(e^{-\delta\sum_{n=1}^N|V_n|}\mathbf{1}_{\{V_{N}=0\}}\bigr)
&=\mathbf{E}_\beta\bigl(e^{-\delta\sum_{n=1}^N|V_n+x|}\mathbf{1}_{\{V_N=-x\}}\bigr)\\
\nonumber&=\mathbf{E}_\beta\bigl(e^{-\delta\sum_{n=1}^N|V_N-V_{N-n}+x|}\mathbf{1}_{\{V_{N}=-x\}}\bigr)\\
&=\mathbf{E}_\beta\bigl(e^{-\delta\sum_{n=1}^{N-1}|V_n|}\mathbf{1}_{\{V_{N}=-x\}}\bigr).
\end{align}
Thanks to the symmetry of $V$ and since $\sum_{n=1}^{N-1}|V_n|\leq A_N$,
the inequalities \eqref{eq:partE2N} and \eqref{inegi} allow us to write 
\begin{equation}\label{inegii}
\mathbf{E}_\beta\bigl(e^{-\delta A_{2N}}\mathbf{1}_{\{V_{2N}=0\}}\bigr)\geq\Big[\sup_{x\in\mathcal{I}_{N^2}}\mathbf{E}_\beta\bigl(e^{-\delta A_N}\mathbf{1}_{\{V_N=x\}}\bigr)\Big]^2.
\end{equation}
It remains to apply $\frac{1}{2N}\log$ in both sides of \eqref{inegii} and to let $N\to\infty$ to obtain \eqref{eq:hbetauppb2}, which completes the proof.

\subsection{Large deviation estimates}\label{sec:lcltp}
In this section, we introduce the techniques that will be required to estimate the probability of some large deviation events associated with trajectories making a large arithmetic area. Such estimates will be needed in Section \ref{geo} to approximate the probability that, under the polymer measure, the trajectories make only one bead. 

Following  Dobrushin and Hryniv in \cite{DH96}, for $n\in\mathbb{N}$, we define
\begin{equation}\label{eq:Yn}
Y_n:=\tfrac{1}{n}(V_0+V_1+\dots+V_{n-1}),
\end{equation}
and for a given $q\in (0,\infty)\cap \tfrac{\N}{n}$, we focus on both probabilities $\mathbf{P}_\beta(Y_n=n q,\,V_n=0)$ and $\mathbf{P}_\beta(Y_n=n q,\,V_n=0,\,V_i>0\,\forall i\in\{1,\ldots,n-1\})$. Our aim is to identify the exponential rate at which such probabilities are decreasing and their asymptotic polynomial correction. To that aim,  
we will use an \textit{exponential tilting} of the probability measure $\mathbf{P}_\beta$ (through the Cramer transform) in combination with a local limit theorem. Under the tilted probability measure the event $\{Y_n=n q,\,V_n=0\}$ is not of large deviation type anymore since its probability decays at  polynomial speed instead of exponential speed, as will be seen in Section \ref{sec:lclt}.

For the ease of notations, we set $\Lambda_n:=(Y_n,V_n)$ and we denote its logarithmic moment generating function by $\Ll_{\Lambda_n}({\bf h})$ for ${\bf h}:=(h_0,h_1)\in\mathbb{R}^2$, i.e.,
\begin{equation}\label{eq:LlambdaN}
{\textstyle \Ll_{\Lambda_n}({\bf h}):=\log\mathbf{E}_\beta\bigl[e^{h_0Y_n+h_1V_n}\bigr]=\sum_{i=1}^n \Ll\Bigl(\bigl(1-\tfrac{i}{n}\bigr)h_0+h_1\Bigr).}
\end{equation}
Clearly, $\Ll_{\Lambda_n}({\bf h})$ is finite for all ${\bf h} \in \cD_n$ with
\begin{equation}
{\textstyle\mathcal{D}_n:=\Bigl\{(h_0,h_1)\in\mathbb{R}^2\colon h_1\in\bigl(-\tfrac{\beta}{2},\tfrac{\beta}{2}\bigr),\ (1-\tfrac{1}{n})h_0+h_1\in\bigl(-\tfrac{\beta}{2},\tfrac{\beta}{2}\bigr)\Bigr\}.}
\end{equation}

With the help of \eqref{eq:LlambdaN} and for ${\bf
  h}=(h_{0},h_{1})\in\mathcal{D}_n$, we define the ${\bf h}$-tilted distribution by
\begin{equation}\label{definH}
\frac{\text{d}\mathbf{P}_{n,{\bf h}}}{\text{d}\mathbf{P}_{\beta}}(V)=e^{h_{0}Y_n+h_{1}V_n-\Ll_{\Lambda_n}(H)}.
\end{equation}
For a given $n\in \N$ and $q\in \frac{\N}{n}$, the exponential tilt is given by ${\bf h}_n^q:=(h_{n,0}^q,h_{n,1}^q)$ which, by Lemma \ref{exist} in Section \ref{sec:mgf}, is the unique solution of  
\begin{equation}\label{eq:tildeEC}
\mathbf{E}_{n,{\bf h}}(\tfrac{\Lambda_n}{n})=\nabla\bigl[\tfrac{1}{n}\Ll_{\Lambda_n}\bigr](\mathbf{h})=(q,0),
\end{equation}
and therefore, we have the equality
\begin{equation}\label{revt}
\mathbf{P}_\beta\bigl(\Lambda_n=(nq,0)\bigr)=\mathbf{P}_{n,{\bf h}_n^q}\bigl(\Lambda_n=(nq,0)\bigr)
e^{ n \big( -h_{n,0}^q\,q +\tfrac{1}{n}\Ll_{\Lambda_n}({\bf h}_n^q)\big)}.
\end{equation}
From \eqref{revt} it is easy to deduce that the exponential decay rate of $\mathbf{P}_\beta\bigl(\Lambda_n=(nq,0)\bigr)$ is given by the quantity $-h_{n,0}^q\,q +\frac1n \Ll_{\Lambda_n}({\bf h}_n^q) $ and that the polynomial correction is associated with $\mathbf{P}_{n,{\bf h}_n^q}\bigl(\Lambda_n=(nq,0)\bigr)$. To be more specific, we first state a Proposition which gives a local central limit theorem for the tilted law $\mathbf{P}_{n,{\bf h}_n^q}$. 
\begin{proposition}\label{lem:bHtilde}
For $[q_1,q_2]\subset(0,\infty)$, there exist $C>0,n_0>0$ such that
for all\footnote{to be thorough, we should restrict ourselves to $q$
  such that $n^2 q \in \mathbb{N}$. To ease notations, we shall omit
  this restriction in the sequel} $q\in[q_1,q_2]$ and $n\geq n_0$ we have
\begin{equation}
\tfrac{1}{Cn^2}\leq\mathbf{P}_{n,{\bf h}_n^q}(Y_n=nq,\,V_n=0)\leq\tfrac{C}{n^2}.
\end{equation}
\end{proposition}
\noindent 

The following  Proposition shows that the exponential decay rate induced by
the change of probability in \eqref{definH} can be controlled
uniformly in $n$.
\begin{proposition}[Decay rate of large area probability]\label{convunif}
For $[q_1,q_2]\subset(0,+\infty)$, there exist $c_1,c_2>0$ and $n_0\in\N$ such that 
\begin{equation}\label{eq:convuniff}
\bigl|\bigl[\ullamn({\bf h}_n^q)-h_{n,0}^q\,q\bigr]-\bigl[\Llam(\tilde {\bf h}(q,0))-\tilde h_0(q,0)\,q\bigr]\bigr|\leq \tfrac{c_1}{n},\quad\text{for}\ n\geq n_0,\  q\in [q_1,q_2].
\end{equation}
and 
\begin{equation}\label{eq:convuniff1}
\big|\big|{\bf h}_n^q-\tilde {\bf h}(q,0)\big|\big|\leq \tfrac{c_2}{n},\quad\text{for}\ n\geq n_0,\  q\in [q_1,q_2].
\end{equation}
\end{proposition}

%

Propositions \ref{lem:bHtilde} and  \ref{convunif}   will be proven in Sections 
\ref{sec:lclt} and  \ref{sec:mgf}, respectively. With the help of \eqref{revt} and by applying Proposition \ref{lem:bHtilde} and Proposition \ref{convunif} we can finally give some sharp upper and lower bounds of $\mathbf{P}_{\beta}(Y_n=nq,\,V_n=0)$.
\begin{proposition}\label{lem:impor}
For $[q_1,q_2]\subset(0,\infty)$, there exist  $C_1>C_2>0$ and $n_0\in \N$ such that for all $q\in[q_1,q_2]$ and $n\geq n_0$ we have
\begin{equation}
\tfrac{C_2}{n^2}\, e^{n \big[-\tilde h_{0}(q,0)\,q +\Ll_{\Lambda}(\tilde {\bf h}(q,0))\big]}\leq\mathbf{P}_{\beta}(Y_n=n q,\,V_n=0)\leq\tfrac{C_1}{n^2}\,e^{n\big[-\tilde h_{0}(q,0)\,q +\Ll_{\Lambda}(\tilde {\bf h}(q,0))\big]}.
\end{equation}
\end{proposition}

In addition, we shall need in this paper a precise lower bound on the probability that, under 
$\mathbf{P}_{\beta}$, the random walk $V$ makes only one excursion away from the origin,
conditionally on having a large prescribed area. To our knowledge, such an estimate is not available in the existing literature.
Recall the definition of $Y_n$ in \eqref{eq:Yn}.
\begin{proposition}[Unique excursion for large area]\label{prop:comppos}
For $[q_1,q_2]\subset (0,\infty)$, there exist  $C>0,\mu>0$ and  $n_0\in \N$ such that for all $q\in[q_1,q_2]$ and every $n\geq n_0$
\begin{equation}\label{propolb}
\mathbf{P}_\beta\bigl(V_i>0,0<i<n \mid Y_n=nq,\,V_n=0\bigr)\geq\tfrac{C}{n^\mu}.
\end{equation}
\end{proposition}
Although we can show that for the tilted law ${\bf P}_{n, {\bf h}_n^q}$  (thanks to the  positive, respectively negative drifts of the increments close to $0$, resp. close to $n$) there exists a $C(q_1,q_2)>0$ so that for $q\in [q_1,q_2]$ and $n$ large enough
$${\bf P}_{n, {\bf h}_n^q}  \bigl(V_i>0,0<i<n \mid V_n=0\bigr)>C(q_1,q_2),$$
and although we think that a similar result holds true for the  l.h.s. in \eqref{propolb}, 
we are unable to handle the conditioning by $Y_n=nq$ satisfactorily.

\section{The order of the phase transition}\label{sec:heur}
In Section \ref{pfTh3} below, we prove Theorem \ref{Thm3} that expresses the excess free energy 
as the solution of an equation involving the function $\hbeta$ introduced in Section \ref{sec:chbeta}. In Section \ref{pfTh4}, we first state Lemma \ref{lem:hbeta} which provides the 
behavior of $\hbeta(\tilde f(\beta))$ close to $\beta_c$ and then we combine  this Lemma with Theorem \ref{Thm3} to complete the proof of Theorem \ref{Thm2}.
Finally, in Section \ref{sec:asymhbeta} we give a proof of Lemma \ref{lem:hbeta}.

\subsection{Proof of Theorem \ref{Thm3} (Free energy equation)}\label{pfTh3}
By the representation formula~\eqref{eq:partfunc} and the definition of $\tilde{f}$, we have $\tilde{f}(\beta)=\lim_{L\to\infty}\frac{1}{L}\log\tilde{Z}_{L,\beta}$, where
\begin{equation}\label{wrZ}
{\textstyle \tilde{Z}_{L,\beta}:=\sum_{N=1}^{L}\left(\Gamma(\beta)\right)^N\mathbf{P}_{\beta}(\mathcal{V}_{N+1,L-N}).}
\end{equation}
As a consequence, the excess free energy satisfies $\tilde{f}(\beta)=-\log R$ where $R$ is the radius of convergence of the generating function $G(z)=\sum_{L=1}^\infty\tilde{Z}_{L,\beta}\;z^L$, that is
\begin{equation}\label{eq:fforsup}
{\textstyle \tilde{f}(\beta)=\sup\big\{\delta\geq0\colon\sum_{L=1}^\infty\tilde{Z}_{L,\beta}\;e^{-\delta L}=+\infty\big\},}
\end{equation}
if the set is non-empty and $\tilde{f}(\beta)=0$ otherwise. We recall \eqref{defcV} and we use \eqref{wrZ} to rewrite the sum in \eqref{eq:fforsup} as
\begin{align}\label{rwtZZ}
\nonumber\sum_{L=1}^\infty\tilde{Z}_{L,\beta}\;e^{-\delta L}&=\sum_{L=1}^\infty\sum_{N=1}^L\bigl(\Gamma(\beta)e^{-\delta}\bigr)^N\sum_{\substack{V_0=V_{N+1}=0\\ A_N=L-N}}\mathbf{P}_\beta(V)\, e^{-\delta(L-N)}\\
\nonumber&=\sum_{L=1}^\infty\sum_{N=1}^L\bigl(\Gamma(\beta)e^{-\delta}\bigr)^N\mathbf{E}_\beta\Bigl(e^{-\delta A_N}\mathbf{1}_{\{A_N=L-N,\,V_{N+1}=0\}}\Bigr)\\
&=\sum_{N=1}^\infty\bigl(\Gamma(\beta)e^{-\delta}\bigr)^N\mathbf{E}_\beta\Bigl(e^{-\delta A_N}\mathbf{1}_{\{V_{N+1}=0\}}\Bigr).
\end{align}
Since $A_N=A_{N+1}$ on the set $\{V_{N+1}=0\}$ and by using the definition of $\mathfrak{h}_{N,\beta}(\delta)$ in \eqref{eq:funch2}, the equality \eqref{rwtZZ} becomes
\begin{equation}
\sum_{L=1}^\infty\tilde{Z}_{L,\beta}\;e^{-\delta L}=\sum_{N=1}^\infty\exp{\Big(N\, \big[\log\Gamma(\beta)-\delta+\tfrac{N+1}{N}\,\mathfrak{h}_{N+1,\beta}(\delta)\big]\Big)},
\end{equation}
which  together with \eqref{eq:fforsup} gives
${\textstyle \tilde{f}(\beta)=\sup\big\{\delta\geq 0\colon\log\Gamma(\beta)-\delta+\hbeta(\delta)>0\big\}.}$
Since $\hbeta(\delta)\leq 0$, it follows that $\tilde{f}(\beta)=0$ if $\Gamma(\beta)\leq 1$. When $\Gamma(\beta)>1$, Lemma \ref{funch} gives that $\delta\mapsto -\delta+\hbeta(\delta)$ is continuous, decreasing, non-positive on $[0,\infty)$, equals $0$ at $\delta=0$ and tends to $-\infty$ when $\delta\to\infty$. Therefore, $\tilde{f}(\beta)>0$ and is the unique solution of the equation $\log\Gamma(\beta)-\delta+\hbeta(\delta)=0$. In addition, by recalling the definition of the collapsed phase \eqref{eq:colphase} and the extended phase \eqref{eq:extphase}, we can observe that
\begin{equation}
\mathcal{C}=\{\beta\colon\Gamma(\beta)\leq 1\}\quad\text{and}\quad\mathcal{E}=\{\beta\colon\Gamma(\beta)>1\}.
\end{equation}
We note that $\beta\mapsto\Gamma(\beta)$ is decreasing on $[0,\infty)$ (recall \eqref{lawP} and \eqref{sqq}) and therefore, the collapse transition occurs at $\beta_c$, the unique positive solution of the equation $\Gamma(\beta)=1$.

\subsection{Proof of Theorem \ref{Thm2} (Phase transition asymptotics)}\label{pfTh4}
We display here the proof of Theorem \ref{Thm2} subject to Lemma \ref{lem:hbeta} below,
that will be proven in Section \ref{sec:asymhbeta} afterward.


\begin{lemma}\label{lem:hbeta}
\begin{equation}
\lim_{\beta\to\beta_c}\,\frac{\mathfrak{h}_{\beta}\bigl(\tilde{f}(\beta)\bigr)}{\tilde{f}(\beta)^{2/3}}=-\varsigma_2.
\end{equation}
where we recall that $\varsigma_2$ was defined in \eqref{c1}.
\end{lemma}

Our aim is to study the asymptotic behavior of the equation in Theorem \ref{Thm3} near the critical point. We recall \eqref{sqq} and we perform a first order Taylor expansion of $\Gamma(\beta)$ near $\beta_c$ which gives 
$\log \Gamma(\beta_c-\gep)=\varsigma_1\epsilon (1+o(1))$ as $\epsilon\searrow 0$.
Next, we consider the function $\hbeta$ near $\beta_c$ and it follows from Lemma \ref{lem:hbeta} that when $\epsilon\searrow 0$
\begin{equation}\label{eq:delt}
h_{\beta_c-\epsilon}(\tilde{f}(\beta_c-\epsilon))= -\varsigma_2\tilde{f}(\beta_c-\epsilon)^{2/3} (1+o(1)).
\end{equation}
Therefore, by plugging \eqref{eq:delt} and the expansion of $\log \Gamma(\beta_c-\gep)$ in the equation in Theorem \ref{Thm3} that is verified by the excess free energy, we obtain that
\begin{equation}
\varsigma_1\epsilon (1+o(1))-\tilde{f}(\beta_c-\epsilon)-\varsigma_2\tilde{f}(\beta_c-\epsilon)^{2/3}(1+o(1))=0,
\end{equation} 
which allows to conclude that
\begin{equation}
\tilde{f}(\beta_c-\epsilon)\sim \Bigl(\frac{\varsigma_1}{\varsigma_2}\Bigr)^{3/2}\,\epsilon^{3/2}\quad\text{as}\ \ \epsilon\searrow 0,
\end{equation}
and the proof is complete.


\subsection{Asymptotics of \texorpdfstring{$\hbeta$}{h}}\label{sec:asymhbeta}

\subsubsection{Heuristics} 
Let us give the heuristic explanation of why $\hbeta(\delta)\sim -c\,\delta^{2/3}$ for some constant $c>0$. The main idea is to decompose the trajectory of the random walk $V$ into independent blocks of length $T\delta^{-2/3}$ for $T\in\mathbb{N}$ and $\delta$ small enough: we have approximately $N/(T\delta^{-2/3})$ such blocks. Hence, as $\delta\searrow 0$, we can estimate
\begin{equation}\label{eq:hhbeta}
\lim_{N\to\infty}\frac{1}{N}\log\mathbf{E}_\beta(e^{-\delta A_N})\sim\lim_{T\to\infty}\frac{\delta^{2/3}}{T}\log\mathbf{E}_\beta(e^{-\delta A_{T\delta^{-2/3}}}).
\end{equation}
It is well known that for such random walks (assume that $\mathbf{E}_\beta(U_1^2)=1$) (see \cite[p.~405]{RD05})
\begin{equation}
k^{-3/2}\sum_{i=1}^{Tk}|V_i|\xrightarrow{\mathcal{L}}\int_0^T|B(t)|dt\quad\text{as}\ k\to\infty,
\end{equation}
where $B$ is a standard Brownian motion. Now, let $k=\delta^{-2/3}$ and since $|e^{-\delta A_{T\delta^{-2/3}}}|\leq1$, we conclude that
\begin{equation}\label{eq:Donsker}
\mathbf{E}_\beta(e^{-\delta A_{T\delta^{-2/3}}})\to\mathbf{E}(e^{-\int_0^T|B(t)|dt})\quad\text{as}\ \delta\to 0.
\end{equation}
This convergence and \eqref{eq:hhbeta} would immediately imply $\hbeta(\delta)\sim -c\,\delta^{2/3}$ where $c$ can be estimated via the distribution of the \textit{Brownian area}, that is
\begin{equation}
c=-\lim_{T\to\infty}\frac{1}{T}\log\mathbf{E}(e^{-\int_0^T|B(t)|dt})>0.
\end{equation}

\begin{proof}[Proof of Lemma \ref{lem:hbeta}]
\quad

\subsubsection{Upper bound}\label{sec:upp}
Pick $T\in\mathbb{N}$, $\delta>0$ such that $\delta^{-2/3}\in\mathbb{N}$ and let $\Delta:=\delta^{-2/3}$. We take $N$ that satisfies $N/(T\Delta)\in\mathbb{N}$ and partition $\{1,\ldots,N\}$ into $k=N/(T\Delta)$ intervals of length $T\Delta$. By the Markov property of $V$, we decompose $\mathbf{E}_\beta\bigl(e^{-\delta A_N}\bigr)$ with respect to the position occupied by the random walk $V$ at times $T\Delta,2T\Delta,\ldots,(k-1)T\Delta$,
\begin{equation}\label{eq:uppbouEarea}
\mathbf{E}_\beta\bigl(e^{-\delta A_N}\bigr)=\sum_{\substack{x_0=0,x_i\in\mathbb{Z}\\i=1,\ldots,k}}\;\prod_{i=0}^{k-1}\mathbf{E}_{\beta,x_i}\Bigl(e^{-\delta A_{T\Delta}}\mathbf{1}_{\{V_{T\Delta}=x_{i+1}\}}\Bigr)\leq\Bigl[\;\sup_{x\in\mathbb{Z}}\mathbf{E}_{\beta,x}\bigl(e^{-\delta A_{T\Delta}}\bigr)\Bigr]^k.
\end{equation}
With the help of Lemma \ref{ineEarea} below, we can replace the supremum in the right hand side of  \eqref{eq:uppbouEarea} by the term indexed by $x=0$ only. The proof of Lemma  \ref{ineEarea} is postponed to Appendix \ref{appA}.
\begin{lemma}\label{ineEarea}
For all $\delta>0,n\in\mathbb{N}$ and $x,x'\in\mathbb{Z}$ such that $|x'|\geq|x|$, the following inequality holds true
\begin{equation}
\mathbf{E}_{\beta,x'}\bigl(e^{-\delta A_n}\bigr)\leq\mathbf{E}_{\beta,x}\bigl(e^{-\delta A_n}\bigr).
\end{equation}
\end{lemma}
\noindent Therefore \eqref{eq:uppbouEarea} becomes
\begin{equation}\label{eq:del}
\mathbf{E}_\beta\bigl(e^{-\delta A_N}\bigr)\leq\Bigl[\mathbf{E}_{\beta}\bigl(e^{-\delta A_{T\Delta}}\bigr)\Bigr]^{N/(T\Delta)}.
\end{equation}
Recall that $\Delta:=\delta^{-2/3}$, apply $\frac{1}{N}\log$ to both sides of  \eqref{eq:del} and let $N\to \infty$ to obtain, for $\beta>0$ and $\delta>0$, that 
\begin{equation}\label{eq:hbetaupp}
\frac{\hbeta(\delta)}{\delta^{2/3}}\leq\frac{1}{T}\log\mathbf{E}_{\beta}\bigl(e^{-\delta A_{T\Delta}}\bigr).
\end{equation}
In what follows we need a uniform version (in $\beta$) of the convergence of $\mathbf{E}_{\beta}\bigl(e^{-\delta A_{T\Delta}}\bigr)$ towards $\mathbf{E}(e^{-\int_0^T|B(t)|dt})$ as $\delta\to 0$. For this reason, we introduce the strong approximation theorem (Sakhanenko \cite{Sa80}) to approximate the partial sums of independent random variables $U$ in the right hand side in \eqref{eq:hbetaupp} by independent normal random variables.

\begin{theorem}[Q. M. Shao \cite{Sh95}, Theorem B]\label{thm:shao}
Denote by $\sigma_\beta^2$ the variance of the random variable $U_1$ under $\mathbf{P}_\beta$. We can redefine $\{U_i,i\geq 1\}$ (denoted by $U^\beta$) on a richer probability space together with a sequence of independent standard normal random variables $\{X_i,i\geq 1\}$ such that for every $p>2$, $x>0$,
\begin{equation}
\mathbf{P}\biggl(\max_{i\leq n}\,\biggl|\sum_{j=1}^i U^\beta_j-\sigma_\beta\sum_{j=1}^i X_j\biggr|\geq x\biggr)\leq (Ap)^px^{-p}\sum_{i=1}^n\mathbf{E}|U^\beta_i|^p,
\end{equation}
where $A$ is an absolute positive constant.
\end{theorem}

We  let also, for $n\in \N$,  $Y_n=\sum_{i=1}^n X_i$, $A_n(Y)=\sum_{i=1}^n|Y_i|$ and redefine $V_n^\beta=\sum_{i=1}^n U_i^\beta$, $A_n(V^\beta)=\sum_{i=1}^n|V_i^\beta|$. We pick $T>0$, $p>2$, $\theta>0$ and $K$ a compact subset of $(0,\infty)$.  We use Theorem \ref{thm:shao} and the fact that (recall \eqref{lawP}) $\mathbf{E}\bigl[|U_1^\beta|^p\bigr]$ is bounded from above uniformly in $\beta\in K$,  to assert  that there exists a constant $c_{p,K}>0$ such that for all $\Delta>0$ and $\beta\in K$
\begin{equation}
\mathbf{P}\Bigl(\,\max_{i\leq T\Delta}\,\bigl|V_i^\beta-\sigma_\beta Y_i\bigr|\geq \Delta^\theta\Bigr)\leq c_{p,K}\, T \,\Delta^{1-\theta p}.
\end{equation}
Note that on the event $\{\max_{i\leq T\Delta}\,\bigl|V_i^\beta-\sigma_\beta Y_i\bigr|< \Delta^\theta\}$, we obviously have $|A_{T\Delta}(V^\beta)-\sigma_\beta A_{T\Delta}(Y)|\leq T\Delta^{\theta+1}$. Therefore, since $x\mapsto \exp(-x)$ is 1-Lipschitz on $[0,\infty)$ and since  $\Delta=\delta^{-2/3}$, we can write that for $\beta\in K$ and $\delta>0$
\begin{align}\label{eq:approxvy}
\nonumber\bigl|\mathbf{E}\bigl(e^{-\delta A_{T\Delta}(V^\beta)}-e^{-\delta\sigma_\beta A_{T\Delta}(Y)}\bigr)\bigr|&\leq \mathbf{P}\Bigl(\,\max_{i\leq T\Delta}\,\bigl|V_i^\beta-\sigma_\beta Y_i\bigr|\geq \Delta^\theta\Bigr)+\delta T\Delta^{\theta+1}\\
&\leq c_{p,K} T \delta^{\frac23(\theta p-1)}+T\delta^{\frac13(1-2\theta)}.
\end{align}
We chose $p=3$ and $\theta\in(1/3,1/2)$ and plug it in the right hand side of \eqref{eq:hbetaupp} to obtain that for $\beta\in K$ and $\delta>0$,
\begin{equation}\label{eq:hbetaupp2}
\frac{\hbeta(\delta)}{\delta^{2/3}}\leq\frac{1}{T}\log\Bigl[\mathbf{E}\bigl(e^{-\delta\sigma_\beta A_{T\Delta}(Y)}\bigr)+c_{3,K} T\delta^{\frac{2(3\theta-1)}{3}}+T\delta^{\frac{1-2\theta}{3}}\Bigr].
\end{equation}
\begin{lemma}\label{lem:BrowN}
Let K be a compact subset of $(0,+\infty)$. For $T>0$ and  $\epsilon>0$ there exists a $\delta_0>0$ such that for  $\delta\leq\delta_0$ (with $\Delta=\delta^{-2/3}$),
\begin{equation}\label{tbbpr}
\sup_{\beta\in K}\,\bigl|\mathbf{E}\bigl(e^{-\delta\sigma_\beta A_{T\Delta}(Y)}\bigr)-\mathbf{E}\bigl(e^{-\sigma_\beta\int_0^T|B(t)|dt}\bigr)\bigr|<\epsilon,
\end{equation}
where $B$ is a standard Brownian motion.
\end{lemma}

\begin{proof}[Proof of Lemma \ref{lem:BrowN}]
We can consider $\{B(t),t\geq0\}$ and $\{y_i,i\geq1\}$ on the same probability space by letting $y_i=B(i)-B(i-1)$ and thus $Y_i:=y_1+\dots+y_i=B(i)$ for $i\in\mathbb{N}$. We recall that $A_{T\Delta}(Y)=\sum_{i=1}^{T\Delta}|B(i)|$ and therefore, by Brownian scaling we note that  
$$\Delta^{-3/2}A_{T\Delta}(Y)\overset{d}{=}\Delta^{-1}\sum_{i=1}^{T\Delta}|B(i/\Delta)|.$$
Consequently, by recalling that $\delta=\Delta^{-3/2}$ we can replace 
$\mathbf{E}\bigl(e^{-\delta\sigma_\beta A_{T\Delta}(Y)}\bigr)$ in the left hand side of  \eqref{tbbpr} 
by $\mathbf{E}\bigl(e^{-\sigma_\beta \Delta^{-1}\sum_{i=1}^{T\Delta}|B(i/\Delta)|}\bigr)$.  Since the exponential function is 1-Lipschitz on $(-\infty,0]$, we have
\begin{align}
\sup_{\beta\in K}\,\bigl|\mathbf{E}\bigl(e^{-\sigma_\beta \Delta^{-1}\sum_{i=1}^{T\Delta}|B(i/\Delta)|}\bigr)&-\mathbf{E}\bigl(e^{-\sigma_\beta\int_0^T|B(t)|dt}\bigr)\bigr|\\
\nonumber &{\textstyle\leq \max\{\sigma_\beta, \beta\in K\}\ \mathbf{E}\Bigl[\bigl|\Delta^{-1}\sum_{i=1}^{T\Delta}|B(i/\Delta)|-\textstyle\int_0^T|B(t)|dt\bigr|\Bigr]}.
\end{align}
Since $\max\{\sigma_\beta, \beta\in K\}<\infty$, since by Riemann sum approximation we know that
\begin{equation}
\Delta^{-1}\sum_{i=1}^{T\Delta}|B(i/\Delta)|\xrightarrow[\Delta\to\infty]{a.s.}\int_0^T|B(t)|dt,
\end{equation}
and since we have uniform integrability (because   
$\sup_{\Delta>0}\mathbf{E}(|\Delta^{-1}\sum_{i=1}^{T\Delta}|B(i/\Delta)|^2)<\infty$) we can conclude that
\begin{equation}\label{eq:EDonsker}
\lim_{\Delta\to\infty}   \mathbf{E}\Bigl[\bigl|\Delta^{-1}\sum_{i=1}^{T\Delta}|B(i/\Delta)|-\textstyle\int_0^T|B(t)|dt\bigr|\Bigr]=0.
\end{equation}
This completes the proof.
\end{proof}

We resume the proof of the upper bound. Since  $\theta\in(1/3,1/2)$, the right hand side of \eqref{eq:approxvy} vanishes as $\delta\to 0$ uniformly in $\beta\in K$. Thus, we can  replace $\delta$ by $\tilde{f}(\beta_c)$ in \eqref{eq:hbetaupp2} and use Lemma \ref{lem:BrowN} and the fact that $\lim_{\gep\to 0^+}\tilde{f}(\beta_c-\gep)=0$ to conclude that, for all $T>0$,
\begin{equation}
\limsup_{\gep\to 0^+}\frac{\hbeta(\tilde{f}(\beta_c-\gep))}{\tilde{f}(\beta_c-\gep)^{2/3}}\leq\frac{1}{T}\log\mathbf{E}\bigl(e^{-\sigma_{\beta_c}\int_0^T|B(t)|dt}\bigr).
\end{equation}
It remains to let $T$ tend to infinity and to recall \eqref{c1} to obtain
\begin{equation}
\limsup_{\gep\to 0^+}\frac{\hbeta(\tilde{f}(\beta_c-\gep))}{\tilde{f}(\beta_c-\gep)^{2/3}}\leq-\varsigma_2.
\end{equation}

\subsubsection{Lower bound}\label{sec:low}
Recall that $T\in\mathbb{N},\delta>0$ and $\Delta=\delta^{-2/3}\in\mathbb{N}$. We also take $N\in\mathbb{N}$ such that $N/(T\Delta)\in\mathbb{N}$. Pick $\eta>0$ and use the decomposition in \eqref{eq:uppbouEarea} to obtain 
\begin{align}\label{eq:lowbouEarea}
\mathbf{E}_\beta\bigl(e^{-\delta A_N}\bigr)&\geq\sum_{\substack{x_0=0,x_i\in[-\eta\sqrt{\Delta},\eta\sqrt{\Delta}]\\i=1,\ldots,k}}\;\prod_{i=0}^{k-1}\mathbf{E}_{\beta,x_i}\Bigl(e^{-\delta A_{T\Delta}}\mathbf{1}_{\{V_{T\Delta}=x_{i+1}\}}\Bigr)\\
&\geq\Bigl[\;\inf_{x\in[-\eta\sqrt{\Delta},\eta\sqrt{\Delta}]}\mathbf{E}_{\beta,x}\Bigl(e^{-\delta A_{T\Delta}}\mathbf{1}_{\{V_{T\Delta}\in[-\eta\sqrt{\Delta},\eta\sqrt{\Delta}]\}}\Bigr)\Bigr]^{N/(T\Delta)}.
\end{align}
For any integer $x\in[-\eta\sqrt{\Delta},\eta\sqrt{\Delta}]$, we consider the two sets of paths
\begin{equation}
\Pi_1^x=\bigl\{(V_i)_{i=0}^{T\Delta}\colon V_0=x,\,V_{T\Delta}\in[-\eta\sqrt{\Delta},\eta\sqrt{\Delta}]\bigr\},
\end{equation}
and
\begin{equation}
\Pi_2=\bigl\{(V_i)_{i=0}^{T\Delta}\colon V_0=0,\,V_{T\Delta}\in[-\eta\sqrt{\Delta},0]\bigr\}.
\end{equation}
Clearly, if $V=(V_i)_{i=0}^{T\Delta}\in\Pi_2$, then the trajectory $V+x$ starts at $x\in[0,\eta\sqrt{\Delta}]$ and is an element of $\Pi_1^x$. Similarly, for $x\in[-\eta\sqrt{\Delta},0]$, $\Pi'_2+x\subseteq\Pi_1^x$ where
\begin{equation}
\Pi'_2=\bigl\{(V_i)_{i=0}^{T\Delta}\colon V_0=0,\,V_{T\Delta}\in[0,\eta\sqrt{\Delta}]\bigr\}.
\end{equation}
Since $\mathbf{P}_\beta(V\in\Pi_2)=\mathbf{P}_\beta(V\in\Pi'_2)$, we conclude that
\begin{equation}\label{eq:1stlowb}
\mathbf{P}_{\beta,x}(V\in\Pi_1^x)\geq\mathbf{P}_\beta(V\in\Pi'_2)\quad\text{for all}\;x\in[-\eta\sqrt{\Delta},\eta\sqrt{\Delta}].
\end{equation}
Moreover, for any $V^\star\in\Pi_1^x$,
\begin{equation}\label{eq:2ndlowb}
\delta\sum_{i=1}^{T\Delta}|V^\star_i|=\delta\sum_{i=1}^{T\Delta}|x+V_i|\leq\delta\sum_{i=1}^{T\Delta}|V_i|+\delta T\Delta|x|\leq\delta\sum_{i=1}^{T\Delta}|V_i|+\eta T,
\end{equation}
where the trajectory $V$ satisfies $V_0=0$. Combining \eqref{eq:1stlowb} and \eqref{eq:2ndlowb}, we then have, for $x\in[-\eta\sqrt{\Delta},\eta\sqrt{\Delta}]$,
\begin{equation}
\mathbf{E}_{\beta,x}\Bigl(e^{-\delta A_{T\Delta}}\mathbf{1}_{\{V_{T\Delta}\in[-\eta\sqrt{\Delta},\eta\sqrt{\Delta}]}\Bigr)\geq e^{-\eta T}\mathbf{E}_{\beta}\Bigl(e^{-\delta A_{T\Delta}}\mathbf{1}_{\{V_{T\Delta}\in[0,\eta\sqrt{\Delta}]}\Bigr).
\end{equation}
By plugging the lower bound above into \eqref{eq:lowbouEarea} and by using the symmetry of $V$ we immediately get
\begin{equation}\label{eq:bpt}
\mathbf{E}_\beta\bigl(e^{-\delta A_N}\bigr)\geq\Bigl[e^{-\eta T}\mathbf{E}_{\beta}\Bigl(e^{-\delta A_{T\Delta}}\mathbf{1}_{\{V_{T\Delta}\in\,  [0,\eta \sqrt{\Delta}]\}}\Bigr)\Bigr]^{N/T\Delta},
\end{equation}
which, by applying $\frac{1}{N} \log$ to both sides in \eqref{eq:bpt} and by letting $N\to\infty$, gives, for all $\beta>0$,
\begin{equation}
\frac{\hbeta(\delta)}{\delta^{2/3}}\geq\frac{1}{T}\log\mathbf{E}_{\beta}\Bigl(e^{-\delta A_{T\Delta}}\mathbf{1}_{\{V_{T\Delta}\in [0,\eta \sqrt{\Delta}]\}}\Bigr)-\eta,\quad \delta,\eta>0.
\end{equation}
At this stage, we proceed as in the upper bound (from \eqref{eq:hbetaupp}) to obtain, for all $T\in\mathbb{N},\eta>0$,
\begin{equation}
\liminf_{\beta\to\beta_c}\frac{\hbeta(\tilde{f}(\beta))}{\tilde{f}(\beta)^{2/3}}\geq\frac{1}{T}\log\mathbf{E}\bigl(e^{-\sigma_{\beta_c}\int_0^T|B(t)|dt}\mathbf{1}_{\{B(T)\in[0,\eta]\}}\bigr)-\eta.
\end{equation}
It remains to show that  for all $\eta>0$ we have
\begin{equation}\label{eq:etaBrow}
\lim_{T\to\infty}\frac{1}{T}\log\mathbf{E}\bigl(e^{-\sigma_{\beta_c}\int_0^T|B(t)|dt}\mathbf{1}_{\{B(T)\in[0,\eta]\}}\bigr)=\lim_{T\to\infty}\frac{1}{T}\log\mathbf{E}\bigl(e^{-\sigma_{\beta_c}\int_0^T|B(t)|dt}\bigr),
\end{equation}
but the latter convergence can be obtained by adapting the proof of \eqref{eq:ppr} to the continuous setting and for conciseness we will not give the details of the proof here. 
Then, by recalling \eqref{c1}, we achieve the bound
\begin{equation}
\liminf_{\beta\to\beta_c}\frac{\hbeta(\tilde{f}(\beta))}{\tilde{f}(\beta)^{2/3}}\geq-\varsigma_2-\eta,
\end{equation}
for all $\eta>0$. It remains to let $\eta\to 0$ to complete the proof.

\end{proof}

\section{Geometry of the collapsed phase}\label{geo}

In Section \ref{Th44} below, a proof of Theorem \ref{Thm4} is displayed subject to Lemma \ref{lem2}, which ensures that the horizontal extension of the polymer inside the collapsed phase is of order $\sqrt{L}$, and to Proposition \ref{prop1}, which provides a sharp estimate of the partition function restricted to those trajectories making only one bead. Proposition  \ref{prop1} is proven in Section \ref{pr:prop1} subject  to Lemma \ref{lem1}, which is the counterpart of Lemma \ref{lem2} for the one bead trajectory and to Proposition \ref{prop:comppos}, which gives a lower bound on the probability that the random walk $V$ makes an $n$-step excursion away from the origin  conditioned on the large deviation event $\{Y_n=q n, V_n=0\}$. Lemmas \ref{lem2} and \ref{lem1} are proven in Section \ref{pr:lem2} 
whereas the proof of Proposition \ref{prop:comppos} is postponed to  Section \ref{sec:wulff} because it requires more preparation. Section \ref{pp5} is dedicated to the proof of Theorem \ref{Prop5} and Section \ref{pr:Thm7} 
to the proof of Theorem  \ref{Convenv}.

\subsection{Proof of Theorem \ref{Thm4} (One bead Theorem)}\label{Th44}

The proof of Theorem \ref{Thm4} will be displayed subject to Lemma \ref{lem2} and Proposition \ref{prop1} that are stated below.

\begin{lemma}\label{lem2}
For  $\beta>\beta_c$, there exist $a,a_1,a_2>0$ such that
\begin{equation}\label{rst}
P_{L,\beta}(N_L(l)\geq a_1\sqrt{L})\leq a_2\,e^{-a \sqrt L},\quad L\in \N.
\end{equation}
\end{lemma}

Recall (\ref{eq:dist1}--\ref{eq:dist11})


\begin{proposition}\label{prop1}
For  $\beta>\beta_c$, there exist $c,c_1,c_2>0$ and $\kappa>1/2$ such that 
\begin{equation}\label{state}
\frac{c_1}{L^\kappa}\,e^{\beta L-c\sqrt{L}}\leq Z_{L,\beta}^{\mathsf{o}}\leq\frac{c_2}{\sqrt{L}}\,e^{\beta L-c\sqrt{L}},\quad L\in \N.
\end{equation}
\end{proposition}

\subsubsection{Proof of Theorem \ref{Thm4}} We will first show that, for $\beta>\beta_c$ and under the polymer measure, the probability that there is exactly one macroscopic bead in the polymer tends to $1$ as $L\to\infty$. Then, we will show that, with a probability converging to $1$ as $L\to \infty$, the first step and the last step of this macroscopic bead are at distance less than $(\log L)^4$ from $0$ and $L$, respectively. For $r\in\N$, we denote by $Z_{L,\beta}[r]$ the partition function restricted to those trajectories that do not have any bead larger than $r$, i.e.,

\begin{equation}
Z_{L,\beta}[r]=\sum_{l\in\Omega_L\colon |I_{j_{\text{max}}}|\leq r}e^{\beta H_{L}(l)}.
\end{equation}
At this stage, we pick $s>0$ and we let $\cA_{L,s}$ be the subset consisting of those trajectories having at most one bead larger than $s(\log L)^2$, i.e.,
\begin{equation}
\cA_{L,s}=\Bigl\{l\in\Omega_L\colon\bigl|\bigl\{j\in\{1,\dots,n_L(l)\}\colon\,|I_{j}|\geq s(\log L)^2\bigr\}\bigr|\leq 1\Bigr\}.
\end{equation}
Partition $\cA_{L,s}^c$ with respect to the locations of the two
subintervals $\{i_1+1,\dots,i_2\}$ and $\{i_3+1,\dots,i_4\}$
associated with the first two beads that are larger than $s (\log
L)^2$. For notational convenience we let $L_1:=i_2-i_1$ and
$L_2:=i_4-i_3$ be the length of these two first large beads. We do not
have Markov property but, with the help of Lemma \ref{lem:decoZ} below, we can estimate the partition function restricted to those trajectory that make a bead between two given steps.

Recall (cf.  notations introduced in Section \ref{maires} prior to Theorem \ref{Thm4}) that $x_1$ denotes the horizontal extension of the first bead, and that $u_{x_1}$ corresponds to its total length.
\begin{lemma}\label{lem:decoZ}  For $L\in \N$,  
\begin{equation}
\tfrac{1}{2}\,Z^\mathsf{o}_{L',\beta}\,Z_{L-L',\beta}\leq Z_{L,\beta}(u_{x_1}=L')\leq Z^\mathsf{o}_{L',\beta}\,Z_{L-L',\beta}\quad\text{for}\;L'\in\{1,\ldots,L\}.
\end{equation}
\end{lemma}
\begin{proof}[Proof of Lemma \ref{lem:decoZ}]
In the case $u_{x_1}=1$, the first bead contains only one horizontal step, hence the sign of the stretch after $x_1$ is arbitrary, so that obviously $Z_{L,\beta}(u_{x_1}=1)=Z^\mathsf{o}_{1,\beta}Z_{L-1,\beta}$. In case $u_{x_1}=L'>1$, note that the stretch $l_{x_1}$ is non-zero, therefore the next stretch has the same sign as $l_{x_1}$. By concatenating the trajectories
\begin{align}
Z_{L,\beta}(u_{x_1}=L')&=Z^\mathsf{o}_{L',\beta}(l_{N_{L'}}>0)Z_{L-L',\beta}(l_1\geq0)+Z^\mathsf{o}_{L',\beta}(l_{N_{L'}}<0)Z_{L-L',\beta}(l_1\leq0)\\
&=Z^\mathsf{o}_{L',\beta}\,Z_{L-L',\beta}(l_1\geq0).
\end{align}
In both cases, thanks to the symmetry of the stretches, we have
\begin{equation}
\tfrac{1}{2}\,Z^\mathsf{o}_{L',\beta}\,Z_{L-L',\beta}\leq Z_{L,\beta}(u_{x_1}=L')\leq Z^\mathsf{o}_{L',\beta}\,Z_{L-L',\beta}\quad\text{for}\;L'\in\{1,\ldots,L\}.
\end{equation}
\end{proof}

We resume the proof of  Theorem \ref{Thm4} and, we use Lemma \ref{lem:decoZ} to obtain
\begin{equation}
P_{L,\beta}(\cA_{L,s}^c)\leq\sum_{\stackrel{1\leq i_1<i_2<i_3<i_4\leq L}{L_1,L_2\geq s(\log L)^2}}\frac{Z_{i_1,\beta}\, \big[s(\log L)^2\big]\,Z_{L_1,\beta}^\mathsf{o}\, Z_{i_3-i_2,\beta}\big[s(\log L)^2\big]\,Z_{L_2,\beta}^\mathsf{o}\,Z_{L-i_4,\beta}}{Z_{L,\beta}},
\end{equation}
and we write the lower bound
\begin{equation}
Z_{L,\beta}\geq(\tfrac{1}{2})^3 Z_{i_1,\beta}\big[s(\log L)^2\big]\,Z_{L_1+L_2,\beta}^\mathsf{o}\,Z_{i_3-i_2,\beta}\big[s(\log L)^2\big]Z_{L-i_4,\beta}
\end{equation}
such that 
\begin{equation}\label{eq:maj}
P_{L,\beta}(\cA_{L,s}^c)\leq8\sum_{\stackrel{1\leq i_1<i_2<i_3<i_4\leq L}{L_1,L_2\geq s(\log L)^2}}\frac{Z_{L_1,\beta}^\mathsf{o}\,Z_{L_2,\beta}^\mathsf{o}}{Z_{L_1+L_2,\beta}^\mathsf{o}}.
\end{equation}
By using Proposition \ref{prop1} and the convex inequality
\begin{equation}
\sqrt{L_1}+\sqrt{L_2}-\sqrt{L_1+L_2}\geq\tfrac{1}{2}\sqrt{\min\{L_1,L_2\}},
\end{equation}
we can  bound from above the quantity in the sum in \eqref{eq:maj} by
\begin{align}\label{eq:maj2}
\frac{Z_{L_1,\beta}^\mathsf{o}Z_{L_2,\beta}^\mathsf{o}}{Z_{L_1+L_2,\beta}^\mathsf{o}}&\leq\frac{c_1^2(L_1+L_2)^\kappa}{c_2\sqrt{L_1L_2}}\;e^{-\tilde G(a_{\beta}) [\sqrt{L_1}+\sqrt{L_2}-\sqrt{L_1+L_2}]}\\
&\leq\frac{c_1^2(L_1+L_2)^\kappa}{c_2\sqrt{L_1L_2}}\;e^{-\tfrac{\tilde G(a_{\beta}) \sqrt{s}\log L}{2}}
\end{align}
and since $\frac{(L_1+L_2)^\kappa}{\sqrt{L_1L_2}}\leq L^\kappa$ we can state that, for $L$ large enough, \eqref{eq:maj} becomes
\begin{equation}
P_{L,\beta}(\cA_{L,s}^c)\leq \tfrac{8c_1^2}{c_2}\, L^{\kappa+4} e^{-\tfrac{\tilde G(a_{\beta}) \sqrt{s}\log L}{2}}.
\end{equation}
Therefore, it suffices to choose $\sqrt{s}=\tfrac{4(\kappa+4)}{c}$ to conclude that $\lim_{L\to\infty}P_{L,\beta}(\cA_{L,s}^c)=0$.

At this stage we set $\cB_{L,s}=\cA_{L,s}\cap\{N_L(l)\leq a_1\sqrt{L}\}$ and we can use Lemma \ref{lem2} and the fact that $P_{L,\beta}(\cA_{L,s}^c)$ vanishes as $L\to\infty$ to conclude that $\lim_{L\to \infty}P_{L,\beta}(\cB_{L,s})=1$. Moreover, it comes easily that under the event $\cB_{L,s}$ there is exactly one bead larger than $s(\log L)^2$ because if there were no bead larger than $s(\log L)^2$, then the total number of beads $n_L(l)$ would be larger than $\frac{L}{s(\log L)^2}$ which contradicts the fact that $N_L(l)\leq a_1\sqrt{L}$ because each bead contains at least one horizontal step and consequently $N_L(l)\geq n_L(l)$. Under the event $\cB_{L,s}$ we denote by $i_1$ and $i_2$ the end-steps of the maximal bead, i.e., $I_{j_{\text{max}}}=\{i_1+1,\dots,i_2\}$. Then, the proof of Theorem \ref{Thm4} will be complete once we show that there exists a $v>0$ such that
\begin{align}
\lim_{L\to\infty}P_{L,\beta}(\cB_{L,s}\cap\{i_1\geq v(\log L)^4\})&=0\\
\lim_{L\to\infty}P_{L,\beta}(\cB_{L,s}\cap\{i_2\leq L-v(\log L)^4\})&=0.
\end{align}
We can bound from above
\begin{align}
\nonumber P_{L,\beta}(\cB_{L,s}\cap\{i_1\geq v(\log L)^4\})
&=\sum_{t=v(\log L)^4}^L P_{L,\beta}(\cB_{L,s}\cap\{i_1=t\})\\
\nonumber &\leq\sum_{t=v(\log L)^4}^L P_{L,\beta}\Big(\exists j\in\{1,\dots,n_L(l)\}\colon u_{x_j}=t,\\
\nonumber &\hspace{3.5cm}|I_d|\leq s(\log L)^2\quad\forall d\in\{1,\dots,j\}\Big)\\
&\leq\tfrac{1}{2}\sum_{t=v(\log L)^4}^L\frac{Z_{t,\beta}[s(\log L)^2]\,Z_{L-t,\beta}}{Z_{t,\beta}\,Z_{L-t,\beta}},
\end{align}
which finally gives
\begin{align}
P_{L,\beta}(B_{L,s}\cap\{i_1\geq v(\log L)^4\})
&\leq\tfrac{1}{2}\sum_{t=v(\log L)^4}^L P_{t,\beta}\big(|I_{j_{\text{max}}}|\leq s(\log L)^2\big).
\end{align}
We note that, under $P_{t,\beta}$ and on the event $\{|I_{j_{\text{max}}}|\leq s(\log L)^2\}$, the number of beads is larger than $\frac{t}{s(\log L)^2}$, therefore $N_{t}(l)\geq\frac{t}{s(\log L)^2}$ and since $\sqrt{t}\geq\sqrt{v}(\log L)^2$ we obtain that $N_{t}(l)\geq \sqrt{t}(\sqrt{v}/s)$. By choosing $v=(a_1 s)^2$, we can apply Lemma \ref{lem2} to get
\begin{align}\label{ths}
\nonumber P_{L,\beta}(\cB_{L,s}\cap\{i_1\geq v(\log L)^4\})&\leq\tfrac{1}{2}\sum_{t=v(\log L)^4}^L P_{t,\beta}\big(N_{t}(l)\geq a_1\sqrt{t}\big)\\
&\leq\tfrac{1}{2}a_2\,\sum_{t=v(\log L)^4}^L e^{-a\sqrt{t}}.
\end{align}
Since the sum in \eqref{ths} vanishes as $L\to\infty$, the proof is complete.

\subsection{Proof of Proposition \ref{prop1}}\label{pr:prop1}
We recall the definition of the one bead partition function introduced in Section \ref{sec:rep}, equations (\ref{deftu}--\ref{eq:dist11}). Henceforth, we will use the notation $\tilde{Z}_{L,\beta}^{\mathsf{o}}=Z_{L,\beta}^{\mathsf{m},\mathsf{o}} e^{-\beta L}/c_\beta$, so that Proposition \ref{prop1} will be proven once we show that there exist $c_1,c_2>0$ and $\kappa>1/2$ such that
\begin{equation}\label{inb}
\frac{c_1}{L^\kappa}\,e^{-\tilde G(a_{\beta}) \, \sqrt{L}}\leq \tilde{Z}_{L,\beta}^{\mathsf{o}}\leq\frac{c_2}{\sqrt{L}}\,e^{- \tilde G(a_{\beta})\, \sqrt{L}},\quad\text{for $L\in \N$ }.
\end{equation}

We will prove \eqref{inb} subject to Lemma \ref{lem1} below and
Proposition \ref{prop:comppos}. The proof of Lemma \ref{lem1} is given
in Section \ref{pr:lem2} whereas the proof of Proposition
\ref{prop:comppos} is postponed to Section \ref{sec:wulff}. For $K\subset\{1,\dots,L\}$, we set
\begin{equation}\label{form}
\tilde Z^{\mathsf{o}}_{L,\beta}(N\in K)=2\sum_{N\in K}\left(\Gamma(\beta)\right)^N\mathbf{P}_{\beta}(\mathcal{V}^+_{N+1,L-N})\,,
\end{equation}
and similarly we have
 \begin{equation}\label{eq:partfunco}
 \tilde Z^{\mathsf{o}}_{L,\beta}=
 2\sum_{N=1}^{L}\left(\Gamma(\beta)\right)^N\mathbf{P}_{\beta}(\mathcal{V}^+_{N+1,L-N}).
 \end{equation}

\begin{lemma}\label{lem1}
For $\beta>\beta_c$, there exists $a_2>a_1>0$ such that for $L\in\mathbb{N}$,
\begin{equation}\label{desc}
\lim_{L\to \infty}\frac{\tilde Z^{\mathsf{o}}_{L,\beta}(a_1\sqrt{L}\leq N\leq a_2\sqrt{L})}{\tilde Z^{\mathsf{o}}_{L,\beta}}=1.
\end{equation}
\end{lemma}


By using Lemma \ref{lem1}, we note that 
it suffices to prove \eqref{inb} with $\tilde Z^{\mathsf{o}}_{L,\beta}\big(N\in\sqrt{L}\,[a_1,a_2]\big)$ instead of $\tilde Z^{\mathsf{o}}_{L,\beta}$. For the ease of notation, we will rather take $a_2$ a bit larger and consider $\tilde Z^{\mathsf{o}}_{L,\beta}\big(1+N \in\sqrt{L}\,[a_1,a_2]\big)$. In view of \eqref{form}, we write
\begin{equation}\label{thh}
\textstyle\tilde Z^{\mathsf{o}}_{L,\beta}\big(1+N\in\sqrt{L}\,[a_1,a_2]\big)=
2\sum_{N=a_1\sqrt{L}}^{a_2\sqrt{L}}\left(\Gamma(\beta)\right)^{N-1}\mathbf{P}_{\beta}(\mathcal{V}^+_{N,L-N+1}).
\end{equation}
For $n\in\mathbb{N}$, we recall \eqref{defcV} and \eqref{eq:Yn} and we note that $nY_n=A_n$ on the set $\{V_n=0,\,V_i>0\ \forall i\in[1,N-1]\cap\mathbb{N}\}$. Therefore, we set $q_{N,L}:=\frac{L-N+1}{N^2}$ for $N\in \sqrt{L} [a_1,a_2] \cap \N$ and we can write
\begin{equation}\label{eq:vrep}
\mathcal{V}^+_{N,L-N+1}=\{V\colon Y_N=Nq_{N,L},\,V_N=0,\,V_i>0\ \forall i\in[1,N-1]\cap\mathbb{N}\}.
\end{equation}
At this stage, our aim is to bound from above and below the quantities $\mathbf{P}_{\beta}(\mathcal{V}^+_{N,L-N+1})$ for $N\in \sqrt{L}[a_1,a_2]\cap \N$. The upper bound is obvious, i.e.,  
\begin{equation}\label{ptg}
\mathbf{P}_{\beta}(\mathcal{V}^+_{N,L-N+1})\leq\mathbf{P}_{\beta}(Y_N=Nq_{N,L},\,V_N=0),
\end{equation}
while the lower bound is obtained as follows.
Since $q_{N,L}\in\bigl[\frac{1}{2a_2^2},\frac{1}{a_1^2}\bigr]$ when $N\in\sqrt{L}[a_1,a_2]$, we can apply  Proposition \ref{prop:comppos} to claim that, there exists $C,\mu>0$ such that for $L$ large enough,
\begin{equation}\label{ptg2}
\mathbf{P}_{\beta}(\mathcal{V}^+_{N,L-N+1})\geq\tfrac{C}{N^\mu}\mathbf{P}_{\beta}(Y_N=Nq_{N,L},\,V_N=0),\quad N\in\sqrt{L}[a_1,a_2]\cap \N.
\end{equation}
By using again the fact that  $q_{N,L}\in\bigl[\frac{1}{2a_2^2},\frac{1}{a_1^2}\bigr]$ when $N\in\sqrt{L}\,[a_1,a_2]$, we can apply Proposition \ref{lem:impor}, which provides a lower and an upper bound on $\mathbf{P}_{\beta}(Y_N=Nq_{N,L},\,V_N=0)$. By combining these last two bounds with  (\ref{ptg}--\ref{ptg2}) and by setting $\kappa=1+\mu/2$ we can assert that there exists $R_1>R_2>0$ such that for $L$ 
large enough and all $N\in \sqrt{L}\, [a_1,a_2]$ we have that
\begin{align}\label{ptg1}
\tfrac{R_2}{L^\kappa}\, &e^{N \big[-\tilde h_{0}(q_{N,L},0)\,q_{N,L} +\mathfrak{L}_{\Lambda}(\tilde {\bf h}(q_{N,L},0))\big]}\\
\nonumber&\hspace{2cm}\leq \mathbf{P}_{\beta}(\cV^+_{N,L-N+1})
\leq\tfrac{R_1}{L}\,e^{N\big[ -\tilde h_{0}(q_{N,L},0)\,q_{N,L} +\mathfrak{L}_{\Lambda}(\tilde {\bf h}(q_{N,L},0))\big]}.
\end{align}
At this stage, we recall the definition of $\tilde G$ in \eqref{defg} and we set  
\begin{equation}\label{defQ}
Q_{L,\beta}:= \sum_{N=a_1\sqrt{L}}^{a_2\sqrt{L}} e^{\sqrt{L}\, G_{L,N}}
\end{equation} 
with 
\begin{equation}\label{eqqc2}
 G_{L,N}=\tfrac{N }{\sqrt{L}}\,  \big(q_{N,L}\big)^{1/2} \,\tilde G\Big(\tfrac{1}{(q_{N,L})^{1/2}}\Big)
\end{equation}
and we use \eqref{form} and \eqref{ptg1} to claim that there exists $R_3>R_4>0$
(depending on $\beta$ only) such that
for $L$ large enough,
\begin{equation}\label{QZbounds}
\tfrac{R_4}{L^\kappa} \,Q_{L,\beta}\leq\tilde Z^{\mathsf{o}}_{L,\beta}\big(N\in\sqrt{L}\,[a_1,a_2]\big)\leq \tfrac{R_3}{L} \, Q_{L,\beta}.
\end{equation}
We recall that $a\mapsto \tilde G(a)$ is a strictly negative and strictly concave function on $(0,\infty)$ and reaches its unique maximum at $a_{\beta}$, which obviously belongs to $[a_1,a_2]$. 
Since, by Lemma \ref{diffeo}, $a\mapsto \tilde G(a)$ is $\mathcal{C}^1$ on $(0,\infty)$, we can assert that it is Lipschitz on each compact subset of $(0,\infty)$. Moreover, there exists a $C>0$ such that $|q_{N+1,L}-q_{N,L}|\leq C/\sqrt{L}$ for $N\in\sqrt{L}\,[a_1,a_2]$ and
we have that
\begin{equation}\label{encad}
{\textstyle \Big(1-\frac{a_2}{\sqrt{L}}\Big)^{\frac{1}{2}}\leq \tfrac{N}{\sqrt{L}} \, (q_{N,L})^{\frac{1}{2}}\leq \Big(1-\frac{a_1}{\sqrt{L}}\Big)^{\frac{1}{2}}}, \quad N\in \sqrt{L}[a_1,a_2],
\end{equation}
 therefore, we can take the supremum of $G_{L,N}$ on $N\in\bigl[a_1\sqrt{L},a_2\sqrt{L}\,\bigr]\cap\mathbb{N}$ and it comes that
\begin{equation}\label{approx}
\sup\big\{G_{L,N};\,N\in\sqrt{L}\,[a_1,a_2]\cap\mathbb{N}\big\}= \tilde  G(a_{\beta})+O(\tfrac{1}{\sqrt{L}}).
\end{equation}
By putting together \eqref{defQ} and \eqref{approx}  we obtain that there exists $R_5>R_6>0$
such that   for L large enough,
\begin{equation}\label{limp}
R_6\, e^{\tilde G(a_{\beta}) \sqrt{L}}\leq Q_{L,\beta}\leq R_5 \sqrt{L}\, e^{\tilde  G(a_{\beta})\sqrt{L}}.
\end{equation}
At this stage it suffices to combine \eqref{QZbounds} with \eqref{limp} to complete the proof of \eqref{inb} with $\kappa=\mu/2+1$.

\subsection{Proof of Lemmas \ref{lem2} and \ref{lem1}}\label{pr:lem2}
We will only display the proof of Lemma \ref{lem1} because the proof of Lemma \ref{lem2} is obtained in a very similar manner. We recall \eqref{form} and \eqref{eq:partfunco}  and we will first show that there exists $\gamma>0$ and $c>0$ such that
\begin{equation}\label{rtf}
\tilde Z_{L,\beta}^{\mathsf{o}}\geq c\,e^{-\gamma\sqrt{L}},\quad L\in\mathbb{N}.
\end{equation}
Then, we will show that there exist $a_2>a_1>0$ and $c_1,c_2>0$ such that 
\begin{align}\label{rtf2}
\nonumber\tilde Z_{L,\beta}^{\mathsf{o}}(N\geq a_2\sqrt{L})&\leq c_2\,e^{-2\gamma\sqrt{L}},\quad L\in\mathbb{N},\\
\tilde Z_{L,\beta}^{\mathsf{o}}(N\leq a_1\sqrt{L})&\leq c_1\,e^{-2\gamma\sqrt{L}},\quad L\in\mathbb{N}.
\end{align}
Putting together \eqref{rtf} and \eqref{rtf2}, we will immediately obtain \eqref{desc}. To begin with, set $r:=\Big\lfloor\tfrac{L}{1+\lfloor\sqrt{L}\rfloor}\Big\rfloor$, $u:=L-r-(r-1)\lfloor\sqrt{L}\rfloor$ and note that $u\in\{\lfloor \sqrt{L}\rfloor,\dots,2\lfloor \sqrt{L}\rfloor\}$. Then, consider the trajectory $V^*\in\mathcal{V}_{r+1,L-r}^+$ defined as $V_0=V_{r+1}=0$, $V_1=\dots=V_{r-1}=\lfloor\sqrt{L}\rfloor$ and $V_r=u$. One can therefore compute
\begin{equation}\label{dede}
\mathbf{P}_{\beta}(V^*)=\big(\tfrac{1}{c_\beta}\big)^{r+1}e^{-\tfrac{\beta}{2}(2u)}\geq
\big(\tfrac{1}{c_\beta}\big)^{r+1}e^{-2\beta \lfloor\sqrt{L}\rfloor}  ,
\end{equation}
and consequently by restricting the sum in \eqref{form} to $N=r$, by using \eqref{dede} and the inequality $\lfloor\sqrt{L}\rfloor\leq \sqrt{L}$, we obtain 
\begin{equation}
\tilde Z_{L,\beta}^{\mathsf{o}}\geq\tfrac{2}{c_\beta}\left(\tfrac{\Gamma(\beta)}{c_\beta}\right)^r\ e^{-2\beta\sqrt{L}}.
\end{equation}
It remains to note that $r\leq\sqrt{L}$ and to recall that $c_\beta>1$ and that $\Gamma(\beta)<1$ because $\beta>\beta_c$. This is sufficient to obtain \eqref{rtf}.

Proving the first inequality in \eqref{rtf2} is easy because $\Gamma(\beta)<1$ and thus,  we can use \eqref{form} to claim
that there exists a $C>0$ such that 
\begin{equation}
\tilde Z^{\mathsf{o}}_{L,\beta}(N\geq a_2\sqrt L)\leq 2\sum_{N=a_2\sqrt L}^{\infty} \left(\Gamma(\beta)\right)^N\leq C e^{a_2 \log (\Gamma(\beta)) \sqrt{L}}.
\end{equation}
Since $\log (\Gamma(\beta))<0$, it suffices to choose $a_2$ large enough to obtain the first inequality in \eqref{rtf2}.

To prove the last inequality in \eqref{rtf2}, we note that, for $N\leq a_1\sqrt{L}$ and for all $(V_i)_{i=0}^{N+1}\in\mathcal{V}^+_{N+1,L-N}$ we have $\max\{V_j,\,j\in\{1,\dots,N\}\}\geq\frac{L-N}{N}\geq\frac{\sqrt{L}}{a_1}-1$ and therefore, for $L$ large enough we have
\begin{align}
\mathbf{P}_{\beta}(\mathcal{V}^+_{N+1,L-N})&\leq\mathbf{P}_{\beta}\Big(\max\{V_j,j\leq a_1\sqrt{L}\}\geq\tfrac{\sqrt{L}}{2a_1}\Big)\\
&\leq\textstyle\mathbf{P}_{\beta}\Big(\sum_{i=1}^{a_1\sqrt{L}}|U_i|>\tfrac{\sqrt{L}}{2 a_1}\Big),
\end{align}
and since $U_1$ has some finite exponential moments, we can apply a standard Cramer's Theorem to obtain that for $L$ large enough, there exists $g(a_1)>0$ such that $\lim_{a_1\to 0^+}g(a_1)=\infty$ and that $\mathbf{P}_{\beta}(\mathcal{V}^+_{N+1,L-N})\leq e^{-g(a_1)\sqrt{L}}$ for $N\leq a_1\sqrt{L}$. Therefore, by taking $a_1$ small enough we obtain the second inequality in \eqref{rtf2}, which completes the proof of Lemma \ref{lem1}.

\subsection{Proof of Theorem \ref{Prop5} (Horizontal extension)}\label{pp5}
To begin this section, we prove that $\tilde G$ is strictly concave and reaches its maximum at a unique point $a_\beta\in (0,\infty)$. Recall \eqref{defg} and compute its first two derivatives (by using that $\nabla \Llam(\tilde {\bf h}(q,0))=(q,0)$), i.e.,
\begin{align}
\frac{d}{da}\tilde{G}(a)&=\log \Gamma(\beta)+\tfrac{1}{a^2}\tilde
h_0\bigl(\tfrac{1}{a^2},0\bigr)+\Llam(\tilde{\bf h}(\tfrac{1}{a^2},0)),\\
\frac{d^2}{da^2}\tilde{G}(a)&=-\tfrac{2}{a^3}\tilde h_0\bigl(\tfrac{1}{a^2},0\bigr)-\tfrac{4}{a^5}\partial_1\tilde h_0\bigl(\tfrac{1}{a^2},0\bigr).
\end{align}
It suffices to show that $\frac{d^2}{da^2}\tilde{G}(a)<0$ on $(0,\infty)$ and that
$\frac{d}{da}\tilde{G}(a)$ has a zero on $(0,\infty)$.  
Since $\tilde h_0(x,0)=-2 \tilde h_1(x,0)$ (recall Remark \ref{rem:relhzerohuncontinu}), we consider
$R:u\mapsto \int_0^1 x \Ll'((x-\tfrac12) u) dx$ so that $\partial_1
(\Llam) (\tilde {\bf h}(x,0))=R(\tilde h_0(x,0))$. Clearly $R(0)=0$ and $R'(u)=2\int_0^1 x^2 \Ll''(xu)
dx$ because $\Ll$ is even (recall \eqref{defL}). Therefore $R'(u)>0$ when $u\neq 0$ and $R<0$ on $(-\infty,0)$ and $R>0$ on $(0,\infty)$.  Since $R(\tilde h_0(x,0))=x$ for $x\in \R$, we can claim that $\tilde h_0(x,0)>0$ for  $x\in (0,\infty)$ and by differentiating this latter equality we obtain that $\partial_1 \tilde h_0(x,0)=1/ R'(\tilde h_0(x,0))$ which is strictly positive on $(0,\infty)$.
This completes the proof.

Let us start the proof of Theorem \ref{Prop5}. Recall that $i_1$ and $i_2$ are the end-steps of the largest bead $I_{j_{\text{max}}}$, i.e., $I_{j_{\text{max}}}=\{i_1+1,\dots,i_2\}$. For $v>0$, we let
\begin{equation}
\cT_{L,v}:=\bigl\{l\in\Omega_L\colon i_1\leq v(\log L)^4,\,i_2\geq L-v(\log L)^4,\,I_{j_{\text{max}}}=\{i_1+1,\dots,i_2\}\bigr\}.
\end{equation} 
By Theorem \ref{Thm4}, there exists a $v>0$ such that $\lim_{L\to\infty}P_{L,\beta}(\cT_{L,v})=1$. Therefore, the proof will be complete once we show that
\begin{equation}
\lim_{L\to\infty}P_{L,\beta}\biggl(\Bigl\{\Bigl|\frac{N_L(l)}{\sqrt{L}}-a_\beta\Bigr|>\epsilon\Bigr\}\cap \cT_{L,v}\biggr)=0.
\end{equation}
Let $N_{I_{j_{\text{max}}}}$ denote the number of horizontal steps made by the random walk in its largest bead. Pick $\gep'<\gep$ and since the first step and the last step of the largest bead are at distance less than $v (\log L)^4$ from $0$ and $L$, respectively, we can write that for $L$ large enough
\begin{align}\label{eq:vmax}
\nonumber P_{L,\beta}\biggl(\Bigl\{\Bigl|\frac{N_L(l)}{\sqrt{L}}-a_\beta\Bigr|>\epsilon\Bigr\}\cap \cT_{L,v}\biggr)&\\
\nonumber \leq\sum_{\substack{1\leq i_1\leq v(\log L)^4\\  L-v(\log L)^4\leq i_2\leq L}}&
P_{L,\beta}\biggl(\Bigl|\tfrac{N_{I_{j_{\text{max}}}}}{\sqrt{i_2-i_1}}-a_\beta\Bigr|>\epsilon',\,I_{j_{\text{max}}}=\{i_1+1,\dots,i_2\}\biggr)\\
&\leq 4 \sum_{\substack{1\leq i_1\leq v(\log L)^4\\ L-v(\log L)^4\leq i_2\leq L}} \frac{Z^{\mathsf{o}}_{i_2-i_1,\beta}\biggl(\Bigl|\tfrac{N}{\sqrt{i_2-i_1}}-a_\beta\Bigr|>\epsilon'\biggr)}{Z^{\mathsf{o}}_{i_2-i_1,\beta}},
\end{align}
where the coefficient $4$ in front of the r.h.s. in \eqref{eq:vmax} comes from a direct application of Lemma \ref{lem:decoZ}.
Now, we focus on the numerator of the r.h.s. in \eqref{eq:vmax}  and since $\tilde{G}$ is strictly concave and reaches its maximum at $a_\beta$ we can claim that the maximum of $\tilde G$ on $(0,a_\beta-\gep']\cup [a_\beta+\gep',\infty)$ is given by  $T(\gep')=\max\{\tilde G(a_\beta-\gep'),\tilde  G(a_\beta+\gep')\}$.  
We  proceed as in \eqref{thh}-\eqref{approx} and we get that there exits a $C_1>0$ such that
\begin{equation}
Z^{\mathsf{o}}_{i_2-i_1,\beta}\biggl(\Bigl|\tfrac{N}{\sqrt{i_2-i_1}}-a_\beta\Bigr|>\epsilon'\biggr)\leq\tfrac{C_1}{\sqrt{i_2-i_1}}\,  e^{\beta (i_2-i_1)} \,e^{T(\epsilon')\sqrt{i_2-i_1}}.
\end{equation}
We apply Proposition \ref{prop1} and  the denominator can be bounded from below as
\begin{equation}
Z^{\mathsf{o}}_{i_2-i_1,\beta}\geq\tfrac{C_2}{(i_2-i_1)^\kappa}\,  e^{\beta (i_2-i_1)}\,e^{\tilde{G}(a_\beta)\sqrt{i_2-i_1}},
\end{equation}
for some constants $\kappa>1/2$ and $C_2>0$. Since $L-2v(\log L)^4\leq i_2-i_1\leq L$, we can state that, for $L$ large enough, \eqref{eq:vmax} becomes
\begin{equation}
P_{L,\beta}\biggl(\Bigl\{\Bigl|\frac{N_L(l)}{\sqrt{L}}-a_\beta\Bigr|>\epsilon\Bigr\}\cap\,  \cT_{L,v}\biggr)\leq C_3 \,L^{\kappa-\frac{1}{2}} \,\log(L)^8\,e^{-(\tilde{G}(a_\beta)-T(\gep'))\sqrt{L-2v(\log L)^4}}.
\end{equation}
Since $\tilde{G}(a_\beta)>T(\gep')$, the right hand side vanishes as $L\to\infty$ and this completes the proof.

\subsection{Proof of Theorem \ref{Convenv} (Wulff shape)}\label{pr:Thm7}
Before displaying the proof of Theorem \ref{Convenv}, we provide a rigorous definition of $\gamma^*_{\beta}$
and we associate with each trajectory $l\in \Omega_L$ the process $M_l$ that links the middle of each stretch consecutively.

The Wulff shape $\gamma^*_{\beta}$ can be defined\footnote{the set on
  the right hand side of \eqref{defga} is not empty since it contains
  the hat function $\gamma(t)=\gamma(1-t)=\frac{2t}{a_\beta}$ for
  $0\le t\le \frac12$} as 
\be{defga}
\gamma^*_{\beta}=\text{argmin}\big\{J(\gamma),\, \gamma \in \cB_{[0,1]},\, {\textstyle \int_0^1\gamma(t) dt}=\tfrac{1}{a_\beta^2}, \gamma(0)=\gamma(1)=0 \big\},
\ee
where  $\cB_{[0,1]}$ is the set containing the cadlag  real functions defined on $[0,1]$, where $J: \cB_{[0,1]}\to [0,\infty]$ is defined as
\be{defJ}
J(\gamma)=\begin{dcases*}
	\int_0^1 \Ll^*(\gamma'(t)) dt\quad \text{if} \quad \gamma \in \cA\cC,\\
   +\infty \quad \text{otherwise},
  \end{dcases*}
\ee
where $\cA\cC$ is the set of absolutely continuous functions and where $\mathfrak{L}^*$ is the Legendre transform of $\mathfrak{L}$, i.e.,
\be{defLam}
\Ll^*(u)=\sup\big\{h u-\Ll(h), h\in (-\tfrac{\beta}{2},\tfrac{\beta}{2})\big\}, \quad u\in \R.
\ee
Using the duality between $\Ll$ and $\Ll^*$ we easily obtain the
formula \eqref{defgamma} given in the introduction,
which easily implies (recall \ref{defg}) that $
\tilde{G}(a_{\beta})=a_\beta (\log \Gamma(\beta)-J(\gamma^*_{\beta}))$.
Finally, we note that one can prove without further difficulty that
\be{argm}
\{-\gamma^*_{\beta}, \gamma^*_{\beta}\}=\text{argmin}\big\{J(\gamma),\, \gamma \in \cB_{[0,1]},\, A(\gamma)=\tfrac{1}{a_\beta^2}, \gamma(0)=\gamma(1)=0 \big\},
\ee
where $A(\gamma):=\int_0^1|\gamma(s)| ds$  is the geometric area enclosed between the graph of $\gamma$ and 
the $x$-axis.

We recall the definition of $\cE_{l}^+$ and $\cE_l^-$ in \eqref{trek} and  
we also associate with each $l \in \cL_{N,L}$ the path $M_l=\big(M_{l,i}\big)_{i=0}^{N+1}$ that links the middles of each stretch consecutively and is defined as  $M_{l,0}=0$ 
\be{droitmi}
M_{l,i}=l_1+\dots+l_{i-1}+\frac{l_i}{2},\quad i\in \{1,\dots,N\},
\ee
and $M_{l,N+1}=l_1+\dots+l_N$. We recall that the $T_N$ transformation, defined in Section \ref{sec:rep}, associates with each $l\in \cL_{N,L}$ the path $V_l=(T_N)^{-1}(l)$
such that $V_{l,0}=0$,  $V_{l,i}=(-1)^{i-1} l_i$ for all $i\in\{1,\dots,N\}$ and $V_{l,N+1}=0$.
As a consequence, $\cE_l^+=M_l+\frac{|V_l |}{2}$ and $\cE_l^-=M_l-\frac{|V_l|}{2}$, i.e.,
\begin{align}\label{caract}
\nonumber \cE^+_{l,i}&=M_{l,i}+\frac{|V_{l,i}|}{2},\quad i\in \{0,\dots,N+1\},\\
\cE^-_{l,i}&=M_{l,i}-\frac{|V_{l,i}|}{2},\quad i\in \{0,\dots,N+1\},
\end{align}
and the path $(M_{l,i})_{i=0}^{N+1}$ can be rewritten with the increments $(U_i)_{i=1}^{N+1}$ of the $V_l$ random walk as
\be{defM}
M_{l,i}=\sum_{j=1}^{i} (-1)^{j+1} \frac{U_j}{2},\quad i\in \{1,\dots,N\}.\\
\ee
Similarly to what we did to define $\tilde \cE^+_l$ and $\tilde \cE^-_l$ in \eqref{defti}, we let 
$\tilde M_l$ and $\tilde{V}_l$ be  the time-space rescaled cadlag process associated to $M_l$ and $V_l$.
%
%


\noindent {\it Proof of Theorem \ref{Convenv}.}
Equations \eqref{caract} that allows to express $\cE^+_l$ and $\cE^-_l$ with the help of 
the two processes $V_l$ and $M_l$ can be translated in terms of the time-space rescaled processes
as $\tilde \cE_l^+=\tilde M_l+\frac{|\tilde V_l |}{2}$ and $\tilde \cE_l^-=\tilde M_l-\frac{|\tilde V_l|}{2}$. Therefore, Theorem \ref{Convenv} is a straightforward consequence of the two following Lemmas. 

\bl{ConvProf} For $\beta>\beta_c$ and $\gep>0$,
\begin{equation}\label{le1}
\lim_{L\to \infty} P_{L,\beta}\Big(  \big\| |\widetilde{V}_{l}|-\gamma^*_{\beta}\big\|_{\infty} >\gep \Big)=0.
\end{equation}
\el

\bl{tret} For $\beta>0$ and $\gep>0$,
\be{le2}
\lim_{L\to \infty} P_{L,\beta}\Big(  \big\|\widetilde{M}_{l}\big\|_{\infty} >\gep \Big)=0.
\ee
\el

\noindent \emph{Proof of Lemma \ref{ConvProf}.}
For  conciseness we set $\cU_{L,\gep}=\big\{l \in \Omega_L\colon\; \big\| |\widetilde{V}_{l}|-\gamma^*_{\beta}\big\|_{\infty} >\gep\big\}$. Thanks to Theorem \ref{Prop5}, Lemma \ref{ConvProf} will be proven once we show that there exists an $\eta>0$ such that 
\be{ddtt}
\lim_{L\to \infty} P_{L,\beta}\Big(\cU_{L,\gep} \cap \Big\{ \big|\tfrac{N_L(l)}{\sqrt{L}}- a_\beta\big|\leq \eta\Big\} \Big)=0.
\ee
We decompose  the left hand side in \eqref{ddtt} with respect to the value taken by $N_L(l)$, i.e.,
\be{ddtt1}
P_{L,\beta}\Big( \cU_{L,\gep} \cap \Big\{ \big|\tfrac{N_L(l)}{\sqrt{L}}- a_\beta\big|\leq \eta\Big\} \Big)=
 \sum_{N\in I_{\eta,L}}  P_{L,\beta}\Big( \cU_{L,\gep} \cap \{ N_L(l)=N\} \Big),
\ee
where $I_{\eta,L}=\big\{(a_\beta-\eta) \sqrt L, \dots, (a_\beta+\eta) \sqrt L\big\}$.
By recalling Section \ref{sec:rep}, the probability in the r.h.s. of \eqref{ddtt1} can be rewritten, with the help of the random walk representation, as 
\begin{align}\label{dd2}
P_{L,\beta}\Big( \cU_{L,\gep}& \cap \{ N_L(l)=N\} \Big)\\
\nonumber &=\frac{ (\Gamma(\beta))^N}{\tilde Z_{L,\beta}}  \,
\bP_\beta\Big( \big\| |\widetilde{V}_{N+1}|-\gamma^*_{\beta}\big\|_{\infty} >\gep, \widetilde{V}_{N+1}(1)=0, A\big(\widetilde{V}_{N+1}\big)=\tfrac{L-N}{(N+1)^2}\Big),
\end{align}
where $(V_i)_{i=0}^{N+1}$ is a random walk of law $\bP_\beta$ and $\tilde V_{N+1}$ is the time-space rescaled process associated with $(V_i)_{i=0}^{N+1}$, i.e.,
$$\tilde V_{N+1}(t)= \frac{1}{N+1}\,  V_{{\lfloor t\,(N+1)\rfloor}},\quad t\in [0,1],$$
and where  $\tilde{Z}_{L,\beta}=Z_{L,\beta} e^{-\beta L}/c_\beta$.
Note that there exists a function $g: \R^+\to \R^+$ such that $\lim_{\eta\to 0} g(\eta)=0$ and such that  for $N\in I_{\eta,L}$ the probability in the r.h.s. of \eqref{dd2} is bounded from above by $\bP_\beta(\widetilde V_N \in \c
H_{\gep,\eta})$, where
\be{defH}
\cH_{\gep,\eta}=\big\{\gamma \in \cB_{[0,1]} \colon  A(\gamma)\geq \tfrac{1}{a_\beta^2}-g(\eta), \gamma(0)=\gamma(1)=0, \big\| |\gamma|-\gamma^*_{\beta}\big\|_\infty\geq 
\gep\big\}.
\ee
Thus, we need to identify the exponential growth rate of  $\bP_\beta(\widetilde V_N \in \cH_{\gep,\eta})$. To that aim, we apply the Mogulskii Theorem (see \cite[Theorem 5.1.2]{DZ}) which ensures that  $(\widetilde V_N)_{N\in \N}$ follows a large deviation principle on the  set $\cB([0,1])$ endowed with  the supremum norm $\|\cdot\|_\infty$ and  with the good rate function $J$  defined in \eqref{defJ}.  Since 
$\cH_{\gep,\eta}$ is a closed subset of $\big(\cB_{[0,1]},\| \cdot \|_\infty\big)$ we can assert that 
\be{LDT}
\limsup_{n\to \infty} \tfrac{1}{N} \log  \bP_\beta(\widetilde V_N \in \cH_{\gep,\eta})\leq -\inf\{J(\gamma), \gamma\in \cH_{\gep,\eta}\}.
\ee 
We pick $M>\inf\{J(\gamma), \gamma\in \cH_{\gep,1}\}$ and set $\cH^M_{\gep,\eta}= \{\gamma\in \cH_{\gep,\eta}\colon\, J(\gamma)\leq M\}$ such that 
the inequality \eqref{LDT} becomes
\be{LDT2}
\limsup_{n\to \infty} \tfrac{1}{N} \log  \bP_\beta(\widetilde V_N\in \cH_{\gep,\eta})\leq -\inf\{J(\gamma), \gamma\in \cH^M_{\gep,\eta}\}.
\ee
At this stage, it remains to show that there exists $\alpha>0$ and $\eta_0>0$ such that 
for all $\eta\in(0,\eta_0]$, 
\be{inegali}
\inf\{J(\gamma), \gamma\in \cH^M_{\gep,\eta}\}-\alpha\geq \inf\{J(\gamma), \gamma\in \cH_{0,0}\}=J(\gamma^*_{\beta}).
\ee
Assume that \eqref{inegali} fails to be true, then, there exists a strictly positive sequence $(z_n)_{n\in \N}$ that tends to $0$ as $n\to \infty$ such that for all $n\in \N$ there exists a $\gamma_n\in \cH^M_{\gep,z_n}$ satisfying
$J(\gamma_n)\leq J(\gamma^*_{\beta})+1/n$. Since $J$ is a good rate function, we can assert that $\cH^M_{\gep,1}$ is a compact set of $(\cB_{[0,1]} , \| \cdot \|_\infty )$ and consequently $ \gamma_n$ is converging by subsequence towards some 
$\gamma_\infty\in \cH^M_{\gep,1}$. Since $A$ and $J$ are continuous and lower semi-continuous on $(\cB_{[0,1]},\|\cdot\|_\infty)$, respectively, it comes that $\gamma_\infty\in \cH^M_{\gep,0}$ and $J(\gamma_\infty)\leq J(\gamma^*_{\beta})$, which leads to a contradiction because $-\gamma^*_{\beta}$ and $\gamma^*_{\beta}$ are the unique maximizer of $J$ on $\cH_{0,0}$ and $\gamma_\infty\notin \{-\gamma^*_{\beta}, \gamma^*_{\beta}\}$. At this stage, we go back to \eqref{dd2} and we can write, for $\eta\in (0,1]$
\be{dd3}
P_{L,\beta}\Big( \cU_{L,\gep} \cap \{ |N_L(l)- a_\beta|\leq \eta\} \Big)\leq 
\frac{2\eta}{\tilde Z_{L,\beta}} \sqrt{L} \, (\Gamma(\beta))^{(a_\beta-\eta)\sqrt{L}} \,
\bP_\beta\big(\widetilde{V}_{N+1}\in \cH_{\gep, \eta}\big).
\ee
Thus, by \eqref{LDT2} and \eqref{dd3}  we can assert that for all $\eta\in (0,\eta_0]$ and for $L$ large enough
\begin{align}\label{dd4}
\nonumber P_{L,\beta}\Big( \cU_{L,\gep} \cap \{ |N_L(l)- a_\beta|\leq \eta\} \Big)&\leq 
\frac{2\eta\sqrt{L}}{\tilde Z_{L,\beta}} \, (\Gamma(\beta))^{(a_\beta-\eta)\sqrt{L}} \,
e^{-(a_\beta-\eta)\sqrt{L} (J(\gamma^*_{\beta})+\alpha)},\\
&\leq 
\frac{ 2\eta\sqrt{L}}{\tilde Z_{L,\beta}} e^{\sqrt{L} (a_\beta-\eta) (\log \Gamma(\beta)-J(\gamma^*_{\beta})-\alpha)}.
\end{align}

Recall the equality $\tilde{G}(a_{\beta})=a_\beta (\log \Gamma(\beta)-J(\gamma^*_{\beta}))$
and recall that for $\beta>\beta_c$, we have proved in \eqref{inb} that there exists $c_1>0$ and  $\kappa>0$ such that
for $L$ large enough,
\be{Zbound}
\tilde Z_{L,\beta}\geq  \tfrac{c_1}{L^\kappa}\, e^{\sqrt{L} \tilde G(a_{\beta})}.
\ee
Thus, we can use \eqref{dd4} to claim  that by choosing $\eta$ small
enough and $L$ large enough we have for a constant $c_2>0$,
\begin{align}\label{dd8}
P_{L,\beta}\Big( \cU_{L,\gep} \cap \{ |N_L(l)- a_\beta|\leq \eta\} \Big)&\leq 
\tfrac{1}{c_2} L^{\frac12+\kappa} e^{-\tfrac{\alpha}{2} a_\beta\sqrt{L}}.
\end{align}
which completes the proof of Lemma \ref{ConvProf}.\\

\noindent \emph{Proof of lemma \ref{tret}.}
Lemma \ref{tret} will be proven once we show that for all $\gep>0$,
\be{ddz}
\lim_{L\to \infty} P_{L,\beta}\Big(\tfrac{1}{1+N_L(l)} \max_{i\leq 1+N_L(l)}  |M_{l,i}| \geq \gep \Big)=0.
\ee
Proving \eqref{ddz} requires to control, under $P_{L,\beta}$, the probability that,  the gap between 
the modulus of the algebraic area ($ N_L(l)\, |Y_{l}|:=\,|\sum_{i=1}^{N_L(l)} V_{l,i}|$) and the geometric area ($\sum_{i=1}^{N_L(l)} |V_{l,i}|)$ of  the random walk trajectory $V_l=(T_{N_L(l)})^{-1}(l)$ associated 
with $l\in \Omega_L$ does not exceed $\log(L)^4$. This is the object of Lemma \ref{refine31} below.
\bl{refine31}
For $\beta>\beta_c$  there exists a $c>0$ such that 
\be{refine0}
\lim_{L\to \infty} \, P_{L,\beta}\big( N_L(l)\, |Y_l| \notin [L-N_L(l)-c(\log L)^4,L-N_L(l)] \big)=0.
\ee
\el

\begin{proof}

By Theorem \ref{Thm4} there exists a $c>0$ such that 
\be{ricol}
\lim_{L\to \infty}   P_{L,\beta}\big[ \valabs{I_{j_\text{max}}}\leq L-c (\log L)^4 \big]=0.
\ee
Note that for $l \in \Omega_L$, we have $\sum_{i=1}^{N_L(l)} |V_{l,i}|=\sum_{i=1}^{N_L(l)} |l_i|= L-N_L(l)$ and that, with the definition of $j_{\text{max}}$ and $x_{j_{\text{max}}}$ in \eqref{ut} and \eqref{ut2} we have also
\begin{align}\label{aira}
& \sum_{i=1}^{N_L(l)}  |V_{l,i}|-2 \sum_{i\notin \cO_l}   |V_{l,i}|        \leq \big|\sum_{i=1}^{N_L(l)} V_{l,i}\big| 
\leq \sum_{i=1}^{N_L(l)} |V_{l,i}|,
\end{align}
where $\cO_l=\{x_{j_{\text{max}}-1}+1,\dots,x_{j_{\text{max}}}\}$ gathers the indexes of those stretches 
in $l=(l_1,\dots,l_{N_L(l)})$ that  belong to the largest bead described by $l$.
Moreover, we note that $l\in \{ \valabs{I_{j_\text{max}}}\geq L-c (\log L)^4\}$ yields 
\be{aira2}
\sum_{i\notin \cO_l } |V_{l,i}| =\sum_{i\notin \cO_l } |l_i| \leq c (\log L)^4.
\ee
At this stage, we recall that $N_L(l) \, Y_{l}=\sum_{i=1}^{N_L(l)} V_{l,i}$ and we  use \eqref{aira} and \eqref{aira2} to assert that $l\in \{ \valabs{I_{j_\text{max}}}\geq L-c (\log L)^4\}$ implies $  N_L(l) \, | Y_{l} | \in [L-N_L(l)-2c(\log L)^4,L-N_L(l)]$.  It remains to use \eqref{ricol} to complete the proof of Lemma \ref{refine31}.

\epr

We resume the proof of Lemma \ref{tret}. We set $K_{L,\gep}=\{\tfrac{1}{1+N_L(l)} \max_{i\leq 1+N_L(l)}  |M_{l,i}| \geq \gep\}$ for $\gep>0$  and we set for $\eta>0$ 
\begin{align}
 R_{L,\eta}= \{ |\tfrac{N_L(l)}{\sqrt{L}}- a_\beta|\leq \eta\}\cap \{N_L(l)\, | Y_{l} | \in [L-N_L(l)-c(\log L)^4,L-N_L(l)] \}
\end{align}
Thanks to Theorem \ref{Prop5} and Lemma \ref{refine31}, it suffices to show that there exists $\eta>0$ such that 
for all $\gep>0$,
\be{dd}
\lim_{L\to \infty} P_{L,\beta}\big(K_{L,\gep}\cap R_{L,\eta}   \big)=0.
\ee
We decompose  the left hand side in \eqref{dd} with respect to  the value taken by $N_L(l)$ and $Y_{l}$, i.e.,
\begin{align}\label{dd1}
P_{L,\beta}\big( K_{L,\gep} \cap R_{L,\eta}\big)=
\sum_{N\in I_{\eta,L}} \sum_{q\in F_{L,N}}& \Big[P_{L,\beta}\big( K_{L,\gep} \cap \{ N_L(l)=N\}  \cap \{   Y_{l}=q\, (N+1)\} \big)\\
\nonumber &\,+  P_{L,\beta}\big( K_{L,\gep} \cap \{ N_L(l)=N\}  \cap \{ Y_{l}=-q\, (N+1)\} \big)\Big],
\end{align}
where 
\begin{align*}
I_{\eta,L}&=\big\{(a_\beta-\eta) \sqrt L, \dots, (a_\beta+\eta) \sqrt L\big\},\\
F_{L,N}&=\frac{1}{N (N+1)} \{L-N-c (\log L)^4,\dots,L-N\}.
\end{align*}
We recall the definition of $A_N$ below \eqref{eq:funch} and of $Y_N$ in \eqref{eq:Yn}. With the random walk representation we obtain, for 
$N\in I_{\eta,L}$ and $q\in F_{L,N}$, that 
\begin{align}\label{dd5}
\nonumber  P_{L,\beta}&\big( K_{L,\gep}\, \cap\, \{ N_L(l)=N\}  \cap \{  Y_l=q (1+N)\}\big)\\
\nonumber & = \frac{(\Gamma(\beta))^N} {\tilde Z_{L,\beta}}   \bP_\beta\big(A_N=L-N, Y_{N+1}=q (N+1), \tfrac{1}{1+N} \max_{i\leq 1+N}  |M_{N+1,i}| \geq \gep,V_{N+1}=0\big) \\
 & \leq \frac{(\Gamma(\beta))^N} {\tilde Z_{L,\beta}}   \bP_\beta\big(  Y_{N+1}=q (N+1), V_{N+1}=0\big) D_{N+1,q}
 \end{align}
where  $\tilde{Z}_{L,\beta}=Z_{L,\beta}\,  e^{-\beta L} /c_\beta$, where 
$(M_{N+1,i})_{i=0}^{N+1}$ is defined with the increments   $(U_i)_{i=1}^{N+1}$ of the $V$ random walk (recall \eqref{defM}) as
$M_{N+1,i}=\sum_{j=1}^{i} (-1)^{i+1} \frac{U_i}{2}$ for $i=1,\dots,N+1$, and where
\begin{align}\label{dd6}
D_{N,q}=\bP_\beta\big(\tfrac{1}{N} \max_{i\leq N}  |M_{N,i}| \geq \gep\ |\   Y_N= q N, V_{N}=0\big).
\end{align}
By picking $\eta=a_\beta/2$ we can easily check that there exists $[q_1,q_2]\subset (0,\infty)$ such that 
for all $N\in I_{\eta,L}$ we have 
$F_{N,L}\subset [q_1,q_2]$. We recall \eqref{revt} and  we tilt $\bP_\beta$ into $\bP_{N,{\bf h}_N^{q}}$ so that we can use Proposition \ref{lem:bHtilde} and claim  that there exists a $c>0$ such that for $L$ large enough,  we have
\begin{align}\label{dd66}
D_{N,q}\leq\frac{\bP_{N,{\bf h}_N^{q}}\big(\tfrac{1}{N} \max_{i\leq N}  |M_{N,i}| \geq \gep\big)}{\bP_{N,{\bf h}_N^{q}}\big( Y_N=q N, V_{N+1}=0\big)}\leq c N^2  \bP_{N,{\bf h}_N^{q}}\Big( \max_{i\leq N}  |M_{N,i}| \geq \gep N\Big).
\end{align}
At this stage, we use \eqref{dd1}, \eqref{dd5}, \eqref{dd66} and the inequalities $\Gamma(\beta)<1$ and \eqref{Zbound} to assert that  the proof of Lemma \ref{tret} will be complete once we show that for $[q_1,q_2]\in (0,\infty)$ and
$\gep>0$ there exists a $\theta>0$ such that for $N$ large enough we have
\be{confi}
\sup_{q\in [q_1,q_2]} \bP_{N,{\bf h}_N^q}\Big(\max _{i\leq N} |M_{N,i}| \geq \gep N\Big)\leq e^{-\theta N}.
\ee
We recall that, for $1\leq j\leq N$, we have  $\bE_{N,{\bf h}_N^{q}}(U_j)=\Ll'(h_{N}^j)$ with $h_N^j= (1-\tfrac{j}{N}) h_{N,0}^q+h_{N,1}^q$. 
As a consequence, and because of Lemma 
\ref{exist}, we can assert that, for $N$ large enough and uniformly in $q\in [q_1,q_2]$, all $h_N^i$ belong to some compact set 
$K\subset (-\frac{\beta}{2},\frac{\beta}{2})$.  Therefore,  we can show that there exists $c_1>0$ and $M_1>0$
 such that for $N$ large enough 
 \be{confit2}
\sup_{q\in [q_1,q_2]} \sup_{1\leq i\leq N} \bE_{N, {\bf h}_N^q}(e^{c_1 |U_i|})\leq M_1,
\ee
which is sufficient  to deduce, still for $N$ large enough, that there exists $c_2>0$ and $\delta_0>0$  such that
\be{consst2}
\sup_{q\in [q_1,q_2]} \sup_{1\leq i\leq N}  \sup_{\delta \in[-\delta_0,\delta_0]}  \bE_{N,{\bf h}_{N}^q}(e^{ \delta (U_i-L'(h_{N,j}) }))\leq e^{c_2\delta^2}.
\ee
Then , we set   
\be{defMH}
\widehat M_{N,i}=M_{N,i}-\bE_{N,{\bf h}_N^q}(M_{N,i})=\frac12 \sum_{j=1}^{i} (-1)^{j+1} (U_j-\Ll'(h_{N,j})),\quad i=1,\dots,N,
\ee
and since, under the law $\bP_{N,{\bf h}_N^q}$, the increments $(U_i)_{i=0}^{N}$  are independent, we deduce from 
\eqref{consst2} that, for $N$ large enough, there exists $c_3>0$ and $\delta_0>0$ such that 
\be{consst3}
\sup_{q\in [q_1,q_2]} \sup_{1\leq i\leq N} \sup_{\delta
  \in[-\delta_0,\delta_0]}  \bE_{N,{\bf h}_{N}^q}(e^{ \delta \widehat
  M_{N,i} })\leq e^{c_3\delta^2\, N}.
\ee
The inequality in \eqref{consst3} is sufficient to derive \eqref{confi} with random variables $(\widehat M_{N,i})_{i=1}^N$
instead of $(M_{N,i})_{i=1}^N$. Then, we recover \eqref{confi} by showing  that $\bE_{N,{\bf h}_N^q}(M_{N,i})$ is bounded by some constant  uniformly in $q\in [q_1,q_2]$, $N\geq 2$ and $i\in \{1,\dots,N\}$. The latter boundedness is obtained by writing, 
for all $1\leq i\leq N$ that,
\begin{align}\label{computbis}
2\valabs{\bE_{N,{\bf h}_N^q}\big[  M_{N,i}\big]} &= \valabs{\sum_{j=1}^i (-1)^j
  \Ll'(h^j_N)}\notag\\
& \le ||\Ll'||_{\infty,K} + \valabs{\sum_{j=1}^{\lfloor
    \frac{i}{2}\rfloor} \Ll'(h_N^{2j-1})-\Ll'(h_N^{2j})}\notag\\
&\le ||\Ll'||_{\infty,K} +C ||\Ll''||_{\infty,K}\le C_3\,,
\end{align}
with $\norme{f}_{\infty,K}=\sup_{x\in K} \valabs{f(x)}$ being the sup
norm on the compact $K$.

\section{Decay rate of large area probability}

\subsection{Proof of Proposition \ref{convunif} (Decay rate of large area probability)}\label{sec:mgf}
We will display here the proof of Proposition \ref{convunif} subject to Lemma \ref{convu}, Corollary \ref{graconv} and Lemmas \ref{diffeo}, \ref{exist}  that are stated and proven below.


In what follows we use the notations $\|{\bf x}\|=\max\{|x_1|,|x_2|\}$ and ${\bf x}\cdot {\bf y}=x_1y_1+x_2y_2$
and ${\text d}({\bf x},F)=\inf_{{\bf y}\in F} \| {\bf x}-{\bf y}\|$ for ${\bf x}=(x_1,x_2)\in \R^2$, ${\bf y}=(y_1,y_2)\in \R^2$
and $F\subset \R^2$. We also denote by $\partial F$ the boundary of $F\subset \R^2$.
 
\begin{lemma}\label{convu}
For all $(j_1,j_2)\in (\mathbb{N}\cup \{0\})^2$ and all compact and convex subsets $K$ in $\mathcal{D}$, there exists $c>0$ such that
\begin{equation}
\sup_{{\bf h} \in K}\,\bigl
|\partial^{(j_1,j_2)}\bigl[\ullamn\bigr]({\bf
  h})-\partial^{(j_1,j_2)}\Llam({\bf h})\bigr |\leq\tfrac{c}{n},\quad n\in \N.
\end{equation}
\end{lemma}
\begin{proof}
For all $(j_1,j_2)\in\mathbb{N}^2$, we first differentiate inside the integral
\begin{equation}
\partial^{(j_1,j_2)}\Llam({\bf h})=\int_0^1 \partial^{(j_1,j_2)}_{h_0,h_1}\Ll(xh_0+h_1)dx.
\end{equation}
Then, by using the error estimate for the Riemann sum of $x\mapsto\partial^{(j_1,j_2)}_{h_0,h_1}\Ll(xh_0+h_1)$, we obtain the result.
\end{proof}

By applying Lemma \ref{convu} for $(j_1,j_2)=(0,1)$ and $(j_1,j_2)=(1,0)$, we immediately obtain

\begin{corollary}\label{graconv}
For all compact and convex subsets $K$ in $\mathcal{D}$, there exist a $c>0$ such that
\begin{equation}
\sup_{{\bf h}\in K}\,\bigl\|\nabla\bigl[\ullamn\bigr]({\bf h})-\nabla \Llam({\bf h})\bigr\|\leq\tfrac{c}{n},\quad n\in \N.
\end{equation}
\end{corollary}
For $\eta>0$, we let $K_\eta$ be the compact and convex subset of $\mathcal{D}$ defined as
\begin{equation}
K_\eta:=\Bigl\{{\bf h}=(h_0,h_1)\in\mathbb{R}^2\colon h_1\in\bigl[-\tfrac{\beta}{2}+\eta,\tfrac{\beta}{2}-\eta\bigr],\ h_0+h_1\in\bigl[-\tfrac{\beta}{2}+\eta,\tfrac{\beta}{2}-\eta\bigr]\Bigr\}.
\end{equation}

\begin{lemma}\label{diffeo}
The function $\nabla \Llam\colon\mathcal{D}\mapsto\mathbb{R}^2$ defined as
\begin{align}
\nonumber \nabla \Llam({\bf h})&=(\partial_{h_0}\Llam,\partial_{h_1}\Llam)({\bf h})\\
&=\Bigl(\int_0^1x\Ll'(xh_0+h_1)dx,\ \int_0^1\Ll'(xh_0+h_1)dx\Bigr).
\end{align}
is a $\mathcal{C}^1$ diffeomorphism. Moreover, for all $M>0$ there exists a $\eta>0$ such that $\|\nabla \Llam({\bf h})\|>M$ for ${\bf h}\in\mathcal{D}\setminus K_\eta$. 
\end{lemma}

\bpr
The fact that $h\mapsto \Ll'(h)$ is $\cC^1$ and that $\Ll''(h)$ is strictly positive on $(-\tfrac{\beta}{2},\tfrac{\beta}{2})$ ensures that $\nabla \Llam$ is $\cC^1$  and that
its Jacobian determinant that takes value
\begin{equation}
{\textstyle J_{\bf h}\nabla \Llam= \int_0^1x^2\Ll''(xh_0+h_1)dx \int_0^1 \Ll''(xh_0+h_1)dx-[\int_0^1x \Ll''(xh_0+h_1)dx]^2}
\end{equation}
is, by Cauchy Schwartz inequality, strictly positive. Thus, the proof that $\nabla \Llam$ is a $C^1$ diffeomorphism from $\cD$ to $\R^2$ will be complete once we show that $\nabla \Llam$ is a bijection from $\cD$ to $\R^2$.

At this stage we note that for each ${\bf y}\in \R^2$ the function 

\begin{equation}\label{tfonc}
T_{\bf y}: {\bf h} \to \Llam({\bf h})-{\bf y}\cdot{\bf h }
\end{equation}

is strictly convex and tends to $\infty$ as ${\text d}({\bf h},\partial \cD)\to 0$. Therefore, $T_{\bf y}$ admits a unique minimum on $\cD$ at 
$\tilde {\bf h}({\bf y})$ that is also the unique solution of $\nabla \Llam({\bf h})={\bf y}$. Thus, $\nabla \Llam$ is a bijection from $\cD$ to $\R^2$. 

We complete the proof of this Lemma by assuming that there exists an $M_0>0$ and a sequence $({\bf h}_n)_{n=0}^\infty$ in $\cD$
so that ${\text d}({\bf h}_n, \partial \cD)\to 0$ as $n\to \infty$ and $\big\| \nabla \Llam({\bf h}_n)\big\|\leq M_0 $. Then, set
${\bf y}_n=\nabla \Llam({\bf h}_n)$ and recall that $ {\bf h}_n$ is the minimum of $T_{{\bf y}_n}$ for all $n\in \N$. However,
$T_{{\bf y}_n}(0,0)=0$ and consequently $T_{{\bf y}_n}({{\bf h}_n})\leq 0$ for all $n\in \N$ and then
$\Llam({{\bf h}_n})\leq {{\bf y}_n}\cdot {{\bf h}_n}$ which brings a contradiction because $\lim_{n\to \infty} \Llam({{\bf h}_n})=\infty$ (since ${\text d}({\bf h}_n, \partial \cD)\to 0$) whereas ${{\bf y}_n}\cdot {{\bf h}_n}$ is smaller than $M_0$ times the diameter of $\cD$.

\epr

\begin{lemma}\label{exist}
For $n\in \N\setminus\{0,1\}$, the function $\nabla \big[\tfrac{1}{n}\Ll_{\Lambda_n}\big]\colon \mathcal{D}_n\mapsto\mathbb{R}^2$ defined as
\begin{align}
\nabla [\tfrac{1}{n}\Ll_{\Lambda_n}\bigr]({\bf h})&=(\partial_{h_0}[\tfrac{1}{n} \Ll_{\Lambda_n},\partial_{h_1}\Ll_{\Lambda_n})({\bf h})\\
&=\bigg(\tfrac{1}{n}\sum_{i=0}^{n-1}\tfrac{i}{n}\Ll'(\tfrac{i}{n}h_0+h_1),\ \tfrac{1}{n}\sum_{i=0}^{n-1}\Ll'(\tfrac{i}{n}h_0+h_1)\bigg)
\end{align}
is a $\mathcal{C}^1$ diffeomorphism. Moreover, for all $M>0$ there exists a $\eta>0$ and a $n_0\in \N$ so that  
$\big\|\nabla [\tfrac{1}{n}\Ll_{\Lambda_n}\bigr]({\bf h})\big\|> M$ for  $n\geq n_0$ and  ${\bf h}\in \cD_n\setminus K_\eta$. 
\end{lemma}

\bpr
The first part of the proof, i.e., showing that $\nabla [\tfrac{1}{n}\Ll_{\Lambda_n}\bigr]$ is a $\cC^1$ diffeomorphism, 
is similar to that of Lemma \ref{diffeo} above. For the second part of the Lemma, we first note that 
$\lim_{\eta\to 0^+} \min\{\Llam({\bf h})\colon\, {\bf h}\in \partial K_\eta\}=\infty$. Then, for a given $M>0$ we can pick 
$\eta_0>0$ so that $\Llam$ remains larger than $2M$ on $\partial K_{\eta_0}$. Moreover, Lemma \ref{convu} ensures that 
$\frac{1}{n} \Ll_{\Lambda_n}$ converges to $\Llam$ uniformly on $K_{\eta_0}$ and therefore, there exists $n_0\in \N$ such that 
for all $n\geq n_0$, $\frac1n \Ll_{\Lambda_n}$ remains strictly larger than $M$ on $\partial K_{\eta_0}$. Consider ${\bf h}\in \cD_n\setminus K_\eta$ and let $t\in (0,1)$ be the unique solution of $t {\bf h}\in \partial K_{\eta_0}$. By convexity and since 
$\frac1n \Ll_{\Lambda_n}(0,0)=0$ we claim that $\frac1n \Ll_{\Lambda_n}({\bf h})\geq  \frac1n \Ll_{\Lambda_n}(t {\bf h}) >M$ which completes the proof.

\epr

\begin{remark}\label{rem:relhzerohuncontinu}
As in the proof of Lemma \ref{diffeo} above  we  denote by $\tilde
{\bf h}:=(\tilde { h}_0,\tilde { h}_1)$ the inverse function of
$\nabla \Llam$. Since $\Ll$ is an even function, we easily obtain,
for instance by observing that $T_{(q,0)}(h_0,h_1) = T_{(q,0)}(h_0,-h_0-h_1)$, that $\tilde{h}_0(q,0)=-2\tilde{h}_1(q,0)>0$ for all $q>0$.
We will also denote by ${\bf h}_n^q:=({ h}_{n,0}^q,{ h}_{n,1}^q)$ the unique solution of 
$\nabla [\tfrac{1}{n}\Ll_{\Lambda_n}\bigr]({\bf h})=(q,0)$ for all $n\geq 2$ and $q>0$. Again the fact that $\Ll$ is even ensures that 
${ h}_{n,0}^q (1-\tfrac{1}{n})=-2 { h}_{n,1}^q >0$.
\end{remark}

At this stage, we have enough tools to prove Proposition \ref{convunif}.
\begin{proof}[Proof of Proposition \ref{convunif}]
Pick $q\in[q_1,q_2]$, $n\in\mathbb{N}$ and note that
\begin{equation}\label{ttg}
\bigl|\bigl[\tfrac{1}{n}\Ll_{\Lambda_n}({\bf h}_n^q)-h_{n,0}^q\,q\bigr]-\bigl[\Ll_{\Lambda}(\tilde {\bf h}(q,0))-\tilde h_0(q,0)\,q\bigr]\bigr|\leq A+B+C
\end{equation}
with
\begin{align}\label{UVW}
A&=\bigl|\tfrac{1}{n}\Ll_{\Lambda_n}({\bf h}_n^q)-\Ll_{\Lambda}({\bf h}_n^q)\bigr|,\quad B=\bigl|\Ll_{\Lambda}({\bf h}_n^q)-\Ll_{\Lambda}(\tilde {\bf h}(q,0)) \bigr|,\quad C=q\,\bigl|h_{n,0}^q-\tilde h_0(q,0)\bigr|.
\end{align}
From Lemma \ref{exist}, we know that there exists an $\eta>0$ and a $n_0\in\mathbb{N}$ such that ${\bf h}_n^q\in K_\eta$ for all $q\in[q_1,q_2]$ and $n\geq n_0$. By using Lemma \ref{convu} with $(j_1,j_2)=(0,0)$ and $K=K_\eta$ we can claim that there exists a $c_1>0$ satisfying $A\leq\tfrac{c_1}{n}$ for $n\geq n_0$ and $q\in[q_1,q_2]$. The $B$ quantity is dealt with by applying Corollary \ref{graconv} with $K=K_\eta$, that is there exists a $c_2>0$ such that 
\begin{equation}\label{gfd1}
\sup_{x\in K_\eta}\bigl\|\nabla\bigl[\tfrac{1}{n}\Ll_{\Lambda_n}\bigr](x)-\nabla \Llam(x)\bigr\|\leq\frac{c_2}{n},\quad n\geq n_0.
\end{equation}
Therefore, for $q\in[q_1,q_2]$ and $n\geq n_0$ we can write
\begin{align}
\nabla\bigl[\tfrac{1}{n}\Ll_{\Lambda_n}\bigr]({\bf h}_n^q)&=\nabla \Ll_{\Lambda}({\bf h}_n^q)+\epsilon_{n,q},\\
\nonumber(q,0)&=\nabla \Ll_{\Lambda}({\bf h}_n^q)+\epsilon_{n,q}
\end{align}
with $||\gep_{n,q}||\leq\tfrac{c_2}{n}$. Therefore, by Lemma \ref{diffeo}, we can claim that ${\bf h}_n^q=\tilde {\bf h}\bigl((q,0)-\gep_{n,q}\bigr)$. We set
\begin{equation*}
Q_n=\bigl\{(x,y)\in\mathbb{R}^2\colon\,{\text d}\bigl((x,y),[q_1,q_2]\times\{0\}\bigr)\leq\tfrac{c_2}{n}\bigr\},
\end{equation*}
so that there exists a $n_1\geq n_0$ such that $Q_{n_1}$ is a convex subset of $\mathcal{D}$ and since ${\bf x}\mapsto\tilde {\bf h}({\bf x})$ is $\mathcal{C}^1$ on $\mathcal{D}$ we can claim that $\tilde {\bf h}$ is Lipschitz on $Q_{n_1}$. Thus, there exists a $c_3>0$ such that
\begin{equation}\label{ineg}
||{\bf h}_n^q-\tilde {\bf h}((q,0))||\leq c_3\,||\epsilon_{n,q}||\leq\tfrac{c_2 c_3}{n},\quad q\in[q_1,q_2],n\geq n_1,
\end{equation}
and this proves \eqref{eq:convuniff1}. Moreover
\begin{equation}\label{C}
C\leq q_2\,||{\bf h}_n^q-\tilde {\bf h}((q,0))||\leq\tfrac{q_2 c_2 c_3}{n},\quad q\in[q_1,q_2],n\geq n_1.
\end{equation}
Finally, since $\Llam$ is $\mathcal{C}^1$ on $\mathcal{D}$, there exists a $c_4>0$ such that $\Llam$ is Lipschitz with constant $c_4$ on $Q_{n_1}$. Thus,
\begin{equation}\label{B}
B\leq c_4||{\bf h}_n^q-\tilde {\bf h}((q,0))||\leq\tfrac{c_2 c_3 c_4}{n},\quad q\in[q_1,q_2],n\geq n_1.
\end{equation}
This completes the proof of Proposition \ref{convunif}.
\end{proof}

\section{Limit theorems for the joint distribution}\label{sec:lclt}
In Section \ref{pf:clt} below, we give a proof of Proposition \ref{lem:bHtilde} which estimates, uniformly in $q\in [q_1,q_2]\subset (0,\infty)$, the probability of the event  $\{\Lambda_n=(Y_n,V_n)=(nq,0)\}$  under the tilted law $\mathbf{P}_{n,{\bf h}_n^q}$ (recall \eqref{revt}). To that aim, we state and prove Proposition \ref{prop:lclt}, which gives a local central limit Theorem for $(Y_n,V_n)$ under $\mathbf{P}_{n,{\bf h}_n^q}$.
In Section \ref{sec:wulff}, we prove Proposition \ref{prop:comppos} which allows us to
bound from below the probability that, under $\mathbf{P}_\beta$ and conditioned on both $V_n=0$ and $Y_n=nq$ the random walk $V$ remains strictly positive.

\subsection{Proof of Proposition \ref{lem:bHtilde}}\label{pf:clt}  

We display the proof of Proposition \ref{lem:bHtilde} which turns out to be a straightforward consequence of  Proposition 
\ref{prop:lclt} below. The latter Proposition will be proven at the end of the Section.

\begin{proof}
Recall (\ref{eq:Yn}--\ref{revt}) and 
for any ${\bf h}\in\mathcal{D}$, define the matrix
\begin{equation}
\mathbf{B}({\bf h}):=\Hess \Llam({\bf h})
\end{equation}
and let $\Theta$ be the Gaussian random vector with zero mean and covariance matrix $\mathbf{B}({\bf h})$. We denote the density of $\Theta$ by
\begin{equation}
f_{\bf h}(X)=\frac{1}{2\pi\sqrt{\det\mathbf{B}({\bf h})}}\exp\Bigl(-\tfrac{1}{2}\langle\mathbf{B}({\bf h})^{-1}X,X\rangle\Bigr),\quad X\in \R^2,
\end{equation}
and its characteristic function by 
\begin{equation}\label{phibarr}
\bar{\Phi}_{{\bf h}}(T)=\exp\bigl(-\tfrac{1}{2}\langle\mathbf{B}({\bf h})T,T\rangle\bigr),\quad T\in \R^2.
\end{equation}

Consider now the case $(Y_N,V_N)=(Nq_{N,L},0)$ as in Section \ref{pr:prop1} and recall that $q_{N,L}\in\bigl[\frac{1}{2a_2^2},\frac{1}{a_1^2}\bigr]$. We will show that the local central limit theorem below is valid uniformly in $q$ in some compact subsets.
\begin{proposition}\label{prop:lclt}
For $[q_1,q_2]\subset\mathbb{R}$ we have
\begin{equation}\label{eq:lclt}
\tau_N:=\sup_{q\in[q_1,q_2]}\sup_{x,y\in\mathbb{Z}}\,\Bigl|N^2\mathbf{P}_{N,{\bf h}_N^q}\bigl(NY_N=N^2q+x,V_N=y\bigr)-f_{\tilde {\bf h}(q,0)}\Bigl(\tfrac{x}{N^{3/2}},\tfrac{y}{\sqrt{N}}\Bigr)\Bigr|\to 0,
\end{equation}
as $N\to\infty$. 
\end{proposition}
By applying Proposition \ref{prop:lclt} with $x=y=0$, we obtain that
\begin{equation}\label{lhs}
\sup_{q\in[q_1,q_2]}\Bigl|N^2\mathbf{P}_{N,{\bf h}_N^q}(NY_N=N^2q,\,V_N=0)-f_{\tilde {\bf h}(q,0)}(0,0)\Bigr|\leq\tau_N\to0,
\end{equation}
and since the Hessian matrix ${\bf B}\big(\tilde {\bf h}(q,0)\big)$ is uniformly bounded in $q\in[q_1,q_2]$, we observe that there exists $C>0$ such that 
\begin{equation}
\tfrac{1}{CN^2}\leq\mathbf{P}_{N,{\bf h}_N^q}(NY_N=N^2q,\,V_N=0)\leq\tfrac{C}{N^2}\quad\text{for $N$ large enough}
\end{equation}
which completes the proof of Proposition \ref{lem:bHtilde}.
\end{proof}

\subsubsection{\bf Proof of Proposition \ref{prop:lclt}}
%
We follow closely the proof of Dobrushin and Hryniv in \cite{DH96}, making sure that the result holds uniformly in $q\in [q_1,q_2]$. From Lemma \ref{diffeo} and Lemma \ref{exist}, there exists $\eta>0$ such that both $\tilde {\bf h}(q,0)$ and ${\bf h}_N^q$ are in $K_\eta$ for all $q\in[q_1,q_2$] and for $N$ large enough. 

We let $\mathfrak{E}$ be the holomorphic function defined on  $\{z\in \mathbb{C} \colon \text{Re}(z)\in (-\beta/2,\beta/2) \}$
by $\mathfrak{E}(z)=\bE_\beta(e^{z U_1})$.
For any $h\in (-\beta/2,\beta/2)$ and $t\in \R$ we set
\begin{equation}\label{etui}
\phi_h(t):=\mathfrak{E}(h+it)/\mathfrak{E}(h).
\end{equation}

Let us state some properties of the function $\phi_h(t)$ that will be
used in the sequel (they are established
in \cite{DH96}). First of all, for any $h\in\mathcal{K}:=[-\beta/2+\eta, \beta/2-\eta]$ and $t\in\mathbb{R}$
\begin{equation}\label{eq:chabou1}
|\phi_h(t)|\leq\phi_h(0)=1.
\end{equation}
Secondly, for any $\delta\in(0,\pi)$, there exists a constant $C=C(\mathcal{K},\delta)>0$ such that for every $h\in\mathcal{K}$ and any $t\in[\delta,2\pi-\delta]$, we have
\begin{equation}\label{eq:chabou2}
|\phi_h(t)|\leq e^{-C}.
\end{equation}
And finally, there exists a constant $\alpha=\alpha(\mathcal{K})>0$ such that for all $h\in\mathcal{K}$ and any $t$, $|t|\leq\pi$, the following inequality holds
\begin{equation}\label{eq:chabou3}
|\phi_h(t)|\leq\exp(-\alpha^2t^2\Ll''(h)).
\end{equation}
For any $T=(t_0,t_1)\in\mathbb{R}^2$, let $\Phi_{N,{\bf h}_N^q}(T)$ be the characteristic function of the random vector $\Lambda_N=(Y_N,V_N)$. Let us rewrite it with the functions $\phi_h(t)$,
\begin{equation}\label{eq:chaLambda}
\Phi_{N,{\bf h}_N^q}(T)=\mathbf{E}_{N,{\bf h}_N^q}\bigl[e^{i\langle T,\Lambda_N\rangle}\bigr]=\prod_{j=1}^N\phi_{h_{j,N}}(t_{j,N}),
\end{equation}
where
\begin{equation}\label{eq:tjN}
h_{j,N}=(1-\tfrac{j}{N})h_{N,0}^q+h_{N,1}^q\quad\text{and}\quad t_{j,N}=(1-\tfrac{j}{N})t_0+t_1.
\end{equation}
Note that
\begin{equation}\label{chet}
\hat{\Phi}_{N,{\bf h}_N^q}(T)=\Phi_{N,{\bf h}_n^q}(N^{-1/2}T)\exp\Bigl(-\tfrac{i}{\sqrt{N}}\langle T,\mathbf{E}_{N,{\bf h}_N^q}(\Lambda_N)\rangle\Bigr)
\end{equation}
is the characteristic function of the centered random vector $\Lambda^\star_N:=\Lambda_N-\mathbf{E}_{N,{\bf h}_N^q}(\Lambda_N)$.

Let ${\bf v}_N=(\tfrac{x}{N^{3/2}},\tfrac{y}{\sqrt{N}})$. Using the well know inversion formula for the Fourier transform, we rewrite the left hand side of \eqref{eq:lclt}, i.e.,
\begin{equation}
R_N=N^2\mathbf{P}_{N,{\bf
    h}_N^q}\bigl(NY_N=N^2q+x,V_N=y\bigr)-f_{\tilde {\bf h}(q,0)}({\bf v}_N)
\end{equation}
in the form
\begin{equation}\label{eq:inverFour}
R_N=\tfrac{1}{(2\pi)^2}\int_\mathcal{A}\hat{\Phi}_{N,{\bf
    h}_N^q}(T)e^{-i\langle T,{\bf
    v}_N\rangle}dT-\tfrac{1}{(2\pi)^2}\int_{\mathbb{R}^2}\bar{\Phi}_{\tilde
  {\bf h}(q,0)}(T)e^{-i\langle T,{\bf v}_N\rangle}dT,
\end{equation}
where
\begin{equation}
\mathcal{A}=\{T=(t_0,t_1)\in\mathbb{R}^2\colon|t_0|\leq\pi N^{3/2},\,|t_1|\leq\pi\sqrt{N}\}.
\end{equation}
Following the proof in \cite{DH96} we bound the left hand side of \eqref{eq:inverFour} by the sum of four terms,
\begin{equation}
|R_N|\leq(2\pi)^{-2}(J_1^{(q)}+J_2^{(q)}+J_3^{(q)}+J_4^{(q)})
\end{equation}
where, for some positive constants $A$ and $\Delta$,
\begin{align}
J_1^{(q)}&=\int_{\mathcal{A}_1}\bigl|\hat{\Phi}_{N,{\bf h}_N^{q}}(T)-\bar{\Phi}_{\tilde {\bf h}(q,0)}(T)\bigr|dT,\quad \mathcal{A}_1=[-A,A]^2,\\
J_2^{(q)}&=\int_{\mathcal{A}_2}\bar{\Phi}_{\tilde {\bf h}(q,0)}(T)dT,\quad \mathcal{A}_2=\mathbb{R}^2\setminus \mathcal{A}_1,\\
J_3^{(q)}&=\int_{\mathcal{A}_3}\bigl|\hat{\Phi}_{N,{\bf h}_N^{q}}(T)\bigr|dT,\quad \mathcal{A}_3=\{T\in\mathbb{R}^2\colon|t_l|\leq\Delta\sqrt{N},\,l=0,1\}\setminus \mathcal{A}_1,\\
J_4^{(q)}&=\int_{\mathcal{A}_4}\bigl|\hat{\Phi}_{N,{\bf h}_N^{q}}(T)\bigr|dT,\quad \mathcal{A}_4=\mathcal{A}\setminus(\mathcal{A}_1\cup \mathcal{A}_3).
\end{align}

For an arbitrary $\epsilon>0$, Dobrushin and Hryniv proved that for a convenient choice of the constants $A=A(\epsilon)$ and $\Delta$, we have the bounds $J_i^{(q)}<\epsilon/4$ for $i=1,2,3,4$ for sufficiently large $N$. Therefore, the proof will be complete once we show that this assertion is also valid uniformly in $q\in[q_1,q_2]$. It remains to evaluate all $J_i^{(q)}$.

First, we bound $J_1^{(q)}$. For ${\bf h}\in\mathcal{D}_n$, define the matrix
\begin{equation}
\mathbf{B}_n({\bf h}):=\tfrac{1}{n}\Hess \Ll_{\Lambda_n}({\bf h}),\quad n\in\mathbb{N}.
\end{equation}
By Lemma \ref{convu} and Proposition \ref{convunif}, we obtain the relation
\begin{equation}\label{eq:matrixBN}
\mathbf{B}_N({\bf h}_N^{q})=\mathbf{B}(\tilde {\bf h}(q,0))+R'_N,
\end{equation}
with the bound $\valabs{R'_N} \le C_1(q_1,q_2) N^{-1}$  uniform in
$q\in[q_1,q_2]$. 

Recall that $\cK=[-\beta/2+\eta,\beta/2-\eta]$. Since $\mathfrak{E}$ is holomorphic on  $\{z\in \mathbb{C} \colon \text{Re}(z)\in (-\beta/2,\beta/2) \}$, for any $\eta>0$ there exists an $A'>0$ so that 
$\text{Re}(\mathfrak{E}(z))>0$ for $z\in \cK+i [-A',A']$ and therefore we can use a branch of the complex logarithm 
to extend the function $\mathfrak{L}$ (that equals $\log \mathfrak{E}$) to $\cK+i [-A',A']$. 
We observe that  ${\bf h}\in K_\eta$ and $T\in \tfrac12[-A',A']^2$ yield $(1-\tfrac{j}{n})h_0+h_1\in \cK$  and 
  $(1-\tfrac{j}{n})t_0+t_1\in [-A',A']$ for all  $j\in \{1,\dots,N\}$. Thus,  we can 
extend  $\Ll_{\Lambda_n}$ to $K_\eta \times \tfrac12 [-A',A']^2$ with the formula 
\begin{equation}\label{eq:LlambdaN}
{\textstyle \Ll_{\Lambda_n}({\bf h}+iT):=\sum_{j=1}^n \Ll\Bigl(\bigl(1-\tfrac{j}{n}\bigr)(h_0+it_0)+h_1+i t_1\Bigr).}
\end{equation}
Similarly, we extend   $\Ll_{\Lambda}$ to $K_\eta \times \tfrac12 [-A',A']^2$ and Lemma \ref{convu} can, without further difficulty, be extended to $K_\eta \times \tfrac12 [-A',A']^2$. In particular, any partial derivative
of order $3$ of $\frac{1}{n} \mathfrak{L}_{\Lambda_n}$ converges uniformly to its counterpart of $\mathfrak{L}_\Lambda$ on $K_\eta \times \tfrac12 [-A',A']^2$.   
Consequently, for $N$ large enough, we make sure that for $q\in [q_1,q_2]$ and for $T\in \cA_1$, 
we have ${\bf h}_{N}^q\in K_\eta$ and  $T/N\in \frac12 [-A',A']^2$ so that  we can consider the remainder
\be{remindert}
R''_N= \Ll_{\Lambda_N}({\bf h}_N^{q}+iN^{-1/2}T)-\Ll_{\Lambda_N}({\bf h}_N^{q})-\tfrac{i}{\sqrt{N}}\langle T,\mathbf{E}_{N,{\bf h}_N^{q}}(\Lambda_N)\rangle+\tfrac{1}{2}\langle\mathbf{B}_N({\bf h}_N^{q})T,T\rangle,
\ee
and apply a Taylor-Lagrange inequality to assert that  there exists  a constant $C(A,q_1,q_2)>0$ such that for $N$ large enough 
$|R'_N|\leq C(A,q_1,q_2)/\sqrt{N}$ uniformly in $q\in [q_1,q_2]$ and
$T\in \cA_1$. 



Therefore, we can use \eqref{phibarr}, \eqref{etui}, (\ref{eq:chaLambda}--\ref{chet}) and \eqref{eq:matrixBN} to get
\begin{equation}
\sup_{q\in[q_1,q_2],T\in \cA_1}\;\bigl|\hat{\Phi}_{N,{\bf
    h}_N^{q}}(T)-\bar{\Phi}_{\tilde {\bf h}(q,0)}(T)\bigr| = 
\sup_{q\in[q_1,q_2],T\in \cA_1}\valabs{e^{\frac12 R'_N \norme{T}^2 + R''_N} -1} 
\to 0\quad\text{as}\ N\to\infty.
\end{equation}
Hence, for every finite $A>0$, we obtain the convergence $J_1^{(q)}\to 0$ as $N\to\infty$ uniformly in $q\in[q_1,q_2]$.

Let $\underline{B}$ be such that $0<\underline{B}\leq {\bf B}(\tilde {\bf h}(q,0))$ for all $q\in[q_1,q_2]$. Hence, we can bound $J_2^{(q)}$ as follows
\begin{equation}
\sup_{q\in[q_1,q_2]}J_2^{(q)}\leq\int_{\mathcal{A}_2}e^{-\tfrac{1}{2}\langle\underline{B}T,T\rangle}dT\to0\quad\text{as}\ \ A\to\infty.
\end{equation} 

To estimate $J_3^{(q)}$ we fix any $T\in\mathcal{A}_3$ and put $\Delta=\pi/2$. Then all the numbers $t_{j,N}$ in \eqref{eq:tjN} satisfy the condition $|t_{j,N}|\leq \pi\sqrt{N}$, evaluating each factor in \eqref{eq:chaLambda} with the help of \eqref{eq:chabou3} and \eqref{eq:matrixBN} we obtain the bound
\begin{equation}
\bigl|\hat{\Phi}_{N,{\bf h}_N^{q}}(T)\bigr|\leq\exp(-\alpha^2\langle {\bf B}_N({\bf h}_N^{q})T,T\rangle)\leq C \exp(-\alpha^2 \langle {\bf B}(\tilde {\bf h}(q,0))T,T\rangle),
\end{equation}
for some constant $C>0$. As a result,
\begin{equation}
\sup_{q\in[q_1,q_2]}J_3^{(q)}=\sup_{q\in[q_1,q_2]}\int_{\mathcal{A}_3}\bigl|\hat{\Phi}_{N,{\bf h}_N^{q}}(T)\bigr|dT\leq C \int_{\mathcal{A}_2}\exp(-\alpha \langle\underline{B}T,T\rangle)dT\to0\quad\text{as}\ A\to\infty.
\end{equation}

To evaluate $J_4^{(q)}$ put
$\delta=\tfrac{1}{17\,  (2)^2}$ and
for any $T\in\mathcal{A}_4$ denote by $\mathbf{N}_N(T)$ the number of indexes $j=1,2,\ldots,N$ such that $\tau_{j,N}\notin\mathcal{O}_\delta:=\cup_{m\in\mathbb{Z}}[m-\delta,m+\delta]$, where
\begin{equation}
\tau_{j,N}:=\tfrac{1}{2\pi\sqrt{N}}t_{j,N}.
\end{equation}
Use \eqref{eq:chabou1} and \eqref{eq:chabou2} to estimate those factors in \eqref{eq:chaLambda} and we have
\begin{equation}
\bigl|\hat{\Phi}_{N,{\bf h}_N^{q}}(T)\bigr|=\prod_{j=1}^N\Bigl|\phi_{h_{j,N}}\Bigl(\tfrac{1}{\sqrt{N}}t_{j,N}\Bigr)\Bigr|\leq\exp(-C\mathbf{N}_N(T)).
\end{equation}
A lower bound of $\mathbf{N}_N(T)$ is given in \cite[p.~443]{DH96}: for all $T\in\mathcal{A}_4$ and $N$ large enough, there exists a constant $\kappa>0$ such that
$\mathbf{N}_N(T)\geq\kappa N$.
Then, uniformly in $q\in[q_1,q_2]$,
\begin{equation}
J_4^{(q)}=\int_{\mathcal{A}_4}\bigl|\hat{\Phi}_{N,{\bf h}_N^{q}}(T)\bigr|dT\leq(2\pi)^2 N^2 \exp(-C\kappa N)\to 0\quad\text{as}\;N\to\infty.
\end{equation}

\subsection{Proof of Proposition \ref{prop:comppos} (Unique excursion for large area)} \label{sec:wulff}

From now on, the letters $C,C',C_1,\ldots$ shall denote constants that do not depend on $N$ and on $q\in[q_1,q_2]\subset(0,\infty)$. In other words, all the bounds we are going to establish are uniform in $N\geq N_0$ and $q\in[q_1,q_2]$.

To begin with, we prove Lemma \ref{lem:ub:vi}  subject to  Lemmas \ref{lem:deflambda}
and \ref{lem:symetrie} below. Lemma \ref{lem:ub:vi} is crucial in the proof of Proposition \ref{prop:comppos}. It allows us indeed to bound from below, for any $j\in \N$, the probability 
that the random walk $V$,  conditioned on making a large area, is below $0$ at time $j$. Such a lower bound  was available in \cite{DH96} but only for $j$ of order $N$. Here, we deal with any $j\leq N$.
The first step of the proof is an upper bound on the moment generating function of the tilted random walk $V$.
\begin{lemma}\label{lem:deflambda}
There exist three positive constants $C',C_1,\lambda$ such that for every integer $j\leq N/2$, the following bound holds
\begin{equation}
\mathbf{E}_{N,{\bf h}_N^{q}}\bigl[e^{-\lambda V_j}\bigr]\leq C'e^{-C_1 j},\quad N\in\mathbb{N}.
\end{equation}
\end{lemma}

\begin{proof}
Under the tilted law (see \ref{definH}) the increments $U_i=V_i-V_{i-1}$ are
still independent but no more identically distributed. For any positive $\lambda$ we have
\begin{equation}\label{egals}
\log\mathbf{E}_{N,{\bf h}_N^{q}}\bigl[e^{-\lambda V_j}\bigr]=\sum_{1\leq i\leq j}\bigl(\Ll(-\lambda+h_N^i)-\Ll(h_N^i)\bigr)
\end{equation}
with $h_N^i:=(1-\frac{i}{N})h^q_{N,0}+h^q_{N,1}$. By Remark \ref{rem:relhzerohuncontinu}, we know that for all
$q>0$ and $N\geq 2$,
\be{masubr}
h_{N,0}^q (1-\tfrac{1}{N})=-2 h_{N,1}^q>0.
\ee 
A straightforward consequence of \eqref{masubr} is that   $h_N^i\geq 0$  for all $i\leq N/2$. 
Then, the convexity of
$\Ll(\cdot)$ and the fact that $\Ll(0)=\Ll'(0)=0$ yield that there exists a $c>0$ so that for all  $i\leq N/2$ and $\lambda$ small enough
\be{tresf}
\Ll(-\lambda+h_N^i)-\Ll(h_N^i)\leq \Ll(-\lambda)\leq c \lambda^2.
\ee

We established in Proposition \ref{convunif} the existence of   $C>0$ and $N_0\in \N$ such that for all $N\geq N_0$, and every $q\in[q_1,q_2]$, we have
\begin{equation}\label{ecart2}
\bigl\|{\bf h}_N^{q}-\tilde {\bf h}(q,0)\bigr\|\leq\tfrac{C}{N}.
\end{equation} 
Thanks to Lemma \ref{diffeo} and Remark \ref{rem:relhzerohuncontinu}, there exists a constant $R>0$ such that
\begin{equation}\label{ecart}
\tilde{h}_0(q,0)\geq R>0\quad\forall q\in[q_1,q_2].
\end{equation}
Thus, provided  $N_0$ is chosen large enough, we deduce from \eqref{ecart2} and \eqref{ecart} that $h_{N,0}^q\geq R/2$ for $N\geq N_0$
and $q\in [q_1,q_2]$.
Moreover, thanks to \eqref{masubr}, we also write $h_N^i\geq \frac14 h_{N,0}^q$  for $i\leq N/4$ such that finally
$h_N^i\geq R/8$ for $i\leq N/4$.
Observe that by convexity of $\Ll(.)$,
\begin{equation}
\sum_{1\leq i\leq j}\bigl(\Ll(-\lambda+h_N^i)-\Ll(h_N^i)\bigr)\leq -\lambda\sum_{1\leq i\leq j}\Ll'(-\lambda+h^i_N).
\end{equation}
Hence, for $j\leq N/4$ and for $\lambda\leq R/16$ we have 
\begin{equation}\label{ecart3}
\sum_{1\leq i\leq j}\bigl(\Ll(-\lambda+h_N^i)-\Ll(h_N^i)\bigr)\leq -\lambda j \Ll'(\tfrac{R}{16}).
\end{equation}
For $N/4\leq j \leq N/2$ in turn we split the sum in the l.h.s. of \eqref{ecart3} into a sum over 
$i\leq N/4$ (that is dealt with as in \eqref{ecart3}) and a sum over $i\geq N/4$ (that is dealt with by using \eqref{tresf}). 
Thus, 
\begin{equation}\label{tsty}
\sum_{1\leq i\leq j}\bigl(\Ll(-\lambda+h_N^i)-\Ll(h_N^i)\bigr)=-\lambda \tfrac{N}{4} \Ll'(\tfrac{R}{16})+c (j-\tfrac{N}{4}) \lambda^2
\leq \tfrac{N}{4} \big(c \lambda^2-\lambda  \Ll'(\tfrac{R}{16})\big).
\end{equation}
It remains to choose $\lambda>0$ small enough to make sure that $c \lambda^2-\lambda  \Ll'(R/16)>0$ and then, 
\eqref{ecart3} and \eqref{tsty} complete the proof.

\end{proof}

The next lemma ensures that we can restrict ourselves to $j\leq N/2$.
\begin{lemma}\label{lem:symetrie}
For $a\in\mathbb{R}$ and $j\in\{1,\ldots,N\}$
\begin{equation}
\mathbf{P}_\beta(V_j\leq a,\,Y_N=Nq,\,V_N=0)=\mathbf{P}_\beta(V_{N-j}\leq a,\,Y_N=Nq,\,V_N=0).
\end{equation}
\end{lemma}
\begin{proof}
We just need to use time reversal, i.e.,
\begin{equation}
(V_N-V_{N-j},\,0\leq j\leq N)\overset{d}{=}(V_j,\,0\leq j\leq N),
\end{equation}
to obtain that
\begin{equation}
\mathbf{P}_\beta(V_j\leq a,\,Y_N=Nq,\,V_N=0)=\mathbf{P}_\beta(-V_{N-j}\leq-a,\,-Y_N=Nq,\,V_N=0).
\end{equation}
By using the symmetry of $V$, we complete the proof:
\begin{equation}
(-V_j,\,0\leq j\leq N)\overset{d}{=}(V_j,\,0\leq j\leq N).
\end{equation}
\end{proof}

At this stage, we need to use precise results for the local central
limit theorem. We recall \eqref{revt} and for convenience we use the
notations 
\begin{equation}\label{eq:notalphaxi}
\alpha_N^q:=\mathbf{P}_{N,{\bf
    h}_N^{q}}(NY_N=N^2q,\,V_N=0) \quad \text{and}\quad
\xi_N^q:=\exp{(\Ll_{\Lambda_N}({\bf h}_N^{q})-Nh^q_{N,0}\,q)}\,.
\end{equation}
Hence, we have
\begin{equation}
\mathbf{P}_\beta(Y_N=Nq,\,V_N=0)=\xi_N^q\alpha_N^q.
\end{equation}
We can handle $\alpha_N^q$ with the help of Proposition \ref{lem:bHtilde}: there exists a $C_2>0$ such that 
\begin{equation}\label{eq:balpha}
\tfrac{1}{C_2}\tfrac{1}{N^2}\leq\alpha_N^q\leq\tfrac{C_2}{N^2}.
\end{equation}
Proposition \ref{convunif} allows us to write that there exists a positive constant $C_3$ so that
\begin{equation}\label{eq:bbeta}
e^{-C_3}e^{N(\Ll_{\Lambda}(\tilde {\bf h}(q,0))-\tilde h_0(q,0)\,q)}
\leq\xi^q_N\leq e^{C_3}e^{N(\Ll_{\Lambda}(\tilde {\bf h}(q,0))-\tilde h_0(q,0)\,q)}.
\end{equation}
We can state that
\begin{lemma}\label{lem:ub:vi}
There exists a constant $\lambda>0$ such that for all $a>0, q\in[q_1,q_2], N\geq N_0$ and $0\leq j\leq N$
\begin{equation}
\mathbf{P}_\beta(V_j\leq-a,\,Y_N=Nq,\,V_N=0)\leq\xi_N^q\,C'e^{-C_1(j\wedge(N-j))-\lambda a}.
\end{equation}
\end{lemma}
\begin{proof}
By the symmetry in Lemma \ref{lem:symetrie}, we can without loss of generality assume $j\leq N/2$. By using Lemma~\ref{lem:deflambda}, we can write
\begin{align*}
\mathbf{P}_\beta(V_j\leq -a,\,Y_N=Nq,\,V_N=0)&\leq\mathbf{E}_\beta\bigl[e^{-\lambda V_j},\,Y_N=Nq,\,V_N=0\bigr]e^{-\lambda a}\\
&=\xi_N^q e^{-\lambda a}\mathbf{E}_{N,{\bf h}_N^{q}}\bigl[e^{-\lambda V_j},\,Y_N=Nq,\,V_N=0\bigr]\\
&\leq\xi_N^q e^{-\lambda a}\mathbf{E}_{N,{\bf h}_N^{q}}\bigl[e^{-\lambda V_j}\bigr]\leq\xi_N^q\,C' e^{-C_1 j-\lambda a}.
\end{align*}
\end{proof}

\begin{proof}[Proof of Proposition \ref{prop:comppos}]
Let $u_N=\lfloor\nu\log N\rfloor$ where $\nu>0$ will be chosen afterward. The first step is to write
\begin{multline}\label{1st}
\mathbf{P}_\beta(V_i>0,0<i<N;\,NY_N=N^2q,\,V_N=0)\geq\\
\mathbf{P}_\beta(V_1=V_{N-1}=u_N,\,V_i>0,2<i<N-2;\,NY_N=N^2q,\,V_N=0).
\end{multline}
By using  Markov's property at time $1$ and $N-1$, we obtain
\begin{multline}\label{eq:1bound}
\mathbf{P}_\beta(V_1=V_{N-1}=u_N,\,V_i>0,2<i<N-2;\,NY_N=N^2q,\,V_N=0)\\
=
\mathbf{P}_\beta(U_1=u_N)^2\,\mathbf{P}_\beta(V_i>-u_N,1<i<N-3;\,(N-2)Y_{N-2}=N^2q-(N-1)u_N,\,V_{N-2}=0).
\end{multline}
We shall use a basic  lower bound
\begin{multline}\label{trbou}
\mathbf{P}_\beta(V_i>-u_N,1<i<N-3;\,(N-2)Y_{N-2}=N^2q-(N-1)u_N,\,V_{N-2}=0)\\
\geq\;
\mathbf{P}_\beta ((N-2)Y_{N-2}=N^2q-(N-1)u_N,\,V_{N-2}=0)\\
-
 \sum_{i=1}^{N-3}\mathbf{P}_\beta(V_i\leq-u_N,\,(N-2)Y_{N-2}=N^2q-(N-1)u_N,\,V_{N-2}=0).
\end{multline}
We take care of the second term  by letting $q'=\frac{N^2 q -(N-1) u_N}{(N-2)^2}$ in Lemma~\ref{lem:ub:vi}
\begin{multline}\label{eq:2bound}
\sum_{i=1}^{N-3}\mathbf{P}_\beta(V_i\leq-u_N,\,(N-2)Y_{N-2}=(N-2)^2 q',\,V_{N-2}=0)
\\ \leq
\xi_{N-2}^{q'}\sum_{i=1}^{N-3}C' e^{-C_1(i\wedge(N-2-i))-\lambda u_N}\leq C_4\xi_{N-2}^{q'}\,e^{-\lambda u_N}.
\end{multline}
Observe that thanks to the notations~\eqref{eq:notalphaxi} we can write the first term in the r.h.s. of \eqref{trbou} as
\begin{equation}\label{eq:3bound}
\mathbf{P}_\beta\bigl((N-2)Y_{N-2}=(N-2)^2 q',\,V_{N-2}=0\bigr)={\xi}_{N-2}^{q'}{\alpha}_{N-2}^{q'}\,.
\end{equation}

Hence,
\begin{multline}
\mathbf{P}_\beta(V_i>0,0<i<N;\,NY_N=N^2q,\,V_N=0)
\\ \geq
\mathbf{P}_\beta(U_1=u_N)^2 \xi_{N-2}^{q'} \biggl[\alpha_{N-2}^{q'} -C_4\,e^{-\lambda
  u_N}\biggr].
\end{multline}
Observe that
\begin{equation}
\mathbf{P}_\beta(U_1=u_N)^2=\tfrac{1}{c_\beta^2} e^{-\beta \floor{\nu
    \log N}}\ge \tfrac{1}{c_\beta^2}N^{-\beta\nu}
\end{equation}
and recall that
$\mathbf{P}_\beta(NY_N=N^2q,\,V_N=0)=\xi_N^q\alpha_N^q$.  Therefore
\begin{equation}\label{tib}
\mathbf{P}_\beta\bigl(V_i>0,0<i<N\mid NY_N=N^2q,\,V_N=0\bigr) 
\geq\frac{N^{-\beta\nu}}{c_\beta^2} \frac{\xi_{N-2}^{q'}}{\xi^q_N}
\frac{\alpha_{N-2}^{q'} -C_4\,e^{-\lambda u_N}}{\alpha^q_N} 
\end{equation}
We take care of the last factor with the help of the bound
\eqref{eq:balpha}
\begin{equation}\label{minorationalphaqprime}
  \frac{\alpha_{N-2}^{q'} -C_4\,e^{-\lambda u_N}}{\alpha^q_N}  \ge
  \frac{1}{C_2 N^2} \etp{\frac{1}{C_2 (N-2)^2} - C_5 e^{-\lambda u_N}}
  \ge C_5 N^{-4} 
\end{equation}
for $N$ large, by choosing $\nu > \frac{2}{\lambda}$.
For the second factor, we use the bound \eqref{eq:bbeta}, and the
Lipschitz nature of $\Ll$ and $\tilde{\mathbf{h}}$ on a compact set, and the fact that
$\valabs{q-q'} \le C_6 \frac{\log N}{N}$,
\begin{multline}\label{minorationxiqprime}
  \frac{\xi_{N-2}^{q'}}{\xi^q_N} \ge e^{-2 C_3}
  \exp\etp{N\etp{\etc{\Llam(\tilde{\mathbf{h}}(q',0)) -
        \tilde{h}_0(q',0)q'}
- \etc{\Llam(\tilde{\mathbf{h}}(q,0)) -
        \tilde{h}_0(q,0)q}}} \\
\ge  e^{-2 C_3} e^{-N C_7 \valabs{q-q'}} \ge e^{-2 C_3} e^{- C_7 C_6
  \log N} \ge C_8 N^{-C_9}
\end{multline}
Eventually, combining \eqref{tib}, \eqref{minorationalphaqprime} and
\eqref{minorationxiqprime}, we obtain the lower bound, for $\mu=4+C_9 + \beta \nu$ and
$C>0$ a constant
$$\mathbf{P}_\beta\bigl(V_i>0,0<i<N\mid NY_N=N^2q,\,V_N=0\bigr) \ge C N^{-\mu}\,.$$
\end{proof}

\appendix

\section{Equivalence between Theorem \ref{Theo-shape} and Theorems \ref{Prop5} and \ref{Convenv}}\label{equiv:thDEF}
Assume that  Theorem \ref{Prop5} and \ref{Convenv} hold.  We begin by observing that 
\begin{align}\label{altthD2}
\nonumber d_H\Big(\tfrac{S_L(l)}{\sqrt{L}},\cS_\beta\Big)&\leq \tfrac{N_L(l)}{\sqrt L} d_H\Big( \tfrac{S_L(l)}{N_L(l)},\tfrac{\sqrt{L}}{N_L(l)} \cS_\beta\Big)\\
&\leq \tfrac{N_L(l)}{\sqrt L} d_H\Big( \tfrac{S_L(l)}{N_L(l)},\tfrac{\cS_\beta}{a_\beta}\Big)+ \tfrac{N_L(l)}{\sqrt L} d_H\Big(\tfrac{\cS_\beta}{a_\beta},\tfrac{\sqrt{L}}{N_L(l)} \cS_\beta\Big). 
\end{align}
Theorem \ref{Prop5}, and the inequality $d_H\big(\tfrac{\cS_\beta}{a_\beta},\tfrac{\sqrt{L}}{N_L(l)} \cS_\beta\big)  \leq
C\, |\tfrac{\sqrt{L}}{N_L(l)} -\frac{1}{a_\beta}|$ ($C$ is the radius of a ball containing $\cS_\beta$) ensure that the second term in the r.h.s. of \eqref{altthD2} converges to $0$ in  $P_{L,\beta}$ probability. The same convergence holds for the first term in the r.h.s. of \eqref{altthD2}  and this is a consequence of Theorem \ref{Convenv} and of the inequality
$$d_H\Big(\tfrac{S_L(l)}{N_L(l)},\tfrac{\cS_\beta}{a_\beta}\Big)\leq \max\{\big\|\tilde\cE^+_{l}-\tfrac{\gamma^*_{\beta}}{2}\big\|_{\infty},\big\|\tilde\cE^-_{l}+\tfrac{\gamma^*_{\beta}}{2}\big\|_{\infty}  \}+\tfrac{1}{N_L(l)}.$$
Thus, Theorem \ref{Theo-shape} is a consequence of Theorems   \ref{Prop5} and \ref{Convenv}. Using similar arguments, we can prove that  Theorems \ref{Prop5} and \ref{Convenv} are implied by Theorem \ref{Theo-shape} but we do not give the details here.

\section{Proof of Lemma \ref{ineEarea}}\label{appA}
\begin{proof}
Since $V$ and $A_n$ are symmetric, we can assume that $x,x'\in\mathbb{N}_0:=\mathbb{N}\cup\{0\}$ and thus it is sufficient to show that the result holds for $x'=x+1$. We will argue by induction. Since $A_0=0$, the $m=0$ case is trivial. Now, we assume that the inequality holds true for $m\in\mathbb{N}$. We consider the partition function of size $m+1$, and we can decompose it with respect to the position of $V_1$, i.e.,
\begin{align}\label{eq:expEarea}
\nonumber\mathbf{E}_{\beta,x}\bigl(e^{-\delta A_{m+1}}\bigr)&=\sum_{y\in\mathbb{Z}}\mathbf{E}_{\beta,x}\bigl(e^{-\delta (|y|+|V_2|+\ldots+|V_{m+1}|)}\mathbf{1}_{\{V_1=y\}}\bigr)\\
\nonumber&=\sum_{y\in\mathbb{Z}}\mathbf{P}_{\beta}(U_1=y-x)e^{-\delta|y|}\mathbf{E}_{\beta,y}\bigl(e^{-\delta A_m}\bigr)\\
&=\sum_{y\in\mathbb{N}}R_x(y)e^{-\delta y}\mathbf{E}_{\beta,y}\bigl(e^{-\delta A_m}\bigr)+\mathbf{P}_{\beta}(U_1=x)\mathbf{E}_{\beta}\bigl(e^{-\delta A_m}\bigr),
\end{align}
where $R_x(y)=\mathbf{P}_{\beta}(U_1=y-x)+\mathbf{P}_{\beta}(U_1=-y-x)$. Then, we set $\bar{R}_x(y)=\sum_{y'\geq y}R_x(y')$ for $y\in\mathbb{N}$. Since $\bar{R}_x(1)+\mathbf{P}_{\beta}(U_1=x)=1$, we can rewrite the right hand side in \eqref{eq:expEarea} as 
\begin{multline}\label{eq:expEarea2}
\mathbf{E}_{\beta,x}\bigl(e^{-\delta A_{m+1}}\bigr)=\sum_{y\in\mathbb{N}}\bar{R}_x(y)\Bigl[e^{-\delta y}\mathbf{E}_{\beta,y}\bigl(e^{-\delta A_m}\bigr)-e^{-\delta (y-1)}\mathbf{E}_{\beta,(y-1)}\bigl(e^{-\delta A_m}\bigr)\Bigr]
+\mathbf{E}_{\beta}\bigl(e^{-\delta A_m}\bigr).
\end{multline}
We will show that, for all $y\in\mathbb{N}$, the function $x\mapsto\bar{R}_{x}(y)$ is non-decreasing on $\mathbb{N}_0$. First, if $y\geq x+1$, we obviously have
\begin{equation}
\bar{R}_{x}(y)=\sum_{y'\geq y}R_x(y')\leq\sum_{y'\geq y}R_{x+1}(y')=\bar{R}_{x+1}(y).
\end{equation}
Then, if $1\leq y\leq x$, since
\begin{equation}
\bar{R}_x(y)+\sum_{y'=1}^{y-1}R_x(y')+\mathbf{P}_{\beta}(U_1=x)=\bar{R}_{x+1}(y)+\sum_{y'=1}^{y-1}R_{x+1}(y')+\mathbf{P}_{\beta}(U_1=x+1)=1,
\end{equation}
and
\begin{equation}
\mathbf{P}_{\beta}(U_1=x)+\sum_{y'=1}^{y-1}R_x(y')\geq\mathbf{P}_{\beta}(U_1=x+1)+\sum_{y'=1}^{y-1}R_{x+1}(y'),
\end{equation}
we immediately obtain $\bar{R}_x(y)\leq \bar{R}_{x+1}(y)$.
Coming back to \eqref{eq:expEarea2}, we use the induction hypothesis to claim that 
\begin{equation}
e^{-\delta y}\mathbf{E}_{\beta,y}\bigl(e^{-\delta A_m}\bigr)-e^{-\delta (y-1)}\mathbf{E}_{\beta,(y-1)}\bigl(e^{-\delta A_m}\bigr)\leq 0, \quad y\in \N,
\end{equation}
which, together with the monotonicity of $x\mapsto\bar{R}_{x}(y)$ yields that
\begin{align*}
\mathbf{E}_{\beta,x}\bigl(e^{-\delta A_{m+1}}\bigr)&\geq\sum_{y\in\mathbb{N}}\bar{R}_{x+1}(y)\Bigl[e^{-\delta y}\mathbf{E}_{\beta,y}\bigl(e^{-\delta A_m}\bigr)-e^{-\delta (y-1)}\mathbf{E}_{\beta,(y-1)}\bigl(e^{-\delta A_m}\bigr)\Bigr]+\mathbf{E}_{\beta}\bigl(e^{-\delta A_m}\bigr)\\
&=\mathbf{E}_{\beta,x+1}\bigl(e^{-\delta A_{m+1}}\bigr).
\end{align*}
\end{proof}

\bibliographystyle{amsplain}
\bibliography{cnp}

\end{document}